\documentclass[11pt]{article}
\usepackage{amsmath,amsthm,amsfonts,amssymb,mathrsfs,bm,wasysym}
\usepackage{epsfig}
\usepackage{color}
\usepackage{verbatim}
\usepackage{multicol}
\usepackage{comment}
\usepackage{graphicx}
\usepackage{centernot}
\usepackage{tikz}
\usepackage{enumerate}
\usepackage{thmtools}
\usepackage{float}
\usepackage{caption}
\usepackage{subcaption}
\usepackage{textcomp}
\usepackage[normalem]{ulem}
\usepackage[utf8]{inputenc} 
\usepackage[T1]{fontenc}
\usepackage[color=green, textsize=tiny]{todonotes}
\usepackage{hyperref}
\usepackage{titlesec}
\usepackage{chngcntr}
\usepackage{cleveref}
\usepackage{verbatim}

\parskip \smallskipamount
\declaretheoremstyle[
qed={//}
]{defstyle}

\newtheorem{theorem}{Theorem}[section]
\declaretheorem[style=defstyle,sibling=theorem]{definition}

\newtheorem{lemma}[theorem]{Lemma}
\newtheorem{proposition}[theorem]{Proposition}
\newtheorem{corollary}[theorem]{Corollary}

\newtheorem{remark}[theorem]{Remark}

\newtheorem*{claim*}{Claim}
\newtheorem*{lemma*}{Lemma}
\newtheorem{maintheorem}{Theorem}

\numberwithin{equation}{section}

\newcommand{\eps}{\varepsilon}
\newcommand{\Z}{\mathbb{Z}}
\newcommand{\Znonneg}{\Z_{\ge0}}
\newcommand{\Zpos}{\Z_{\ge1}}

\newcommand{\pgreen}{p^{\mathrm{green}}}
\newcommand{\pred}{p^{\mathrm{red}}}
\newcommand{\pblue}{p^{\mathrm{blue}}}

\newcommand{\Vnice}{V_{\mathrm{nice}}}
\newcommand{\Vniceprime}{V'_{\mathrm{nice}}}

\newcommand{\Vg}{V_{\mathrm{g}}}
\newcommand{\Gg}{G_{\mathrm{g}}}
\newcommand{\Ggr}{G_{\mathrm{rg}}}

\newcommand{\Gb}{G_{\mathrm{b}}}
\newcommand{\dg}{d_{\mathrm{g}}}
\newcommand{\dr}{d_{\mathrm{r}}}
\newcommand{\db}{d_{\mathrm{b}}}
\newcommand{\dgr}{d_{\mathrm{rg}}}
\newcommand{\dgrb}{d_{\mathrm{rgb}}}

\newcommand{\dmax}{d_{\mathrm{max}}}
\newcommand{\cb}{c_{\mathrm{b}}}
\newcommand{\Cvol}{C_{\mathrm{vol}}}
\newcommand{\Cdeg}{C_{\mathrm{deg}}}
\newcommand{\Cblue}{C_{\mathrm{blue}}}
\newcommand{\cnice}{c_{\mathrm{nice}}}

\newcommand{\simg}{\sim_{\mathrm{g}}}
\newcommand{\simr}{\sim_{\mathrm{r}}}
\newcommand{\simb}{\sim_{\mathrm{b}}}
\newcommand{\simrgb}{\sim_{\mathrm{rgb}}}
\newcommand{\Rnice}{R_{\mathrm{nice}}}
\newcommand{\dcool}{d_{\mathrm{cool}}}
\newcommand{\Ccool}{C_{\mathrm{cool}}}
\newcommand{\tcool}{t_{\mathrm{cool}}}
\newcommand{\puncool}{p_{\mathrm{uncool}}}
\newcommand{\ctrunc}{c_{\mathrm{trunc}}}
\newcommand{\pesc}{p_{\mathrm{esc}}}
\newcommand{\qesc}{q_{\mathrm{esc}}}

\newcommand{\const}{\mathrm{const}}

\newcommand{\Geom}[1]{\mathrm{Geom}\!\left(#1\right)}

\newcommand{\Bin}[2]{\mathrm{Bin}\!\left(#1,#2\right)}

\newcommand{\Unif}[1]{\mathrm{Unif}\!\left(#1\right)}
\newcommand{\Bern}[1]{\mathrm{Bern}\!\left(#1\right)}

\newcommand{\GGeompos}[1]{\mathrm{GGeom}_{\ge1}\left(#1\right)}
\newcommand{\GGeomnonneg}[1]{\mathrm{GGeom}_{\ge0}\left(#1\right)}

\newcommand{\pr}[1]{\mathbb{P}\!\left(#1\right)}
\newcommand{\E}[1]{\mathbb{E}\!\left[#1\right]}
\newcommand{\estart}[2]{\mathbb{E}_{#2}\!\left[#1\right]}
\newcommand{\prstart}[2]{\mathbb{P}_{#2}\!\left(#1\right)}
\newcommand{\prcond}[2]{\mathbb{P}\left(#1\;\middle\vert\;#2\right)}
\newcommand{\prscond}[3]{\mathbb{P}_{#3}\!\left(#1\;\middle\vert\;#2\right)}
\newcommand{\prstartup}[3]{\mathbb{P}_{#3}^{#2}\!\left(#1\right)}
\newcommand{\econd}[2]{\mathbb{E}\!\left[#1\;\middle\vert\;#2\right]}

\newcommand{\vr}[1]{\mathrm{Var}\left(#1\right)}

\newcommand{\1}[1]{{\text{\Large $\mathfrak 1$}}_{#1}}

\newcommand{\tmixtext}{t_{\mathrm{mix}}}
\newcommand{\tmix}[1]{\tmixtext\left(#1\right)}

\newcommand{\trel}{t_\mathrm{rel}}
\newcommand{\trelabs}{\trel^{\mathrm{abs}}}
\newcommand{\diam}[1]{\mathrm{diam}\left(#1\right)}

\newcommand{\tH}[1]{t_{\mathrm{H}}\left(#1\right)}
\newcommand{\hit}[2]{\mathrm{hit}_{#1}\left(#2\right)}

\newcommand{\Reffin}[3]{R_{\mathrm{eff}}^{#3}\left(#1,\:#2\right)}

\newcommand{\ent}[2]{H_{#1}\left(#2\right)}
\newcommand{\entc}[3]{H_{#1}\left(#2\mid#3\right)}
\newcommand{\entr}[3]{H_{#1}^{#2}\left(#3\right)}
\newcommand{\entrc}[4]{H_{#1}^{#2}\left(#3\mid#4\right)}

\newcommand{\dtv}[2]{d_{\mathrm{TV}}\left({#1},{#2}\right)}

\newcommand{\til}[1]{\widetilde{#1}}
\newcommand{\what}[1]{\widehat{#1}}

\newcommand{\tauatyp}{\tau_{\mathrm{atyp}}}

\newcommand{\tauniceprime}{\tau_{\mathrm{nice}'}}
\newcommand{\tauunif}{\tau_{\mathrm{unif}}}

\newcommand{\taublue}{\tau_{\mathrm{blue}}}

\newcommand{\taubound}{\tau_{\partial}}
\newcommand{\tauboundary}{\tau_{\mathrm{boundary}}}

\newcommand{\Omegatyp}[1]{\Omega_{\mathrm{typ}}\left(#1\right)}
\newcommand{\Omegabound}[1]{\Omega_{\partial}\left(#1\right)}
\newcommand{\Succ}{\Omega_{\mathrm{succ}}}
\newcommand{\Succhat}{\what{\Omega}_{\mathrm{succ}}}
\newcommand{\Omeganbhd}{\Omega_{\mathrm{nbhd}}}

\newcommand{\trunc}[2]{\mathrm{Tr}\left(#1,#2\right)}
\newcommand{\truncprime}[1]{\mathrm{Tr}'\left(#1\right)}

\newcommand{\h}{\mathfrak{h}}
\newcommand{\V}{\mathfrak{V}}

\newcommand{\energy}{\mathcal{E}}
\newcommand{\A}{\mathcal{A}}

\newcommand{\F}{\mathcal{F}}
\newcommand{\D}{\mathcal{D}}
\newcommand{\Scal}{\mathcal{S}}
\newcommand{\I}{\mathcal{I}}
\newcommand{\G}{\mathcal{G}}

\newcommand{\eqdist}{\stackrel{\mathrm{d}}{=}}
\newcommand{\stle}{\le_{\mathrm{st}}}
\newcommand{\stge}{\ge_{\mathrm{st}}}

\title{Random walk on the small-world network model in 3 or more dimensions}
\author{Zsuzsanna Baran \thanks{Paris Dauphine University. {baran@ceremade.dauphine.fr}}
\and Jonathan Hermon \thanks{University of British Columbia. {jhermon@math.ubc.ca}}
\and An{\dj}ela \v{S}arkovi\'c \thanks{King’s College and DPMMS, University of Cambridge. {as2572@cam.ac.uk}}
\and Allan Sly \thanks{Department of Mathematics, Princeton University. {allansly@princeton.edu}}
\and Perla Sousi \thanks{University of Cambridge. {p.sousi@statslab.cam.ac.uk}}}
\date{}

\begin{document}

\maketitle

\begin{abstract}
We study the mixing time of a simple random walk on the small-world network defined by adding edges to $\Z_n^d$ as follows: for each pair $\{x,y\}$ we add an edge with probability $Z_n/\|x-y\|^{d}$ with $Z_n$ chosen so that the average number of added edges to every vertex is $1$. When $d\geq 3$, we show that with high probability the mixing time is of order~$\log n$ and that the random walk does not exhibit cutoff.
\end{abstract}

\section{Introduction}

Random graph models trying to mimic the structure of naturally arising networks and capture their properties, in particular the ``small-world phenomenon'' have been studied for a long time. An early variant was the model of Watts and Strogatz~\cite{SWN_Watts-Strogatz} introduced in 1998, where they started with a one-dimensional torus $\Z_n$ with pairs of vertices within a given distance connected, and then they randomly rewired some of the edges. In 2000, Kleinberg~\cite{Kleinberg} considered a model where they started from a two-dimensional torus $\Z_n^2$ with pairs of vertices within a given distance also connected, and added some further random (directed) edges, with the connection probabilities decaying as a power of the distance on the torus.

In this paper, we study the mixing properties of a random walk on the `small-world network model' as introduced by Dyer et al.~\cite{RWs_on_SWNs} in 2020, defined as follows.

\begin{definition}\label{def:small-world_network}
Let $G_n=(V_n,E_n)$ denote $\Z_n^d$, a $d$-dimensional torus with side length $n$, 
and let $\til{G}_n=(V_n,E_n\sqcup \til{E}_n)$ be a random (multi)graph obtained as follows. For each pair $\{x,y\}$ of distinct vertices, let us include the edge $\{x,y\}$ in $\til{E}_n$ with probability $p_{n,x,y}=\frac{Z_n}{|x-y|^r}$, independently, where $|\cdot|$ denotes the graph distance in $G_n$, and $Z_n$ is chosen such that $\sum_{y}p_{n,x,y}=1$ for all $x$. We refer to $\til{G}_n$ as a \emph{small-world network graph}. In what follows, we almost always drop the subscript $n$ from the notation.
\end{definition}

The paper~\cite{RWs_on_SWNs} focused on $d=2$ and proved that for $r<2$ the mixing time is $\Theta(\log n)$, for $r>2$ it is $n^{\Omega(1)}$ and for $r=2$ it is $O((\log n)^4)$. In the $r=2$ case, it was conjectured that the mixing time is in fact  $\Theta\left((\log n)^2\right)$.

We see that for $r<d$ a constant fraction of the added edges connect two vertices that lie at a distance of order $n$, whereas for $r>d$ shorter edges are strongly preferred. For $r=d$ the distance of the two endpoints of an added edge is uniform over logarithmic scales (i.e.\ the probability of taking values in $[2^k,2^{k+1}]$ is of the same order for all relevant values of $k$).
The above results show that there is a phase transition between these two regimes at $r=d=2$. While there has been a lot of recent progress on understanding the mixing time of random walks on random graphs, in the above model, the geometric structure of the added edges poses challenges.

As pointed out in~\cite{RWs_on_SWNs}, the techniques of that paper could be used to prove that for $r>d$ the mixing time is $n^{\Omega(1)}$. We conjecture that when $r\in [0,d)$, there is cutoff around a logarithmic time for $d\geq 3$. We believe that the techniques developed in~\cite{random_matching, weighted_random_matching} would be useful in proving cutoff from a typical starting point. 

In this work, we focus on the critical choice $r=d$ of the exponent, and we consider the case $d\ge3$. Before stating our main result, we recall the definition of mixing time and absolute relaxation time.

\begin{definition}\label{def:tmix_cutoff}
Let $X$ be a Markov chain on a finite state space $V$, with a unique invariant distribution $\pi$. Then for $\eps\in(0,1)$ we define the $\eps$-mixing time of $X$ as
\[
\tmix{\eps}:=\quad\max_{x\in V}\:\inf\left\{t\ge0:\:\dtv{\prstart{X_t=\cdot}{x}}{\pi(\cdot)}\le\eps\right\},
\]
where $d_{\mathrm{TV}}$ denotes the total variation distance defined as $\dtv{\mu}{\nu}=\frac12\sum_{x}|\mu(x)-\nu(x)|$ for two probability distributions $\mu$, $\nu$ on the same state space.
\end{definition}

\begin{definition}
We say that a sequence $(A_n)$ of events holds with high probability (whp) if\\ $\pr{A_n}=1-o(1)$ as $n\to\infty$. In what follows, the dependence on $n$ will often be suppressed from the notation.
\end{definition}

\begin{definition}
Let $X$ be an irreducible Markov chain with transition matrix $P$. Its absolute relaxation time is defined as $\trelabs=\frac{1}{1-\lambda^*}$ where $\lambda^*=\max\{|\lambda|:\:\lambda\ne1\text{ is an eigenvalue of }P\}$.
\end{definition}

Now we are ready to state the main result of our paper.

\begin{maintheorem}\label{thm:tmix_trel_order_no_cutoff}
Let $d\ge3$ and consider the small-world graph $\til{G}_n$ as in \Cref{def:small-world_network}, with $r=d$. There exist positive constants $c'$ and $C'$, and for all $\theta\in(0,1)$ there exist positive constants $c(\theta)$ and $C(\theta)$, such that with high probability the following hold.

For all $\theta\in(0,1)$ the mixing time of a simple random walk on $\til{G}_n$ satisfies
\[c(\theta)\log n\quad\le\quad\tmix{\theta}\quad\le\quad C(\theta)\log n,\]
and the absolute relaxation time of a simple random walk on $\til{G}_n$ satisfies
\[c'\log n\quad\le\quad\trelabs\quad\le\quad C'\log n.\]
\end{maintheorem}

\begin{remark}
For the lazy random walk on $\til{G}_n$, i.e.\ the Markov chain with transition matrix $\frac12(I+P)$ where $P$ is the transition matrix of a simple random walk on $\til{G}_n$, the results of \Cref{thm:tmix_trel_order_no_cutoff} also hold. We see this by combining Lemmas~\ref{tmixlazy_upper_bound}, \ref{trelabs_lower_bound} and \ref{trel_asymp_trelabs}.
\end{remark}

\begin{remark}
In \Cref{thm:tmix_trel_order_no_cutoff} we establish that with high probability the mixing time and the absolute relaxation time of the walk on $\til{G}_n$ are of the same order, hence by \cite[Proposition 18.4]{MTMC} with high probability there is no cutoff. This means that there exist $c>0$ and $\eps>0$ such that with high probability $\frac{\tmix{\eps}}{\tmix{1-\eps}}\ge1+c$.
\end{remark}

\begin{remark}
We emphasise that the constants in \Cref{thm:tmix_trel_order_no_cutoff} depend on the dimension $d$. We believe that $\tmixtext\asymp d^2\log n$.
\end{remark}

\subsection{Relation to previous work}

As already mentioned earlier, the mixing properties of random walks on random graph models have been studied extensively, and often it was established that mixing happens at an entropic time, with cutoff, see for instance~\cite{RWs_on_random_graph, cutoff_NBRW_on_sparse_random_graphs, mixtures_of_permuted_MCs_reversible, BordCapSalez, SparseMarkovChains, BianchiPassuello, BianchiPassuelloQuattro}. 
In most of these models, there is no underlying geometric structure, and the random graphs are locally tree-like. Models where the local limit of the random graph is not tree-like have been studied in \cite{random_matching}, \cite{weighted_random_matching} and \cite{mixtures_of_permuted_MCs_reversible}. In \cite{random_matching} and \cite{weighted_random_matching}, the random edges are added according to a uniform perfect matching.

In the model of the present paper, we do not add edges uniformly at random, but the probabilities depend on the length of the edges. This distinguishes this model from the previous ones, as the added edges also have a geometric structure. This model also poses a number of additional challenges. A major source of difficulty is that the number of added edges emanating from a given vertex is neither bounded nor necessarily positive. The graph $\til{G}_n$ has, with high probability, regions with no added edges from them (`deserts') and also ones with many vertices with very large degrees (`hubs'). Because of this, one cannot directly apply the techniques developed in~\cite{RWs_on_random_graph} and~\cite{random_matching}. In the next section, we explain how we overcome these issues. 

The model of long-range percolation, where two vertices $x$ and~$y$ are connected by an edge with probability $1-\exp(-\beta\|x-y\|^{-s})$ for $s, \beta>0$, has a long history. We point out that $s=d$ and $\beta= \Theta(1/\log N)$ corresponds to the model we study in this paper. There is a large body of work focusing on the diameter in the finite setting~\cite{CopGamSviri, BiskupJeffrey}, the chemical distance~\cite{BenjaminiBerger, BiskupKrieger} and heat kernel bounds in the infinite setting~\cite{CrawfordSly}. 
We finally mention that~\cite{BenjaminiBergerYadin} focused on the one-dimensional case and established that with high probability the mixing time is at least $c(s,\beta) N^{s-1}$ and at most~$(\log N)^{C(s,\beta)} N^{s-1}$.
When $\beta\asymp 1$, from the perspective of chemical distance, the behaviour of long range percolation exhibits drastic change at $s=2d$. However, our normalisation, i.e.\ the choice of $Z_n$, leads to a different behaviour, where the point at which the mixing time changes is around~$s=d$.

 \subsection{Overview of the proof}
 
 In this section, we give an overview of the main ideas of the proof.
 
As it was often the case in previous works, our first step is to try to locally approximate our random graph $\til{G}_n$ with an auxiliary infinite graph, with a structure that makes the study of a random walk on it more amenable. Unlike many of the previously considered models (Erd\H{o}s-R\'enyi graph, configuration model, etc), the graph $\til{G}_n$ is clearly not locally tree-like (since it has the edges of the torus $\Z_n^d$), so the approximating structure cannot be a tree. This problem was already overcome in \cite{random_matching} by constructing a random tree-like structure $T$, called a `quasi-tree', where the added random edges play the role of the edges of the tree (`long-range edges') and smaller neighbourhoods of the underlying graph are viewed as vertices of the tree (`balls'). The tree-like structure of the quasi-tree $T$ makes the study of the long-term evolution of a walk on it more tractable, and one can prove concentration results about its speed and entropy. Then these results are transferred to the original graph by coupling its local neighbourhoods (chosen to cover the region where the walk is likely to visit by a given time) to agree with those of $T$.

Inspired by this, a natural approach would be to try to construct an analogous quasi-tree with the added random edges of $\til{G}_n$ serving as the long-range edges. The difficulty in our case is that in $\til{G}_n$ added edges are more likely to be short (joining vertices that are close in the graph distance of the underlying torus) and form cycles, hence, the neighbourhoods of $\til{G}_n$ cannot be coupled to agree with those of the quasi-tree.

We overcome this problem by treating the `short' and the `long' added edges in $\til{G}_n$ differently, and only using the long ones as the `long-range edges' of the quasi-tree. (The threshold between short and long edges is chosen such that, on average, a constant proportion of the edges belong in each category.) There is a part of our argument where it is helpful to consider a sparse set of random edges (of vanishing density) that are added uniformly at random, so we also reserve a small number of the added edges for this purpose. So overall we partition the added edges into three different categories: we `colour' a small number of them blue, and among the remaining ones, we colour the short ones green and the long ones red. The precise way of colouring is explained in \Cref{sec:colouring}.

We study a random walk on a random quasi-tree $T$, and we would like to establish a concentration of its speed and entropy by using the tree-like structure of $T$. Many of our techniques need assumptions that are satisfied for `typical' neighbourhoods in $\til{G}_n$, but as mentioned earlier, the graph $\til{G}_n$ can have neighbourhoods with more extreme properties, for example, very few added edges or very high degrees. Our estimates for the walk on $T$, including the proof of the essential property that it is transient, are only established in the case that it does not encounter this sort of atypical neighbourhood.

After understanding the behaviour of a walk on the quasi-tree (which is done in \Cref{sec:T}), we turn our focus back to the walk on $\til{G}_n$. Similarly to \cite{weighted_random_matching} we would like to bound its mixing time by constructing a random time when the distribution of the walk is close to stationary (and then using a comparison between mixing and hitting times from \cite{characterisation_of_cutoff}). To construct this random time, we first wait for the walk to reach a `typical' neighbourhood in $\til{G}_n$, then we wait for it to cross two blue edges (according to the colouring mentioned above), and finally we run it for $O(\log n)$ more steps so it can also reach vertices that do not have a blue edge emanating from them.

We use a method inspired by \cite{weighted_random_matching} (based on the ideas of \cite{cutoff-NBRW-on-sparse-random-graphs}) to show that starting from a typical neighbourhood of $\til{G}_n$, by the second time the walk crosses a blue edge, its distribution is close to uniform on the vertices that have a blue edge. To approximate the transition probability between two vertices $x$ and $y$ at this random time, we consider independent walks from $x$ and $y$ and estimate the probability that they meet by approximating the neighbourhoods of $x$ and $y$ with independent random quasi-trees and studying the walks on those. This argument is found in \Cref{sec:Gstar_and_T}.

The ideas presented above are used in a subsequent work~\cite{twisted_hypercube} establishing the mixing time of the random walk on a twisted hypercube, which is a graph $G^{(n)}$ recursively obtained by taking two independent copies of $G^{(n-1)}$ and adding an edge between each matched pair in an independent uniformly random perfect matching of their vertices, where $G^{(0)}$ is just a single vertex. In some sense, this model is a mean-field analogue of the small-world network analysed here, as it also satisfies that the number of added edges (which are all edges in this case) is uniform over logarithmic scales.  The twisted hypercube model is substantially less technical to analyse, as this graph is $n$-regular, and as one can pick for blue edges just the edges of the perfect matching that was added in the last iteration, and therefore this subsequent work might be helpful for understanding the main ideas behind some proofs presented here.

After this, the main remaining task is to prove that from any starting point, the walk on $\til{G}_n$ likely reaches a typical neighbourhood in order $\log n$ steps. This poses new challenges, as we need an argument that works for any starting neighbourhood that has a positive probability of appearing in $\til{G}_n$. The main step in the proof is showing that in a time of order $\log n$ the walk crosses at least order $\log\log n$ red edges, hence spreading out sufficiently that it is likely to land at a typical neighbourhood. Since we can only make assumptions that are likely to hold in all neighbourhoods of $\til{G}_n$, we are not able to approximate the neighbourhood of the starting point with a quasi-tree and assume that all balls that the walk enters have sufficiently nice properties. Instead, we only consider the immediate neighbourhood of the trajectory of the walk and decompose it into independent components called `boxes'. We show that a constant proportion of these boxes have a property that ensures that the walk is likely to quickly exit them via a red edge, and even for the rest of the boxes, the walk is unlikely to spend too much time in them. The details of this argument are presented in \Cref{sec:hit_nice}.

\subsection{Notation}

Throughout the paper, we assume that $r=d\ge3$ and consider $d$ fixed.

\subsubsection{Some general notation}

For vertices $x$ and $y$ we write $|x-y|$ for their graph distance on $G$. We write $N=n^d$ and $D=\diam{G}$. (Note that $D\asymp n$ and that $Z_n\asymp\frac{1}{\log N}$.)

We denote the set of nonnegative integers by $\Znonneg$ and the set of positive integers by $\Zpos$. For functions $a,b:\Znonneg\to[0,\infty)$ we write $a(n)\lesssim b(n)$ if there exists a constant $C>0$ such that for all sufficiently large $n$ we have $a(n)\le Cb(n)$, and write $a(n)\ll b(n)$ if for any constant $c>0$, for all sufficiently large $n$ (in terms of $c$) we have $a(n)\le cb(n)$. We define $\gtrsim$ and $\gg$ analogously. We write $a(n)\asymp b(n)$ if we have $a(n)\lesssim b(n)\lesssim a(n)$. Unless specified otherwise, these relations are considered as $n\to\infty$.

For random variables $A$ and $B$ taking values in an ordered set, we write $A\stle B$ or $B\stge A$ to mean that $B$ stochastically dominates $A$. For random variables $A$ and $B$ we write $A\eqdist B$ to mean that they have the same distribution. We also use this notation with a distribution in place of one or both of the random variables.

\subsubsection{Colouring of the edges, definition of an auxiliary graph $G^*$}\label{sec:colouring}

Our proof relies on treating the `short' and `long' added edges of $\til{G}$ differently, and also reserving a small proportion of the added edges for a third category.
In what follows, we introduce a colouring of the edges of $\til{G}$, and we also construct a slightly modified graph $G^*$ that can be coupled to $\til{G}$ to agree with it with high probability, but is more convenient to work with.

Let $\beta\in(0,1)$ be a constant strictly smaller than $\frac12$ (for example, we can take $\beta=0.1$). Let
\[
\pblue:=\quad\frac12\min_{x,y}p_{x,y}\quad\asymp\quad\frac{1}{N\log N}
\]
and for each $x$ and $y$ let
\begin{align*}
\pblue_{x,y}:=&\quad\pblue\\
\pgreen_{x,y}:=&\quad \left(p_{x,y}-\pblue_{x,y}\right)\1{|x-y|<D^{1-\beta}}\\
\pred_{x,y}:=&\quad \left(p_{x,y}-\pblue_{x,y}\right)\1{|x-y|\ge D^{1-\beta}}.
\end{align*}

Note that for each $x$ and $y$ we have $\pblue_{x,y}+\pgreen_{x,y}+\pred_{x,y}=p_{x,y}$, and $\pgreen_{x,y}+\pred_{x,y}\asymp p_{x,y}$.

We construct a copy of $\til{G}$ with its edges coloured as follows. We start from the torus $G$ and say that its edges are black. Then for each $x$ and $y$ in $G$, with probability $p_{x,y}$ we add an edge, which is blue with probability $\frac{\pblue_{x,y}}{p_{x,y}}$, green with probability $\frac{\pgreen_{x,y}}{p_{x,y}}$ and red with probability $\frac{\pred_{x,y}}{p_{x,y}}$. The presence and colour of edges are independent for different pairs $\{x,y\}$.

\begin{definition}\label{def:Gstar}
We construct a random graph $G^*$ as follows. We start from the torus $G$, and say that its edges are black. Then for each $x$ and $y$ in $G$ with probability $\pblue_{x,y}$ we add a blue edge between them. Then for each $x$ and $y$ with probability $\pgreen_{x,y}$ we add a green edge between them, and finally for each $x$ and $y$ with probability $\pred_{x,y}$ add a red edge between them. The presence or lack of these edges are all chosen independently from each other.
\end{definition}

Note that $G^*$ does not have the exact same distribution as $\til{G}$, since in $G^*$ it is possible that a pair of vertices is connected via both a blue and a red or green edge, while in $\til{G}$ this is not possible. The distribution of the two graphs is very similar, though.

\begin{lemma}\label{couple_Gtil_Gstar}
There exists a coupling of $G^*$ and $\til{G}$ such that the two graphs agree with probability $1-O\left(\frac{1}{\log N}\right)$.
\end{lemma}

\begin{proof}
For each pair $\{x,y\}$ the probability that the optimal coupling between the added edges between them in $\til{G}$ and $G^*$ fails is $2\pblue_{x,y}\left(\pgreen_{x,y}+\pred_{x,y}\right)\asymp\pblue\cdot p_{x,y}$.

Let us use the optimal coupling for each $\{x,y\}$, independently for different pairs. Then the probability that at least one of these optimal couplings fails is $\lesssim\pblue\sum_{x,y}p_{x,y}\asymp\frac{1}{N\log N}N\asymp\frac{1}{\log N}$.
\end{proof}

This means that it is sufficient to prove \Cref{thm:tmix_trel_order_no_cutoff} for $G^*$ instead of $\til{G}$. For the rest of the paper, we work with $G^*$ instead of $\til{G}$.

\subsubsection{Further notation and terminology regarding $G^*$}\label{sec:further_notation_Gstar}

{Let $\Gg$ be the graph consisting of the vertices and edges of $G$ and the green edges in $G^*$. Let $d_g$ be a random variable distributed as the green degree of a given vertex (i.e.\ the number of green edges emanating from it). For a vertex $v$ let $\dg(v)$ denote the green degree of $v$ in $G^*$. We also define analogous objects with the other colours and with combinations of colours (e.g.\ $\Ggr$ includes the edges of $G$ and both the green and the red edges of $G^*$). We write $x\simg y$ to mean that there is a green edge running between $x$ and $y$, and we use analogous notation for the other colours. We let $\deg(v)$ denote the degree of $v$ in $G^*$.}

In what follows, we will often sample the green, red or blue degree of a given vertex $v$ in $G^*$ without necessarily revealing the other endpoints of the edges from $v$. In this case, we say that we add a given number of green, red or blue \emph{half-edges} from $v$. Given a number of half-edges from a vertex, we may reveal the other endpoints of the edges corresponding to these half-edges, or given a number of half-edges from different vertices, we may sample how they should be matched up to form edges of $G^*$.

In what follows, unless specified otherwise, $X$ denotes a simple random walk on $G^*$.

\subsection{Why the case $d=2$ is harder}
We exploit the transience of $\mathbb{Z}^d$ for $d \ge 3$ in order to establish that the random walk on the quasi-tree which we consider in this work has a positive speed and that its entropy grows linearly in time (as well as that they are both concentrated).  Establishing the same for $d \in \{1,2\}$ requires some new ideas and a more careful analysis. This is the main obstacle for extending our analysis of the order of the mixing time from a ``typical'' starting point to the case that  $d \in \{1,2\}$.

Perhaps surprisingly, the more substantial challenge is determining for $d \in \{1,2\}$  the order of the expected hitting time of the set of ``typical'' starting points, starting from the worst initial state. 
Whp there are many box-shaped ``deserts'' $B$ in $\til{G}$ (i.e., there are no long-range edges between $B$ and $B^c$) of size $\Theta(\log N)$. The expected time it takes the random walk to exit such a desert is $\asymp (\log N)^{2/d}$. Showing that the expected exit time from one such desert is up to a constant factor also the (worst-case) expected time it takes the walk to reach a typical starting point appears to be a genuinely difficult problem.

\subsection{Lack of cutoff from a typical starting point}

As opposed to the vast majority of previous works on random walks on random graphs, the lazy random walk on $\til{G}$ does not exhibit cutoff even from typical starting points. We believe that there exists $c:(0,1) \to (0,\infty)$ satisfying $\lim_{\varepsilon \to 0} c(\varepsilon)=\infty$ such that whp  $\til{G}$ satisfies that for all $\varepsilon \in (0,1)$ the minimal (over all starting points) $\varepsilon$-total variation mixing time is at least $c(\varepsilon) d \log N$. Below we establish the following weaker claim, which rules out the occurrence of cutoff starting from ``typical'' starting points for the lazy walk: for the set $A:=\{(z_1,\ldots,z_d) \in \mathbb{Z}_n^d: -n/4 < z_1 \le n/4 \}$ there is no cutoff for the lazy random walk on $\til{G}$ starting from initial distribution $\pi_A$ or $\pi_{A^c}$, where $\pi_D$ denotes the stationary distribution $\pi$ conditioned on $D$.

 In order to shed some light on the mechanism behind the lack of cutoff from a typical starting point, we describe a certain transitive annealed version of the random walk on $\til{G}$. It also does not exhibit cutoff (from any starting point, by transitivity). We also refer the reader to the previously mentioned work \cite{twisted_hypercube} for further intuition, since the mechanism behind the lack of cutoff in \cite{twisted_hypercube} is very similar and is more transparent. 

 Consider the following annealed version of the model, with state space $\mathbb{Z}_n^d$ and transition matrix $P_{\mathrm{annealed}}=\frac{2d}{2d+1}P_1+\frac{1}{2d+1}P_2$, where $P_1$ is the transition matrix of the simple random walk on the torus and $P_2(x,y)=p_{n,x,y}$ is as defined in \Cref{def:small-world_network} with $r=d$. This is a random walk on the group $\mathbb{Z}_n^d$ with increment distribution $\mu(x):=P_{\mathrm{annealed}}(o,x)$, where $o=(0,0,\ldots,0)$. The mixing time of $P_{\mathrm{annealed}}$ is $O(d \log N)$. Indeed, since $\min_{x,y}P_{\mathrm{annealed}}(x,y) \gtrsim \frac{1}{dN \log N}$ there exists a Geometric random variable $\tau$ with $\mathbb{E}[\tau] \lesssim d \log N$, independent of the walk associated with $P_{\mathrm{annealed}}$, such that at time $\tau$ the walk is uniformly distributed. This $\tau$ is the annealed analog of our $\tau_{\mathrm{blue}}$, defined as the first time at which the walk crosses a blue edge. Of course, this analogy between $\tau$ and $\tau_{\mathrm{blue}}$  is incomplete. Let $t$ be the time required for the entropy to become $\frac{1}{2}\log N$ started from a typical starting point. 
 The best one can hope for is that the random walk on $G^*$, starting from a typical starting point, mixes after $t$ additional steps past the first time $s>t$ at which the walk crosses a blue edge. This is not a faithful description of what we actually prove.

In the proof of \Cref{trelabs_lower_bound} we show that whp the bottleneck ratio of the above set $A$ with respect to the random walk on $G^*$ (and hence also $\til{G}$), denoted by $\phi(A):=\mathbb{P}_{\pi_A}(X_1 \notin A)$, satisfies that  $\phi(A) \lesssim \frac{1}{d\log N}$. The same applies for $P_{\mathrm{annealed}}$ via a similar (simpler) calculation. Using Cheeger's inequality, this shows that the product condition fails for both walks (the one on $\til{G}$ and $P_{\mathrm{annealed}}$) and hence neither exhibits cutoff (see e.g. \cite[Theorem 13.14 and Prop.\ 18.4]{MTMC}). As we now explain, the estimate $\phi(A) \lesssim \frac{1}{d\log N}$ shows that there is no cutoff even when the initial distribution is $\pi_A$ or~$\pi_{A^c}$. 

In \cite{characterisation_of_cutoff} it is proved that for reversible chains cutoff is equivalent in a precise sense to concentration of hitting times of ``worst'' large sets. In~\cite{Htechnical} this characterisation is extended to cutoff from a fixed initial distribution $\mu$ in continuous time. The same argument extends directly to lazy discrete-time chains; alternatively, one may use the standard equivalence between cutoff for the continuous-time chain and for its lazy discrete-time counterpart. In particular, writing $t$ for the usual (worst-case) mixing time of the lazy walk, cutoff from initial distribution $\mu$ implies that for all $c>0$, $\max_{B:\pi(B) \ge 1/4}\mathbb{P}_{\mu}(T_B>(1+c)t)$ vanishes as the index of the chain diverges. This condition can easily be seen to fail for the lazy random walk on $\til{G}$ for the initial distributions $\pi_A$ and $\pi_{A^c}$, using the general relation\footnote{This follows from the fact that for reversible lazy chains, $\mathbb{P}_{\pi_D}(T_{D^c} = \cdot)$ is distributed like a mixture of geometric random variables, for every proper subset $D$ of the state space (see e.g., \cite{characterisation_of_cutoff}).} $\mathbb{P}_{\pi_D}(T_{D^c} > k) \ge (\mathbb{P}_{\pi_D}(T_{D^c}> 1))^k=(1-\phi(D))^k$ for all $k \ge 1$, which holds for all $D$ for reversible lazy (i.e., $\min_x P(x,x) \ge \frac 12$) Markov chains, with $D \in \{A,A^c\}$, together with the general relation $\pi(A)\phi(A)=\pi(A^c)\phi(A^c)$.

\subsection{A rare behaviour of the random walk on $\til{G}$}

Consider $A(k):=\{(z_1,\ldots,z_d) \in \mathbb{Z}_n^d: -n2^{-k-2} < z_1 \le n2^{-k-2}  \}$. A similar calculation as in the proof of \Cref{trelabs_lower_bound} shows that whp $\til{G}$ satisfies that uniformly for all $1 \le k \le \log_2(n)-O_d(1)$ we have that $\phi(A(k)) \asymp \frac{k}{d \log N}$. The same applies for $P_{\mathrm{annealed}}$. Observe that $\sum_{k=1}^{\log_2(n)-2} \frac{1}{\phi(A(k)) } \gtrsim d \log N \log \log N \asymp_d t_{\mathrm{mix}} \log \log N$.
This suggests that one cannot obtain an $O(\log N)$ upper bound on the mixing time using any type of analytic quantity which is robust under comparison of Dirichlet forms (and so, in particular, under small perturbations).

For simple random walk on a graph of bounded average degree with a stationary distribution $\pi$ one should generically expect the order of the total variation mixing time from initial state $o$ to be bounded below by $\max \sum_{k=1}^{\log_2(1/\pi(o))-1} \frac{1}{\phi(B(k)) }$, where the maximum is over all nested $\{o\} \subset B(\log_2(1/\pi_*)-1) \subset \cdots \subset B(1)$ such that $\pi(B_k) \in (2^{-k-1},2^{-k}]$ for all $k$. We are aware only of a handful of other counter-examples with a bounded average degree. These are all pathological, in the sense they are obtained by making some non-trivial modifications to chains which were constructed to demonstrate either i) that small perturbations can change the order of the mixing time or ii) that the spectral profile determines the $\ell^{\infty}$-mixing time of a reversible Markov chain only up to a $\log \log(1/\pi_*)$ multiplicative factor, where $\pi_*:=\min_x \pi(x)$.

\subsection{The structure of the proof}\label{sec:overview}

As mentioned above, we work with $G^*$ instead of $\til{G}$ throughout the paper.

In order to bound the mixing time of a random walk on $G^*$, instead of trying to directly estimate its transition probabilities at a specific time, we use an equivalence between mixing and hitting times. The notion of hitting time we use is defined below. 

\begin{definition}\label{def:hit_time}
Let $X$ be a Markov chain on a finite state space $V$ with a unique invariant distribution $\pi$. For each $\alpha,\theta\in(0,1)$ we define the corresponding \emph{hit time} of $X$ as
\[
\hit{\alpha}{\theta}:=\quad\inf\{t:\:\forall x,\:\forall A\subseteq V\text{ with }\pi(A)\ge\alpha\text{ we have }\prstart{\tau_A>t}{x}\le\theta\},
\]
where $\tau_A=\inf\left\{t:\:X_t\in A\right\}$.
\end{definition}

The proof of \Cref{thm:tmix_trel_order_no_cutoff} relies on the following four results.

\begin{lemma}\label{hit_upper_bound}
For any $\theta,\alpha\in(0,1)$ there exists a positive constant $C$ such that whp the graph $G^*$ has the following property. The hit time of a simple random walk on $G^*$ with parameters $\alpha$ and $\theta$ satisfies $\hit{\alpha}{\theta}\le C\log N$.
\end{lemma}

\begin{lemma}\label{tmixlazy_upper_bound}
There exists a positive constant $C$ such that whp $G^*$ satisfies the following. The mixing time of a lazy random walk on $G^*$ can be upper bounded as $\tmixtext^{\mathrm{lazy}}\le C\log N$.
\end{lemma}

\begin{lemma}\label{trelabs_lower_bound}
There exists a positive constant $C$ such that whp the Cheeger constant of $G^*$ satisfies $\Phi^*\le\frac{C}{\log N}$. From this, it follows that there is a positive constant $c$ such that whp the absolute relaxation time of a simple random walk on $G^*$ satisfies $\trelabs\ge c\log N$.
\end{lemma}

\begin{lemma}\label{trel_asymp_trelabs}
There exist positive constants $c$ and $C$ such that whp $G^*$ has the following property. The relaxation time and the absolute relaxation time of a simple random walk on $G^*$ satisfy $c\trel\le\trelabs\le C\trel$.
\end{lemma}

Once we have these, the proof of \Cref{thm:tmix_trel_order_no_cutoff} can be concluded quickly as follows.

\begin{proof}[Proof of \Cref{thm:tmix_trel_order_no_cutoff}] First let us prove the results for the modified graph $G^*$.

From \Cref{tmixlazy_upper_bound} we know that $\trel\lesssim\tmixtext^{\mathrm{lazy}}\lesssim\log N$ whp, from \Cref{trelabs_lower_bound} we know that $\trelabs\gtrsim\log N$ whp and from \Cref{trel_asymp_trelabs} we know that $\trel\asymp\trelabs$ whp. Together, these give that whp $\trel\asymp\trelabs\asymp\log N$.

Using this, \Cref{hit_upper_bound} and~\cite[Proposition 1.8 and Remark 1.9]{characterisation_of_cutoff} we get that whp $\tmix{\theta}\asymp\tmixtext^{\mathrm{lazy}}\left(\theta\right)\asymp\log N$ for all $\theta\in(0,1)$.

Since $\tmixtext\asymp\trelabs$ and $\tmixtext^{\mathrm{lazy}}\asymp\trel$, neither a simple nor a lazy random walk exhibits cutoff.

From \Cref{couple_Gtil_Gstar} we know that $G^*$ and $\til{G}$ can be coupled to agree whp, hence the results also carry over to $\til{G}$.
\end{proof}

The most difficult part of proving the above lemmas is upper bounding the hit time, which then implies an upper bound on the lazy mixing time on $G^*$. To achieve this, we construct a random time $\tau$ such that it is of order at most $\log N$, and for any starting point, the distribution of $X_{\tau}$ is comparable to the invariant distribution from below. The exact statement we prove is as follows.

\begin{proposition}\label{tauhat_transition_probs}
There exist positive constants $\theta$ and $C$ such that whp $G^*$ satisfies the following. There exists a random time $\tau$ such that $\tau\le C\log N$, and for any $x,y\in V$, we have $\prscond{X_{\tau}=y}{G^*}{x}\ge\theta\frac{\deg(y)}{N}$.
\end{proposition}

The random time $\tau$ we use is of the form $\tauniceprime+\taublue^{(2)}+\tauunif$, where $\tauniceprime$ is the time needed for the walk to reach a `sufficiently nice neighbourhood' of $G^*$, $\taublue^{(2)}$ is the time needed to cross two blue edges after this {(including the case where the same blue edge is crossed twice)}, and $\tauunif$ is a further uniform random time with mean of order $\log N$ to ensure that $X_{\tau}$ is spread out on the set of all vertices, not only on the $\asymp\frac{1}{\log n}$ proportion of them that have a positive blue degree. 

A large proportion of the work goes into proving that (whp $G^*$ is such that) for any starting point that has a sufficiently nice neighbourhood in $G^*$, the walk at time $\taublue^{(2)}-1$ is spread out on the set of vertices that have a positive blue degree and a sufficiently nice neighbourhood in $G^*$. (See \Cref{transition_probs_on_Vnice} for the exact statement, and the preceding parts for the definitions it builds on.)
The proof of this relies on writing
\begin{align}\label{eq:taubluetwo_approx_sum}
\prscond{X_{\taublue^{(2)}-1}=y}{G^*}{x}\quad=\quad\db(y)\sum_{e,f}\prscond{X_{\taublue^X}=e}{G^*}{x}\prscond{Y_{\taublue^Y}=f}{G^*}{y}\1{e\simb f}\:,
\end{align}
where $X$ and $Y$ are independent random walks on $G^*$ from $x$ and $y$ respectively, $\taublue^X$ and $\taublue^Y$ are the times when they first cross a blue (half-)edge, the sum runs over pairs of blue half-edges, and $e\simb f$ means that $e$ and $f$ are the two halves of the same blue edge. Instead of trying to study the behaviour of a walk on $G^*$ directly, we define a random tree-like structure $T$, called a random quasi-tree, that locally approximates $G^*$ well and we study a walk $\til{X}$ on that instead. We explore a random neighbourhood of $x$ and $y$ and couple them to two independent quasi-trees to get a good estimate on the probabilities $\prscond{X_{\taublue^X}=e}{G^*}{x}$ and $\prscond{Y_{\taublue^Y}=f}{G^*}{y}$ for each $e$ and $f$, and we use the randomness of the matching between the blue half-edges to show that the sum \eqref{eq:taubluetwo_approx_sum} is concentrated around a value of order $\db(y)\frac{\log N}{N}$. \footnote{More precisely, we consider a lower bound on the sum where we include various desirable events in the probabilities and we prove that this sum is concentrated.}

\subsection{Further notational conventions}

In some estimates we write `$\const$' to denote a constant, and its value might change from expression to expression. (E.g.\ $\exp\left(-\const\log n+\log\log n\right)\lesssim\exp\left(-\const\log n\right)$.) We often give definitions that have a constant parameter that is assumed to be sufficiently small or sufficiently large. In these cases we mention this assumption in the definition, but we do not necessarily repeat it in the statements of later results. We usually denote constants that take large values by upper case letters and constants that take small values by lower case ones.

In what follows, we often consider quantities like $0.9n$ or $C\log n$ and treat them as integers. In these cases one can take $\lfloor\cdot\rfloor$ or $\lceil\cdot\rceil$, but we omit this from the notation.

When clear from context, we also often omit the subscript or superscript $n$ from the notation.

\subsection{Organisation}

In \Cref{sec:prelim_estimates} we prove some preliminary estimates on the degrees and the growth of balls in $G^*$. In \Cref{sec:T} we define a random quasi-tree $T$ and establish some important properties of a random walk on it, including a form of transience and a concentration of its speed and entropy. In \Cref{sec:Gstar_and_T} we explain the coupling between two neighbourhoods of $G^*$ and two quasi-trees, estimate the probability of its success, and bound the time $\taublue^{(2)}$ transition probabilities between two vertices with nice neighbourhoods. In \Cref{sec:hit_nice} we prove that from any starting point, a walk on $G^*$ hits a nice neighbourhood quickly. In \Cref{sec:conclusion} we conclude the proof of Lemmas~\ref{hit_upper_bound}, \ref{tmixlazy_upper_bound}, \ref{trelabs_lower_bound} and~\ref{trel_asymp_trelabs}.

\section{Some preliminary estimates}\label{sec:prelim_estimates}

In this section, we establish some estimates that will be used throughout the paper.

\subsection{Bounds on the degrees}

\begin{lemma}\label{max_deg}
There exists a constant $C$ such that with probability $1-o\left(\frac1N\right)$ all vertices $v$ in $G^*$ satisfy $\deg(v)\le C\frac{\log N}{\log\log N}=:\dmax$.
\end{lemma}

\begin{proof}
For each $v$ we upper bound $\pr{\deg(v)\ge C\frac{\log N}{\log\log N}}$ using \Cref{sum_of_Berns} and get that for sufficiently large $C$ this is $\ll\frac{1}{N^2}$. Then we take a union bound over $v$.
\end{proof}

\begin{lemma}\label{total_degree}
Whp the sum of all green, red and blue degrees in $G^*$ satisfies $N-N^{\frac23}\le\sum_{v}\dgrb(v)\le N+N^{\frac23}$.
\end{lemma}

\begin{proof}
Note that $\frac12\sum_{v}\dgrb(v)=\sum_{\{x,y\}:x\ne y}\left(\1{x\simg y}+\1{x\simr y}+\1{x\simb y}\right)$ is a sum of independent Bernoulli random variables and has mean $\frac12N$. Then use for example, \Cref{sum_indep_RVs}.
\end{proof}

\begin{lemma}\label{total_blue_degree}
Whp the sum of all blue degrees in $G^*$ satisfies $\cb\frac{N}{\log N}-\sqrt{N}\le\sum_{v}\db(v)\le \cb\frac{N}{\log N}+\sqrt{N}$ where $\cb\asymp1$ is such that $\pblue=\frac{\cb}{N\log N}$. Also, $\pr{\sum_{v}\db(v)\ge2\cb\frac{N}{\log N}}\ll\frac{1}{N^2}$.
\end{lemma}

\begin{proof}
Note that $\frac12\sum_{v}\db(v)=\sum_{\{x,y\}:x\ne y}\1{x\simb y}$ is a sum of independent $\Bern{\pblue}$ random variables, and use e.g.\ \Cref{sum_indep_RVs}.
\end{proof}

\subsection{Bounds on the growth of balls}

In general, we can bound the volume of a ball in $G^*$ as follows.

\begin{lemma}\label{volume_tail_bound}
For any radius $r$, any $s>0$ and any given vertex $x$ of $G^*$, the volume of a ball $B_r$ of radius $r$ centred at $x$ can be bounded as
\[\pr{|B_r|\ge(2d+3)^rs}\quad\le\quad e^{\frac14}e^{-\frac14s}.\]
\end{lemma}

The proof of \Cref{volume_tail_bound} is deferred to \Cref{app:volume_tail_bound}. Once we have this, a union bound quickly implies the following statement, bounding the volumes of all sufficiently large balls in $G^*$ simultaneously.

\begin{corollary}\label{volume_high_prob_bound}
For any constant $C'>0$ there exists a constant $C$ such that for any $r\ge C'\log\log N$ whp every ball of radius $r$ in $G^*$ has volume $\le e^{Cr}$.
\end{corollary}

\section{Limiting tree}\label{sec:T}

In this section we start by defining a random tree-like structure called a quasi-tree, which will be used to locally approximate the graph $G^*$. Then we study a random walk on it, we establish a form of transience, and we prove a concentration result about its speed and entropy.

In the next section, we will continue working with this quasi-tree and explain how to couple it closely with~$G^*$.

In \Cref{sec:hit_nice} we will work towards proving that for any starting point the walk on $G^*$ hits a nice neighbourhood quickly. For this, we will consider a quasi-tree which is different from the one considered here.

\subsection{Values of some parameters}\label{sec:parameters}

For ease of reference, we list all major parameters used in the section below. Their roles will become clear in due course.

We let
\begin{align*}
&R:=C_R\log\log N,\quad K:=C_K\log\log N,\quad M:=C_M\log\log N,\\
&L^-:=c_L\log N+C_L^-\sqrt{(\log N)(\log\log N)},\quad L^+:=c_L(1+c^+_L)\log N,
\end{align*}
where $c_L$ is a given constant, to be specified after \Cref{entropy_concentration}, while $C_R$, $C_K$, $C_M$ and $C_L^-$ are sufficiently large constants and $c_L^+$ is a sufficiently small constant, to be specified later.

\subsection{Definition of a quasi-tree}

We define a random quasi-tree as follows.

\begin{definition}\label{def:random_quasi_tree}
A \emph{random quasi-tree} is a random graph $T$ together with a map $\iota$ from its vertices to the vertex set of $G$, obtained as follows.

We start from a root vertex $\rho$ with some associated $\iota(\rho)$. Then we consider a ball of radius $R$ around $\iota(\rho)$ in a random copy of $\Gg$, and for each vertex in this ball we also add its neighbourhood of radius $R$ in $G$. We call the resulting neighbourhood of $\rho$ an \emph{$(R+R)$-ball}, and for each vertex $v$ in it we let $\iota(v)$ be the corresponding vertex in $G$.

{Then for each vertex $v$ in the $(R+R)$-ball, we sample $\dr(v)\eqdist\dr$ independently, and add this many new edges from it, called \emph{long-range edges}.}
For each $w$ that is the new endpoint of a newly added long-range edge from a vertex $v$, we sample $\iota(w)$ choosing a vertex $x$ with probability proportional to $\pred_{\iota(v),x}$, and attach a new $(R+R)$-ball around it, sampling from a new copy of $\Gg$. {Then we continue analogously from each new $(R+R)$-ball.}

The copies of $\Gg$ we use for the different $(R+R)$-balls, the long-range degrees $\dr(v)$ of the vertices, and the choices of $\iota$ for the new endpoints of the long-range edges are all independent.

For simplicity, we sometimes refer to an $(R+R)$-ball as a \emph{ball}.

We call the set of vertices in an $(R+R)$-ball that are of graph distance $\ge R$ from its centre the \emph{boundary} of the ball. We use the graph distance of the quasi-tree $T$. (Note that this is different from the usual definition of a boundary.)
\end{definition}

Sometimes it is helpful to think about the quasi-tree induced by one specific realisation of the graph~$G^*$, so we also make the following definition. (In this section and the next section, quasi-tree normally refers to a random quasi-tree as in \Cref{def:random_quasi_tree}. We occasionally write quasi-tree to refer to some other $T$ and $\iota$ that is a possible realisation of a random quasi-tree.)

\begin{definition}\label{def:quasi-tree_of_Gstar}
{Given the graph $G^*$ and a vertex $x$ of $G^*$, the \emph{quasi-tree from $x$ associated to $G^*$} is a graph $T_x(G^*)$ together with a map $\iota$ from its vertices to the vertex set of $G$, obtained as follows.}

{We start from a root vertex $\rho$ with $\iota(\rho)=x$.} Then we consider a ball of radius $R$ around $\iota(\rho)$ in $\Gg$ and for each vertex in this ball we also add its neighbourhood of radius $R$ in $G$. We call the resulting neighbourhood of $\rho$ an $(R+R)$-ball, and for each vertex $v$ in it we let $\iota(v)$ be the corresponding vertex in $G$.

Then for each vertex $v$ in the $(R+R)$-ball, including the centre, we add $\dr(\iota(v))$ new edges from it, called long-range edges, and we let $\iota$ map the new endpoints of those edges to the red neighbours of $\iota(v)$. For each new endpoint $w$ we attach an $(R+R)$-ball around it according to the neighbourhood of $\iota(w)$ in $\Gg$, and for each vertex $u$ in this $(R+R)$-ball we add a number of long-range edges from $u$ so that the total number of long-range edges from $u$ is $\dr(\iota(u))$. We continue similarly with the new endpoints of these long-range edges.

We call the set of vertices in an $(R+R)$-ball that are of graph distance $\ge R$ from its centre the boundary of the ball.
\end{definition}

\begin{figure}[h]
\centering
\includegraphics[width=60mm]{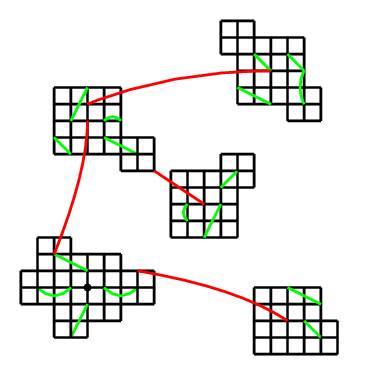}
\caption{Illustration of the first few balls of the quasi-tree associated to $G^*$}
\end{figure}

A few further pieces of notation regarding quasi-trees are listed below.

\begin{definition}\label{def:quasi-tree_notation}
For vertices $x$ and $y$ in a quasi-tree $T$, their \emph{long-range distance} $d_T(x,y)$ is the minimal number of long-range edges a path from $x$ to $y$ has to cross. The \emph{level} of a vertex $x$ is its long-range distance from the root, and it is denoted by $\ell(x)$. Note that all vertices within a ball have the same level, we also refer to this as the level of the ball. For a long-range edge $e=\{x,y\}$ we define its level $\ell(e)$ as $\ell(x)\vee\ell(y)$, and we denote its two endpoints by $e^+$ and $e^-$ such that $\ell(e^+)=\ell(e)$. For a vertex $x$ in $T$, we write $T(x)$ for the quasi-tree obtained by removing the long-range edge that leads to the ball of $x$ from a smaller level (if there is one) and keeping the connected component of~$x$. We say that a long-range edge $e$ is a \emph{descendant} of a long-range edge $f$ in $T$ if every path between $e^+$ and the root of $T$ crosses $f$. We say that $e$ is an \emph{ancestor} of $f$ if $f$ is a descendant of $e$. We use analogous terminology for balls in $T$.
\end{definition}

We note that since $\dr$ and $\dg$ are neither bounded nor necessarily positive, there could be some exceptional balls in $T$ featuring some very large degrees or very large volume, or having large regions with no long-range edges emanating from them. If a walk on $T$ enters one of these exceptional balls, it is difficult to make good predictions about its behaviour, so in most of what follows we make statements about a walk until it visits such a ball (and we use a bootstrapping argument to show that such a visit is unlikely to happen by the time of interest). We introduce a thorough notion of `typicality' in \Cref{def:atypical_ball}, but first we establish some results with a weaker notion defined as follows.

\begin{definition}\label{def:binary_RplusR_ball}
We say that an $(R+R)$-ball is \emph{binary} if it has the following property. For each of its vertices $v$, there is a copy of $\left[0,\frac1dR\right]^d$ in the $(R+R)$-ball that contains $v$ and has at least 3 long-range edges emanating from it (including any edges emanating from $v$). Here $\left[0,\frac1dR\right]^d$ denotes a $d$-dimensional box of side length $\frac1dR$, with edges induced by $\Z^d$.
\end{definition}

\subsection{Transience}

In this section we establish a transience-like property for a random walk on a quasi-tree stopped upon hitting a non-binary ball, and also for a walk stopped upon hitting a boundary vertex.

We start by making a statement about a walk on a quasi-tree where all balls are binary.

\begin{lemma}\label{transience_in_binary_tree}
There exists a constant $\pesc\in(0,1)$ ({not depending on $C_R$ or any parameter other than possibly $d$}) with the following property. Let $T$ be a realisation of a random quasi-tree where each $(R+R)$-ball is binary in the sense of \Cref{def:binary_RplusR_ball}, and let $v$ be a vertex in $T$. Then starting a simple random walk on $T$ from $v$, with probability at least $\frac{\pesc}{\deg_T(v)}$ the walk never returns to $v$.

Moreover, if $(v,w)$ is a long-range edge of $T$, then starting a simple random walk on $T$ from $v$, with probability at least $\frac{\pesc}{\deg_T(v)}$ the walk never returns to $v$ or hits $w$ {and with probability at least $\pesc$ the walk never hits $w$}.
\end{lemma}

By a natural coupling, this immediately implies the following statement about a walk that is stopped at the first time it hits a non-binary ball.

\begin{corollary}\label{transience_until_nonbinary}
Let $T$ be any realisation of a random quasi-tree and let $v$ be a vertex of $T$. Let $\pesc\in(0,1)$ be as in Lemma~\ref{transience_in_binary_tree}. Then
\begin{align*}
{\prscond{\tau_v^+>\tau_{\mathrm{non-bin}}}{T}{v}\quad\ge\quad\frac{\pesc}{\deg_T(v)}},
\end{align*}
where $\tau_{\mathrm{non-bin}}$ denotes the first time that $X$ hits a non-binary $(R+R)$-ball {and $\tau_v^+=\inf\{t>0: X_t=v\}$ is the first time the walk returns to $v$}.

Moreover, if $(v,w)$ is a long-range edge of $T$, then
\begin{align*}
{\prscond{\tau_v^+\wedge\tau_w>\tau_{\mathrm{non-bin}}}{T}{v}\quad\ge\quad\frac{\pesc}{\deg_T(v)} \quad \text{and}\quad \prscond{\tau_w>\tau_{\mathrm{non-bin}}}{T}{v}\quad\ge\quad\pesc}.
\end{align*}
\end{corollary}

\begin{proof}[Proof of \Cref{transience_in_binary_tree}]
It is sufficient to prove the second part of the claim, as it also implies the first one.

From~\cite[Theorem 2.11]{prob-on-trees-and-networks} we know that a SRW on a graph is transient if and only if there is a unit flow of finite energy from some vertex of the graph to $\infty$.

Let $T_{-w}$ be obtained from $T$ by removing the edge $(v,w)$ and only keeping the connected component of $v$. We will construct a unit flow {$\theta$ from $v$ to $\infty$ on $T_{-w}$ of energy bounded by constant}. This shows that a SRW on $T_{-w}$, starting from $v$, has at least {$\frac{\rm{const}}{\deg_T(v)-1}$ probability of never returning to $v$. Then the probability that a SRW on $T$ from $v$ does not pass through $(v,w)$ in the first step and never returns to $v$ is also at least $\frac{\rm{const}}{\deg_T(v)-1}$. Further, on every visit to $v$ the walk has probability $\frac{1}{\deg_T(v)}$ to pass through $(v,w)$, giving the constant lower bound on the probability to never pass through $w$. We set $\pesc$ to be the smallest of these constants.}

In $T_{-w}$ let $\left(e^0_1\right)^+=v$, let $e^{1}_1$ and $e^{1}_2$ be two long-range edges from a box of side length $\frac1dR$ containing $v$, and for $k=1,2,...$ and $j\in\{1,2,...,2^k\}$ let $e^{k+1}_{2j-1}$ and $e^{k+1}_{2j}$ be two long-range edges (other than $e^{k}_{j}$) from a box of side length $\frac1dR$ containing $\left(e^{k}_{j}\right)^+$. From \Cref{unit_flow_in_box} we know that there exist unit flows $\theta^{k+1}_{2j-1}$ and $\theta^{k+1}_{2j}$ in the $(R+R)$-ball of $\left(e^{k}_{j}\right)^+$, from $\left(e^{k}_{j}\right)^+$ to $\left(e^{k+1}_{2j-1}\right)^-$ and $\left(e^{k+1}_{2j}\right)^-$, respectively with energies bounded by a constant $C$.

Let $\theta$ be a unit flow from $v$ to $\infty$ constructed by considering $\sum_{k\ge0}\sum_{j=1}^{2^k}2^{-k}\theta^k_j$ and also adding a flow of strength $2^{-k}$ along each $e^k_j$. Then $\theta$ has energy $\le\sum_{k\ge0}2^k\cdot \left(2^{-2k}\cdot4C\right)+\sum_{k\ge0}2^k\cdot2^{-2k}\lesssim 1$.
\end{proof}

\begin{figure}[h]
\centering
\includegraphics[width=60mm]{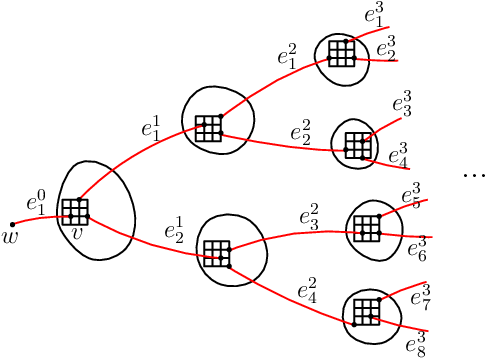}
\caption{Illustration of the edges $e^i_j$ in the proof of \Cref{transience_in_binary_tree}}
\end{figure}

We write $\tauboundary$ for the first time that a walk on a quasi-tree hits the boundary of an $(R+R)$-ball. (Later we will use $\taubound$ to denote a different but related random time.) We prove the following transience-like property of a walk stopped at $\tauboundary$.

\begin{lemma}\label{transience_until_boundary}
There exists a positive constant $c$ with the following property.

For any realisation $T$ of a random quasi-tree, and any vertex $x$ of $T$, we have
\[
\prstart{\tau_x^+>\tauboundary}{x}\quad\ge\quad\frac{c}{\deg_T(x)},
\]
while for any long-range edge $e$, we have
\[
\prstart{\tau_e>\tauboundary}{x}\quad\ge\quad c.
\]
\end{lemma}

\begin{proof} Firstly, we note that for any realisation of $T$ we have $\prstart{\tauboundary=\infty}{x}=0$, so in the following calculations we can use events like $\{\tau_x^+>\tauboundary\}$ and $\{\tau_x^+\ge\tauboundary\}$ interchangeably.

Once we establish the first inequality, the second one follows by noting that it is sufficient to consider the case when $x$ is one of the endpoints of $e$, using that in this case \[\prstart{\tau_e<\tauboundary}{x}\quad=\quad\frac{1}{\deg_T(x)}\sum_{k\ge0}\left(1-\frac{1}{\deg_T(x)}\right)^k\prstartup{\tau_x^+<\tauboundary}{T\setminus\{e\}}{x}^k,\] and using the first inequality for $T\setminus\{e\}$. We present the proof of the first inequality below.

Let $X$ be a simple random walk on $T$ started from $x$ and let $Y$ be a simple random walk on the $(R+R)$-ball of $x$, started from $x$, coupled to $X$ such that it agrees with the restriction of $X$ to the $(R+R)$-ball of $x$ as long as the latter one is defined. \footnote{By the restriction of $X$ to the $(R+R)$-ball we mean a (possibly finite) walk obtained by removing the steps of $X$ taken outside this ball, and then from the remaining walk removing the steps where it stays in place. Note that after $X$ exits the ball, it can only return to it via the same edge, hence the restriction of $X$ to the ball is a simple random walk on the ball, with a random killing time.} Note that if $X$ does not cross a long-range edge in the first step, and $Y$ does not return to $x$ without hitting a boundary first, then $X$ does not return to $x$ without hitting a boundary first either. Hence we can write
\begin{align*}
\prstart{\tau_x^+>\tauboundary}{x}\:\ge&\:\frac{\deg_T(x)-\dr(x)}{\deg_T(x)}\prscond{\tau_x^+>\tauboundary}{\text{first step within the }(R+R)\text{-ball of }x}{x}\\
\ge&\:\frac{\deg_T(x)-\dr(x)}{\deg_T(x)}\prstartup{\tau_x^+>\tauboundary}{Y}{x},
\end{align*}
where $\dr(x)$ denotes the number of long-range edges from $x$ in $T$. So it is sufficient to show that $\prstartup{\tau_x^+>\tauboundary}{Y}{x}\gtrsim\frac{1}{\deg_T(x)-\dr(x)}$.

Let us consider a realisation of $\Gg$ such that the $(R+R)$-ball of $x$ in $T$ is isomorphic to the corresponding neighbourhood of $x$ in $\Gg$ and let $Z$ be a simple random walk on $\Gg$, started from $x$. Note that $Z$ can be coupled to $Y$ so that they agree until $\tauboundary$. Let $B$ be a ball of radius $\frac14n$ around $x$ in $\Gg$, according to the torus graph distance. Since $R\cdot D^{1-\beta}+R<\frac14n$, we have
\begin{align*}
\prstartup{\tau_x^+>\tauboundary}{Y}{x}\quad=\quad\prstartup{\tau_x^+>\tauboundary}{Z}{x}\quad\ge\quad\prstartup{\tau_x^+>\tau_{B^c}}{Z}{x},
\end{align*}
where $\tau_{B^c}$ denotes the first time that $Z$ exits the ball $B$.

Viewing $\Gg$ as an electric network, we see that $\prstartup{\tau_x^+>\tau_{B^c}}{Z}{x}=\frac{1}{\deg_{\Gg}(x)\Reffin{x}{B^c}{\Gg}}$, so it is sufficient to show that $\Reffin{x}{B^c}{\Gg}\lesssim1$. By Rayleigh's monotonicity principle, it is sufficient to show that $\Reffin{x}{B^c}{G}\lesssim1$, which follows by noting that the restriction of the torus $G$ to the ball $B$ is isomorphic to a ball in $\Z^d$ and using the transience of $\Z^d$ for $d\geq 3$.
\end{proof}

\subsection{Atypical behaviour}

We define a notion of typicality for $(R+R)$-balls as follows.

\begin{definition}\label{def:atypical_ball}
We say that an $(R+R)$-ball in a quasi-tree $T$ is \emph{atypical} if any of the following hold.
\begin{enumerate}[(i)]
\item\label{item:nonbinary} It is not binary in the sense of \Cref{def:binary_RplusR_ball}.
\item\label{item:large_deg} It contains a vertex with degree larger than $\dmax$. (Here $\dmax$ is defined as in \Cref{max_deg}).
\item\label{item:hit_boundary}
{It satisfies that $p_\partial >\left(\frac{1}{\log N}\right)^5$ where $p_\partial$ is defined as follows. Let $Y$ be a walk starting from the centre of the $(R+R)$-ball, which evolves like a simple random walk except for the following. Every time it crosses a long-range edge from the ball, with probability $\pesc$ it gets terminated, otherwise, it crosses the long-range edge back and continues. Let $p_\partial$ be the probability of $Y$ visiting at least $R$ different vertices before getting terminated. (Here $\pesc$ is the constant from \Cref{transience_in_binary_tree}.)}
\item\label{item:large_vol} It contains more than $e^{\Cvol R}$ vertices. (Here $\Cvol$ is a sufficiently large constant, to be specified later.)
\end{enumerate}
Otherwise we say that the $(R+R)$-ball is \emph{typical}.
\end{definition}

\begin{remark}\label{rmk:atypical_ball}
Note that the event that a given $(R+R)$-ball is atypical only depends on the ball itself and the red degrees of its vertices, but not on the rest of the quasi-tree. Point \eqref{item:hit_boundary} is formulated as it is to ensure this property.

Note that by~\eqref{item:hit_boundary} and~\Cref{transience_until_nonbinary}, a walk starting from the centre of a typical $(R+R)$-ball has probability $\le\left(\frac{1}{\log N}\right)^5$ of visiting $\ge R$ different vertices in this ball without encountering a non-binary $(R+R)$-ball first.
\end{remark}

\begin{definition}\label{def:typical_events}
For a given quasi-tree $T$, a walk $X$ on it and (possibly random) times $k_1$ and $k_2$ we say that the event $\Omegatyp{k_1,k_2}$ holds if during the time interval $[k_1,k_2]$ the walk $X$ does not visit an atypical $(R+R)$-ball. We write $\Omegatyp{k}$ to mean $\Omegatyp{0,k}$.

We say that $\Omegabound{k_1,k_2}$ holds if there is no $(R+R)$-ball where the walk $X$ visits $\ge R= C_R\log\log N$ different vertices in time interval $[k_1,k_2]$. We also write $\Omegabound{k}$ to mean $\Omegabound{0,k}$.
\end{definition}

We bound the probability that a randomly sampled $(R+R)$-ball is atypical as follows.

\begin{proposition}\label{bound_atyp_prob}
For any positive constant $b$, there exists a positive constant $C$ such that for sufficiently large values of $C_R$ the following holds. For a randomly sampled $(R+R)$-ball, the probability that it is atypical in the sense of \Cref{def:atypical_ball} is $\le C\left(\frac{1}{\log N}\right)^b$.
\end{proposition}

In the proof of \Cref{bound_atyp_prob} we make use of the following results, the first one of which is a variant of \Cref{volume_tail_bound}.

\begin{lemma}\label{T_volume_tail_bound}
For any radius $r$, any $s>0$, the volume of a ball $B_r$ of radius $r$ around the root of a random quasi-tree (in the graph distance of the quasi-tree) can be bounded as
\[\pr{|B_r|\ge(2d+3)^rs}\quad\le\quad e^{\frac14}e^{-\frac14s}.\]
\end{lemma}

\begin{lemma}\label{small_vol_implies_likely_binary}
There exists $c>0$ with the following property. For a randomly sampled $(R+R)$-ball, the probability that it has volume $\le e^{\Cvol R}$, but is not binary, is $\lesssim e^{-cR^d}$.
\end{lemma}

For a sequence $(p_k)_{k\ge1}$ taking values in $[0,1]$ we say that a random variable $Z$ has \emph{generalised geometric distribution} with parameters $(p_k)$, and write $Z\sim\GGeompos{(p_k)}$, if $\pr{Z=z}=\left(\prod_{k=1}^{z-1}(1-p_k)\right)p_z$ for all $z\ge1$. We write $Z\sim\GGeomnonneg{(p_k)}$ to mean that $(Z+1)\sim\GGeompos{(p_k)}$.

\begin{lemma}\label{bound_range_in_ball_until_tauatyp}
There exist positive constants $c_1$, $c_2$ and $c_3$ (with $c_1$ and $c_2$ depending on $d$, and $c_3$ depending on $d$ and $\beta$), and a sequence $(p_k)_{k\ge1}$ in $(0,1)$ satisfying $p_k\ge c_1\wedge\frac{c_2\log\log N}{k}$ for all $k\le N^{c_3}$, such that the following holds.

Let us consider a copy of $\Gg$ and for each of its vertices $v$, let us attach $\dr(v)\eqdist\dr$  red edges to it, independently of each other and everything else. Then let us consider a walk $Y$ that takes steps like a simple random walk except for the following. Every time it crosses a red edge, with probability $\pesc$ it gets terminated, otherwise it crosses the red edge back and continues.

Then we have
\[|R^{Y}\cap \Vg|\1{\Omega_\mathrm{deg}}\quad\stle\quad \GGeompos{(p_k)},\]
where $R^{Y}$ denotes the range of $Y$ (up to the time it gets terminated), $\Vg$ denotes the set of vertices in the copy of $\Gg$ (i.e.\ the new endpoints of the added red edges are not included), and $\Omega_\mathrm{deg}$ denotes the event that $Y$ does not visit any vertex of degree $>\dmax$ until it gets terminated.
\end{lemma}

Once we have these, the proof of \Cref{bound_atyp_prob} can be concluded as follows.

\begin{proof}[Proof of \Cref{bound_atyp_prob}]
By \Cref{max_deg} the probability that \eqref{item:large_deg} holds is $\ll\frac1N$, and by\\ \Cref{T_volume_tail_bound} the probability that \eqref{item:large_vol} holds is also $\ll\frac1N$ for sufficiently large $\Cvol$.

By \Cref{small_vol_implies_likely_binary} the probability that \eqref{item:nonbinary} holds, but \eqref{item:large_vol} does not hold is $\lesssim e^{-cR^d}\ll\left(\frac{1}{\log N}\right)^b$ for any $b>0$.

By \Cref{bound_range_in_ball_until_tauatyp} the probability that \eqref{item:hit_boundary} holds, but \eqref{item:large_deg} does not can be bounded as
\begin{align*}
\le\quad\pr{\GGeompos{(p_k)}\ge R}\quad=\quad\prod_{k=1}^{R-1}(1-p_k)\quad&\le\quad \prod_{k=\frac{c_2}{c_1}\log\log N}^{C_R\log\log N-1}\left(1-\frac{c_2\log\log N}{k}\right)\\
&\le\quad\exp\left(-\sum_{k=\frac{c_2}{c_1}\log\log N}^{(C_R-1)\log\log N}\frac{c_2\log\log N}{k}\right),
\end{align*}
where for the second inequality we assumed that $C_R>\frac{c_2}{c_1}$. Note that for given positive constants $c_1$ and $c_2$, the sum $\sum_{k=\frac{c_2}{c_1}\log\log N}^{(C_R-1)\log\log N}\frac{c_2}{k}$ can be made arbitrarily large by choosing the constant $C_R$ sufficiently large. This finishes the proof.
\end{proof}

The proof of \Cref{T_volume_tail_bound} is the same as the proof of \Cref{volume_tail_bound}, hence it is omitted. The proofs of \Cref{small_vol_implies_likely_binary} and \Cref{bound_range_in_ball_until_tauatyp} are presented below.

\begin{proof}[Proof of \Cref{small_vol_implies_likely_binary}]
Assume that the volume of the $(R+R)$-ball is $\le e^{\Cvol R}$.
Let $(\Lambda_i)_{i\in I}$ be a collection of copies of $\left[0,\frac1dR\right]^d$ in the $(R+R)$-ball that cover all vertices of the ball, such that $|I|\le e^{\Cvol R}$. For each $\Lambda_i$, the probability that there are at most 2 red edges coming out of it is $\lesssim R^{2d}e^{-aR^d}$ where $a$ is a positive constant. The probability that at least one of the $\Lambda_i$ has this property is $\lesssim e^{\Cvol R}R^{2d}e^{-aR^d}\lesssim e^{-cR^d}$ for some positive constant $c$.
\end{proof}

\begin{proof}[Proof of \Cref{bound_range_in_ball_until_tauatyp}]
For the purposes of the proof we consider a walk $Z$ that takes steps like $Y$, except that it can only get terminated when it crosses a red edge from a vertex in $\Vg$ it visited for the first time. It is clear that $Y$ and $Z$ can be coupled such that $Y$ gets terminated no later than $Z$ and they agree up to this time. Hence for any $k$ we have $\pr{|R^Y\cap\Vg|\1{\Omega_\mathrm{deg}}\ge k}\le\pr{|R^Z\cap\Vg|\ge k,\:\Omega^Z_\mathrm{deg}(k)}$ where $R^Z$ denotes the range of $Z$ and $\Omega^Z_\mathrm{deg}(k)$ is the event that the first $k$ different vertices visited by $Z$ (or all different vertices visited by $Z$ if it gets terminated earlier) have degree $\le\dmax$.

Let us explore the random graph (call it $H$) as we proceed with the walk $Z$. Let $v_0=Z_0$. Whenever we encounter a new vertex $x$, we reveal $\dr(x)$ and $\dg(x)$. When $Z$ crosses a red or green edge for the first time, we also reveal the other endpoint of that edge. We will denote the $k$-th new vertex in $\Vg$ that $Z$ visits by $v_k$, and the $\sigma$-algebra generated by everything we know up to then by $\F_k$. Let $c_3$ be the constant from \Cref{cond_green_red_deg}. For $k>N^{c_3}$ let $p_k=0$. For $k=1,2,..., N^{c_3}$ we do the following in the $k$th round.

(A) Description of the exploration:

We start the $k$th round by revealing the green and red degree of $v_{k-1}$. (The distribution of the green degree depends on $\F_{k-1}$.)

Let $A_k$ be the event that the walk in the next step crosses a red edge from $v_{k-1}$ and gets terminated. Note that by definition $\prscond{A_k}{H}{}=\frac{\dr(v_{k-1})}{\dr(v_{k-1})+\dg(v_{k-1})+2d}\cdot\pesc$, hence $\prscond{A_k}{\F_{k-1}}{}=\econd{\frac{\dr(v_{k-1})}{\dr(v_{k-1})+\dg(v_{k-1})+2d}\cdot\pesc}{\F_{k-1}}$, and that on this event we have $|R^Z\cap\Vg|\le k+1$.

In case $A_k$ does not hold, we proceed with the exploration as follows. With probability $\frac{2d}{\dg(v_{k-1})+2d}$ we let the walk cross a uniformly chosen black edge from $v_{k-1}$. With probability $\frac{\dg(v_{k-1})}{\dg(v_{k-1})+2d}$ the walk crosses a uniformly chosen green edge from $v_{k-1}$. We know what the distribution of the other endpoint of a uniformly chosen green edge from $v_{k-1}$ is, so we sample a vertex according to that and let the walk step there. (Note that this distribution depends on what we have explored so far.)

If the step leads the walk to a yet unvisited vertex, then we call that vertex $v_k$. Otherwise we keep taking steps until the walk visits a new vertex, and we call that vertex $v_k$. We denote the $\sigma$-algebra generated by the exploration up to visiting the $k$th new vertex $v_k$ by $\F_k$.

(B) Estimates on the degrees:

We will define a sequence $(p_k)_{k=0}^{N^{c_3}}$ of constants such that for each $k$ we have $p_k\asymp 1\wedge\frac{\log\log N}{k}$ and $\econd{\frac{\dr(v_{k})}{\dr(v_{k})+\dg(v_{k})+2d}\cdot\pesc}{\F_{k}}\ge p_k$.

We know that $\pr{\dr\ge1}$ is bounded away from 0, say it is lower bounded by a constant $q\in(0,1)$. In what follows, let us work conditional on $\F_{k}$ and assume that none of $v_0,...,v_{k-1}$ has green degree larger than $\dmax$.

Let $\ell\in\{0,2,...,k-1\}$ be such that the walk first visited $v_k$ by crossing an edge from $v_{\ell}$ to $v_k$. Then we know that there is a green or black edge running between $v_{\ell}$ and $v_{k}$. Using \eqref{eq:simg_cond_prob} from \Cref{cond_green_red_deg}, we can bound the number of green connections between $v_k$ and $\{v_0,...,v_{k-1}\}\setminus\{v_{\ell}\}$ as 
\[
\left(\1{v_{j}\simg v_{k}}\right)_{j\in\{0,1,...,k-1\}\setminus\{\ell\}}\quad\stle\quad \left(\Bern{c\dg(v_j)\pgreen_{v_j,v_k}}\right)_{j\in\{0,1,...,k-1\}\setminus\{\ell\}},
\]
where $c$ is a positive constant, and the Bernoulli random variables are independent. \footnote{Consider the $\sigma$-algebra $\G_k$ generated by $\F_k$ and by revealing all green edges running between the vertices $v_0,...,v_{k-1}$. Then conditional on $\G_k$, for each $j\in\{0,1,...,k-1\}\setminus\{\ell\}$, the vertex $v_j$ has $\dg'(v_j)\le\dg(v_j)$ green half-edges that are not yet matched, and the only information we have about their other endpoints is that they do not lie in $\{v_0,...,v_{k-1}\}$. Conditional on $\G_k$, the events $\{v_j\simg v_k\}$ are independent for the different $j\in\{0,1,...,k-1\}\setminus\{\ell\}$ and $\prcond{v_j\simg v_k}{\G_k}=\prcond{v_j\simg v_k}{\dg(v_j)=\dg'(v_j),v_j\not\simg v_i\text{ for }i\in\{0,1,...,k-1\} }$. Using \eqref{eq:simg_cond_prob} and that $k\le N^{c_3}$ we see that this probability is $\lesssim\dg'(v_j)\pgreen_{v_j,v_k}\le\dg(v_j)\pgreen_{v_j,v_k}$.}

On the event we are working on, we have $\dg(v_j)\pgreen_{v_j,v_k}\le\dmax\cdot\pgreen_{\max}\asymp\frac{\log N}{\log\log N}\cdot\frac{1}{\log N}$, hence the above random variable is
\[
\stle\quad\left(\Bern{\frac{\til{C}}{\log\log N}}\right)_{j\in\{0,1,...,k-2\}}
\]
where $\til{C}$ is a positive constant, and the Bernoulli random variables are independent.

For any $v\not\in\{v_0,...,v_k\}$, we have $\1{v\simg v_k}\eqdist\Bern{\pgreen_{v,v_k}}$, independently of anything else. Hence
\[
\dg(v_{k})\quad\stle\quad 1+\Bin{k-1}{\frac{\til{C}}{\log\log N}}+\dg,
\]
where the binomial random variable is independent of $\dg$ (and we recall the definition of $\dg$ from \Cref{sec:further_notation_Gstar}). Also, we have $\dr(v_{k})\ge1$ with probability at least $q$, independently of everything else.

We have
\[
\econd{\frac{\dr(v_{k})}{\dr(v_{k})+\dg(v_{k})+2d}\cdot\pesc}{\F_{k}}\quad\ge \quad\pr{\dr(v_k)\ge1}\cdot\econd{\frac{1}{\dg(v_k)+2d+1}}{\F_{k}}\cdot\pesc.
\]

Here
\[
\econd{\frac{1}{\dg(v_k)+2d+1}}{\F_{k}}\quad\ge\quad\E{\frac{1}{\Bin{k}{\frac{\til{C}}{\log\log N}}+\dg+2d+2}},
\]
where $\Bin{k}{\frac{\til{C}}{\log\log N}}$ and $\dg$ are independent. We know that $\E{\dg}\lesssim 1$, hence for a sufficiently large constant $C'$ the event $\{\dg\le C'\}$ has positive probability. Also, for a sufficiently large constant~$C'$ the event $\left\{\Bin{k}{\frac{\til{C}}{\log\log N}}\le C'\frac{k}{\log\log N}\right\}$ has positive probability. Together these show that the expectation on the right-hand side is $\gtrsim1\wedge\frac{\log\log N}{k}$, therefore
\[
\econd{\frac{\dr(v_{k})}{\dr(v_{k})+\dg(v_{k})+2d}\cdot\pesc}{\F_{k}}\quad\gtrsim\quad1\wedge\frac{\log\log N}{k}\:.
\]
This shows that we can choose $(p_k)$ as required.

(C) Concluding the proof:

Now we are ready to upper bound the probability of the event $\left\{|R^Z\cap\Vg|\ge k,\:\Omega^Z_\mathrm{deg}(k)\right\}$.

If this event holds, then we have $\1{A_0}=....=\1{A_{k-1}}=0$ and up to visiting the $k$th new vertex $v_k$ the walk does not visit any vertex with green degree $>\dmax$. This has probability $\le\prod_{j=0}^{k-1}\left(1-p_j\right)$. Hence $|R^Y\cap\Vg|\1{\Omega_\mathrm{deg}}\stle\GGeompos{(p_k)}$ as required.
\end{proof}

\subsection{Decomposition of a quasi-tree}\label{sec:iid_decomp}

Let us consider a random quasi-tree $T$ and a random walk $X$ on it. Let $\tauatyp$ be the first time that $X$ enters an atypical $(R+R)$-ball. We define the loop-erasure and regenerations of $(X_i)_{i=0}^{\tauatyp}$ as follows.

Let $\xi$ be the \emph{loop-erasure} of $(X_i)_{i=0}^{\tauatyp}$ obtained by considering its path, removing the loops in chronological order, and only keeping the long-range edges in the remaining path.

Let us say that $t$ is a \emph{regeneration time} of $(X_i)_{i=0}^{\tauatyp}$ if $t\le\tauatyp$, $(X_{t-1},X_t)$ is a long-range edge, and this edge is crossed exactly once by $(X_i)_{i=0}^{\tauatyp}$ in either direction. In this case we also say that $(X_{t-1},X_t)$ is a \emph{regeneration edge}. We denote the $j$th regeneration time by $\sigma_j$ and we denote its level $\ell(X_{\sigma_j})$ by~$\varphi_j$. (Note that a.s.\ $\tauatyp$ is finite, and so only a finite number of regenerations are defined. When we write $\xi_j$, $\sigma_j$ or $\varphi_j$, it is always implicitly assumed that it is defined.) {In what follows loop-erasures, regeneration times and regeneration levels are always defined with respect to a walk stopped upon hitting an atypical ball, but we often drop this from the notation.}

Let $T^a$ denote a quasi-tree that is obtained from $T$ by adding a new long-range edge from $\rho$, whose other endpoint is a new vertex $\rho^a$, and consider a walk $X^a$ on $T^a$ started from $\rho$. Let us use notation $\sigma^a$, $\xi^a$, etc to denote the regeneration times, loop-erasure, etc associated to $(T^a,X^a)$.

From \Cref{transience_until_nonbinary} we know that $\prstart{\tauatyp^a<\tau_{\rho^a}}{\rho}\ge\pesc$. Let $(T',X')$ be distributed like $(T^a,X^a)$ conditioned on the event $\{\tauatyp^a<\tau_{\rho^a}\}$.

Then we have
\begin{align*}
\left(T\left(X_{\sigma_1}\right),(X_t)_{t\ge\sigma_1}\right)\:\big|\:\left\{\sigma_1<\tauatyp\right\}\quad\eqdist\quad\left(T',X'\right),
\end{align*}
and conditional on the event $\left\{\sigma_1<\tauatyp\right\}$, the parts $\left(T\setminus T\left(X_{\sigma_1}\right),(X_t)_{t\le\sigma_1}\right)$ and $\left(T\left(X_{\sigma_1}\right),(X_t)_{t\ge\sigma_1}\right)$ are independent.

Let $\til{p}_0:=\prstart{\tauatyp<\sigma_1}{\rho}$ for a walk on $T$ and let $\til{p}:=\prscond{\tauatyp^a<\sigma_1^a}{\tauatyp^a<\tau_{\rho^a}}{\rho}$ for a walk on $T^a$. (Note that these depend on $n$.)

\begin{lemma}\label{iid_decomp}
Let us consider a random walk $X$ on a random quasi-tree $T$, and let\\ $J:=\max\left\{j\ge0:\:\sigma_j<\tauatyp\right\}$. Then the following are satisfied.
\begin{itemize}
\item $J\sim\GGeomnonneg{\til{p}_0,\til{p},\til{p},...}$. \footnote{We recall the notation $\GGeomnonneg{\cdot}$ from before \Cref{bound_range_in_ball_until_tauatyp}.}
\item Conditional on $J=j\ge1$, the following $j+1$ parts are independent: $\left(T\setminus T(X_{\sigma_1}),(X_t)_{t\le\sigma_1}\right)$, $\left(T\left(X_{\sigma_i}\right)\setminus T\left(X_{\sigma_{i+1}}\right),(X_t)_{t\in\left[\sigma_i,\sigma_{i+1}\right]}\right)$ for $i=1,...,j-1$, and $\left(T\left(X_{\sigma_j}\right),\left(X_t\right)_{t\in\left[\sigma_j,\tauatyp\right]}\right)$.
\item Conditional on $J=j\ge1$, the following $j-1$ parts are identically distributed, and their distribution does not depend on $j$: $\left(T\left(X_{\sigma_i}\right)\setminus T\left(X_{\sigma_{i+1}}\right),(X_t)_{t\in\left[\sigma_i,\sigma_{i+1}\right]}\right)$ for $i=1,...,j-1$.
\item Conditional on $J\ge1$, the distribution of $\left(T\left(X_{\sigma_J}\right),\left(X_t\right)_{t\in\left[\sigma_J,\tauatyp\right]}\right)$ does not depend on $J$.
\end{itemize}
\end{lemma}

The proof of \Cref{iid_decomp} and the proofs of the statements above it are analogous to the proof of \cite[Lemma 3.6]{random_matching}, hence they are omitted.

\subsection{Speed and entropy}

{In this section we establish concentration results for the speed and entropy of a random walk on a random quasi-tree. The additional difficulty compared to the analogous results in~\cite{weighted_random_matching} comes from the fact that we only consider the walk $X$ until $\tauatyp$.}

\begin{proposition}\label{bound_entropy_T}
{Let $T$ be a random quasi-tree as in Definition~\ref{def:random_quasi_tree} and let $\xi$ be a loop-erasure of a simple random walk $X$ on $T$ and $\til{\xi}$ be an independent loop-erasure on $T$.\footnote{{Recall that in the definition of loop-erasures and regenerations we consider the walk only until it first hits an atypical ball.}} Then for any $b\in\Zpos$ we have
\[\E{\left(-\log\prscond{\xi_{\varphi_1}\in\til{\xi}}{T,X}{}\right)^b\1{\Omegatyp{\sigma_1}}\1{\Omegabound{\sigma_1}}}\quad\lesssim\quad1.\]
Furthermore, the same statement holds if we replace $(T,X,\xi,\til{\xi})$ with $(T',X',\xi',\til{\xi}')$, where $(T',X')$ is defined as in \Cref{sec:iid_decomp}, $\xi'$ is the loop-erasure of $X'$, and $\til{\xi}'$ is an independent copy of $\xi'$ conditional on $T'$.}
\end{proposition}

\begin{lemma}\label{speed_concentration}
Let $\nu:=\frac{\econd{\varphi'_1}{\Omegabound{\sigma'_1},\:\Omegatyp{\sigma'_1}}}{\econd{\sigma'_1}{\Omegabound{\sigma'_1},\:\Omegatyp{\sigma'_1}}}$, where $\sigma'_1$ and $\varphi'_1$ are the analogues of $\sigma_1$ and $\varphi_1$ for $(T',X')$. Then $\nu\asymp1$, and for any $\theta>0$ there exists $C>0$ such that for all $\ell$ and all sufficiently large $t$ we have
\[
\pr{\left|d_T(\rho,X_t)-\nu t\right|>C\sqrt{t},\:t\le\sigma_{\ell},\:\Omegabound{\sigma_{\ell}},\:\Omegatyp{\sigma_{\ell}}}\quad<\quad\theta,
\]
\[
\pr{\max_{s\in[0,t]}d_T(\rho,X_s)>\nu t+2C\sqrt{t}, \:t\le\sigma_{\ell},\:\Omegabound{\sigma_{\ell}},\:\Omegatyp{\sigma_{\ell}}}\quad<\quad\theta,
\]
where we recall $d_T$ denotes long-range distance in the random quasi-tree $T$.
\end{lemma}

\begin{proposition}\label{entropy_concentration}
We use the notation from~\Cref{bound_entropy_T} and let
{\[
\h:=\quad\frac{1}{\E{\varphi'_1\1{\Omegatyp{\sigma'_1}}\1{\Omegabound{\sigma'_1}}}} \E{\left(-\log\prscond{\xi'_{\varphi'_1}\in\til{\xi'}}{T',X'}{}\right)\1{\Omegatyp{\sigma'_1}}\1{\Omegabound{\sigma'_1}}},\qquad\V:=\log\log N.
\]}

For all $\theta>0$ and $\widehat{C}>0$, there exists a positive constant $C$ so that for all $k\le\widehat{C}\log N=:\ell$ we have

{\[
\pr{\left|-\log\prscond{\xi_{k}=\til{\xi}_k}{T,X}{} -\h k\right|>C\sqrt{k\V},\:\Omegatyp{\sigma_{\ell}},\Omegabound{\sigma_{\ell}}}\quad\le\quad\theta.
\]}
\end{proposition}

Let us choose $c_L$ in the definition of $L^-$ and $L^+$ in \Cref{sec:parameters} as $c_L=\frac{1}{2\h}$ (with $\h$ as in \Cref{entropy_concentration}).

{A major step in the proof of the above statements is showing that the regeneration times are of constant order, in the sense that they are stochastically dominated by a constant shift of a power of a geometric random variable. For this, we first stochastically bound the range of the walk until the first regeneration time, in~\Cref{bound_R_sigma1}. This and a bound on the time it takes for the walk to visit a given number of different vertices establishes~\Cref{bound_sigma1}.}

\begin{lemma}\label{bound_R_sigma1}
There exist constants $C\in(0,\infty)$ and $c\in(0,1)$ with the following property.
If $T$ is a random quasi-tree as in Definition~\ref{def:random_quasi_tree} and $X$ is a random walk on it, then we have
\[|R_{\sigma_1\wedge\tauatyp\wedge\taubound}|\quad\stle\quad C+(\Geom{c})^3,\]
where $R_k$ denotes the range of $X$ up to time $k$, $\taubound$ is the first time that the walk has visited $R$ distinct vertices in an $(R+R)$-ball, and we set $\sigma_1\wedge\tauatyp\wedge\taubound=\tauatyp\wedge\taubound$ in the case that $\sigma_1$ is not defined (i.e.\ there is no regeneration until $\tauatyp$).

Furthermore, the same statement holds for $(T',X')$ instead of $(T,X)$.
\end{lemma}

\begin{lemma}\label{bound_sigma1}
There exist constants $C\in(0,\infty)$ and $c\in(0,1)$ with the following property. If $T$ is a random quasi-tree as in Definition~\ref{def:random_quasi_tree} and $X$ is a random walk on it, then we have
\[
\sigma_1\wedge\tauatyp\wedge\taubound\quad\stle\quad C+\left(\Geom{c}\right)^{15},
\]
where $\sigma_1\wedge\tauatyp\wedge\taubound$ is defined as in \Cref{bound_R_sigma1}.

Furthermore, the same statement holds for $(T',X')$ instead of $(T,X)$.
\end{lemma}

The proofs of \Cref{bound_R_sigma1} and \Cref{bound_sigma1} are presented below.

\begin{proof}[Proof of \Cref{bound_R_sigma1}]
Firstly, we prove that for each level $\ell$ the number of different vertices visited by the walk at level $\ell$ can be bounded as
\begin{align}\label{eq:bound_range_in_one_level}
\pr{|R_{\tauatyp\wedge\taubound}\cap V_\ell|\ge k}\quad\lesssim\quad e^{-(\const)k}.
\end{align}
Here $V_\ell$ denotes the set of vertices at level $\ell$.

Note that when the walk is at a vertex $v$ that is the $j$th new vertex seen in its $(R+R)$-ball, the probability of crossing a red edge in the next step is $\ge q_j$, where $q_j\asymp1\wedge\frac{\log\log N}{j}$ for $j\le N^{\const}$. (The proof of this is similar to the proof of \Cref{bound_range_in_ball_until_tauatyp} and is omitted.) Also note that $\min_{j\le R}q_j\asymp1$ and by \Cref{transience_until_nonbinary} if the walk crosses a red edge, there is probability at least $\pesc\asymp1$ that it does not backtrack it until $\tauatyp$. Together these give \eqref{eq:bound_range_in_one_level}.

Next, we bound the probability of the walk visiting more than $k$ levels until $\sigma_1\wedge\tauatyp\wedge\taubound$ as
\begin{align}\label{eq:bound_sigma1_ge_tauk_prob}
\pr{\tau_k<\sigma_1\wedge\tauatyp\wedge\taubound}\quad \leq \quad \pr{\tau_k<\sigma_1\wedge\tauatyp}\quad\lesssim\quad e^{-(\const)\sqrt{k}}.
\end{align}

{If $\tau_k<\sigma_1\wedge\tauatyp$ then none of the red edges crossed by time $\tau_k$ is a regeneration edge, i.e.\ each of them gets backtracked by time $\tauatyp$.}

{In this case at least one of the following events happens: for some $\sqrt{k}\le\ell\le k$ the walk revisits level $\ell-\sqrt{k}$ after hitting level $\ell$, or the walk hits levels $\sqrt{k}$, $\sqrt{k}-1$, $2\sqrt{k}$, $2\sqrt{k}-1$, ... $ k$, $k-1$ in this order.} We can bound the probabilities of these as follows,
\begin{align*}
\sum_{\ell=\sqrt{k}}^{k}\pr{\text{hit level }\ell\text{ and then hit level }\ell-\sqrt{k},\text{ before }\tauatyp}\quad
\lesssim\quad k e^{-(\const)\sqrt{k}}\quad\lesssim\quad e^{-(\const)\sqrt{k}},
\end{align*}
\begin{align*}
\pr{\text{hit levels }\sqrt{k},\:\sqrt{k}-1,\:2\sqrt{k},\:2\sqrt{k}-1,...\:k,\:k-1\text{ in this order, before }\tauatyp}\:\lesssim\: e^{-(\const)\sqrt{k}}.
\end{align*}

Combining these bounds, we get \eqref{eq:bound_sigma1_ge_tauk_prob}.

Using \eqref{eq:bound_range_in_one_level} to bound $\pr{|R_{\tauatyp\wedge\taubound}\cap V_\ell|\ge k^{1/3}}$ for each $\ell\le k^{2/3}$ and using \eqref{eq:bound_sigma1_ge_tauk_prob} to bound\\ $\pr{\tau_{k^{2/3}}<\sigma_1\wedge\tauatyp\wedge\taubound}$ we get that $\pr{|R_{\sigma_1\wedge\tauatyp\wedge\taubound}|\ge k}\lesssim e^{-(\const)k^{1/3}}$ as required.

This finishes the proof for $(T,X)$.

With an analogous proof, we can show an analogous result for $(T^a,X^a)$, and then we can use that $\tau_{\rho^a}>\tauatyp^a$ has probability bounded away from 0 to get the result for $(T',X')$.
\end{proof}

\begin{proof}[Proof of \Cref{bound_sigma1}]
Given a random walk $X$ on $T$, let $s_k:=\min\left\{t:\:|R_t|\ge k\right\}$. Then we have
\[
\pr{\sigma_1\wedge\tauatyp\wedge\taubound\ge2k} \quad\le\quad\pr{\left|R_{\sigma_1\wedge\tauatyp\wedge\taubound}\right|\ge k^{1/5}}+\pr{s_{k^{1/5}}\ge 2k,\:\tauatyp>2k}.
\]
From \Cref{bound_R_sigma1} we know that the first term on the RHS is $\lesssim e^{-(\const)k^{1/15}}$.

Now we bound the second term. Let $R_{k^{1/5}}=\{v_1,...,v_{k^{1/5}}\}$ where $v_i$ are in the order $X$ first visits them. Then
\begin{align*}
\pr{s_{k^{1/5}}\ge 2k,\:\tauatyp>2k}\quad& \le\quad\sum_{j=0}^{k^{1/5}-1}\pr{s_{j+1}-s_{j}\ge2k^{4/5}}\quad\\
&\le\quad\sum_{j=0}^{k^{1/5}-1}\E{\prscond{\tau_{\{v_1,...,v_{j}\}^c}>2k^{4/5}}{T,v_1,...,v_j}{v_j}}
\end{align*}

Given $T$ and $v_1,...,v_{j}$ we can bound $\prscond{\tau_{\{v_1,...,v_{j}\}^c}>2k^{4/5}}{T,v_1,...,v_j}{v_j}$ as follows. Let $H$ be a graph with vertex set $\{v_1,...,v_j,v_{j,\partial}\}$ obtained by contracting $\{v_1,...,v_{j}\}^c$ into a single vertex $v_{j,\partial}$ and for each $i$ in case there are multiple edges between $v_i$ and $v_{j,\partial}$, deleting all of them except for one. By the commute time identity~\cite[Corollary 2.21]{prob-on-trees-and-networks} we know that for any $u\in\{v_1,...,v_j\}$ a SRW on $H$ from $u$ satisfies $\estart{\tau_{v_{j,\partial}}}{u}\le e(H)\Reffin{u}{v_{j,\partial}}{H}$ where $e(H)\le j^2$ is the number of edges in $H$ and $\Reffin{u}{v_{j,\partial}}{H}\le j$ is the effective resistance between $u$ and $v_{j,\partial}$. Then $\prstart{\tau_{v_{j,\partial}}>2j^3}{u}\le\frac12$ for any $u$, hence
\[
\prscond{\tau_{\{v_1,...,v_{j}\}^c}>2k^{4/5}}{T,v_1,...,v_j}{v_j}\quad\le\quad\prstart{\tau_{v_{j,\partial}}>2k^{4/5}}{v_{j}}\quad\le\quad2^{-k^{1/5}}.
\]

Taking a union bound over $j$ we get that $\pr{s_{k^{1/5}}\ge 2k,\:\tauatyp>2k}\lesssim k^{1/5}e^{-(\const)k^{1/5}}\lesssim e^{-(\const)k^{1/5}}$. This finishes the proof for $(T,X)$.

With an analogous proof we can also show that for $(T^a,X^a)$ we have $\pr{\sigma^a_1\wedge\tauatyp^a\wedge\taubound^a\ge k}\lesssim e^{-(\const)k^{1/15}}$. Then using that $\tau_{\rho^a}>\tauatyp^a$ has probability bounded away from 0, we also get the result for $(T',X')$.
\end{proof}

We introduce some notation for quantities related to entropy that will be used in the following proofs.

\begin{definition}\label{def:Hb}
For $b\in\Zpos$ and for a sequence $(p_i)$ of real numbers taking values in $[0,1]$ let
\[\ent{b}{p_1,p_2,...}:=\quad\sum_{i}p_i(-\log p_i)^b,\]
and for a random variable $W$ taking values in a countable set $\mathcal{W}$ let
\[\ent{b}{W}:=\quad\sum_{w\in\mathcal{W}}\pr{W=w}\left(-\log\pr{W=w}\right)^b.\]
In both cases, $p(-\log p)^b$ is considered to be $0$ when $p=0$.
\end{definition}

\begin{definition}\label{def:Hb_conditional}
For $b\in\Zpos$ and for random variables $W$ and $Z$ taking values in countable sets $\mathcal{W}$ and $\mathcal{Z}$ respectively let
\begin{align*}\entc{b}{W}{Z}:&=\quad\sum_{z\in\mathcal{Z}}\pr{Z=z}\ent{b}{W|Z=z}\\
&=\quad\sum_{z\in\mathcal{Z}}\pr{Z=z}\sum_{w\in\mathcal{W}}\prcond{W=w}{Z=z}{}\left(-\log\prcond{W=w}{Z=z}{}\right)^b.
\end{align*}
Here $\pr{Z=z}\prcond{W=w}{Z=z}{}\left(-\log\prcond{W=w}{Z=z}{}\right)^b$ is considered to be $0$ when\\ $\pr{Z=z}=0$ or $\prcond{W=w}{Z=z}{}=0$.

We can also extend the definition to the case when $Z$ takes values in an uncountable set $\mathcal{Z}$, but for almost all $z\in\mathcal{Z}$, the random variable $\left(W\mid Z=z\right)$ takes values from a countable set, by considering $\ent{b}{W|Z=z}$ and taking an expectation over $Z$.
\end{definition}

\begin{definition}\label{def:HbA_and_HbA_conditional}
For $b\in\Zpos$, random variables $W$ and $Z$ taking values in countable sets $\mathcal{W}$ and $\mathcal{Z}$ respectively, and event $A$, let
\begin{align*}
\entr{b}{A}{W}:&=\quad\sum_{w\in\mathcal{W}}\pr{W=w,\:A}\left(-\log\pr{W=w,\:A}\right)^b\\
\entrc{b}{A}{W}{Z}:&=\quad\sum_{z\in\mathcal{Z}}\pr{Z=z}\entr{b}{A}{W|Z=z}.
\end{align*}
We interpret the cases when one of the probabilities is 0 as before. We extend the definition to the case when $\mathcal{Z}$ is uncountable but $\left(W\mid Z=z\right)$ takes values from a countable set for almost all $z\in\mathcal{Z}$ analogously to \Cref{def:Hb_conditional}.
\end{definition}

Some simple estimates regarding these quantities are listed in \Cref{app:entropy}, we make use of these in the following proofs.

\begin{proof}[Proof of \Cref{bound_entropy_T}]
First note that
{\[\E{\left(-\log\prscond{\xi_{\varphi_1}\in\til{\xi}}{T,X}{}\right)^b\1{\Omegatyp{\sigma_1}}\1{\Omegabound{\sigma_1}}}\quad\le\quad \entrc{b}{A}{\xi_{\varphi_1}}{T},\]}
where $A=\Omegatyp{\sigma_1}\cap\Omegabound{\sigma_1}$.

By \Cref{HbA_cond_properties}\eqref{item:HbA_cond_i} we have
\[
\entrc{b}{A}{\xi_{\varphi_1}}{T}\quad\lesssim\quad \entrc{b}{A}{X_{\sigma_1}}{T,\left|R_{\sigma_1}\right|}+\entrc{b}{A}{\left|R_{\sigma_1}\right|}{T}.
\]

From \Cref{HbA_cond_properties}\eqref{item:HbA_cond_i} and \Cref{bound_R_sigma1} we get that
\[
\entrc{b}{A}{\left|R_{\sigma_1}\right|}{T}\:\le\: \entr{b}{A}{\left|R_{\sigma_1}\right|}+O(1) \:=\:\entr{b}{A}{\left|R_{\sigma_1\wedge\tauatyp\wedge\taubound}\right|}+O(1)\:\lesssim\: \ent{b}{\left|R_{\sigma_1\wedge\tauatyp\wedge\taubound}\right|}+O(1)\:\lesssim\:1.
\]

Also by \Cref{HbA_cond_properties}\eqref{item:HbA_cond_iii}, we have
\[
\entrc{b}{A}{X_{\sigma_1}}{T,\left|R_{\sigma_1}\right|}\quad\lesssim\quad \E{\left(\log V_T\left(\left|R_{\sigma_1}\right|\right)\right)^b\1{A}}+\E{\prscond{A}{T,\left|R_{\sigma_1}\right|}{}\left(-\log\prscond{A}{T,\left|R_{\sigma_1}\right|}{}\right)^b}
\]
where $V_T(r)$ denotes the volume of the ball of radius $r$ around the root in $T$.

The second expectation is $\lesssim1$ since the function $p(-\log p)^b$ is bounded for $p\in[0,1]$. We can bound the first expectation as follows.

Let $\lambda\in\left(0,\frac{1}{\log(2d+3)}\right)$ be a constant. Then
\[
\E{\left(\log V_T\left(\left|R_{\sigma_1}\right|\right)\right)^b\1{A}}\quad\le\quad\sum_{s\ge1}\pr{\left(\log V_T\left(\left|R_{\sigma_1}\right|\right)\right)^b\1{A}\ge s}+1
\]
\[
\le\quad\sum_{s\ge1}\pr{\left|R_{\sigma_1}\right|\ge\lambda s^{1/b},A}+ \sum_{s\ge1}\pr{V_T\left(\lambda s^{1/b}\right)\ge \exp\left(s^{1/b}\right),A}+1.
\]

From \Cref{bound_R_sigma1} we get
\[
\sum_{s\ge1}\pr{\left|R_{\sigma_1}\right|\ge\lambda s^{1/b},A}\quad=\quad\E{\left(\frac{\left|R_{\sigma_1}\right|}{\lambda}\right)^b\1A}\quad\lesssim\quad1.
\]

Also, using \Cref{T_volume_tail_bound} and that $\lambda<\frac{1}{\log(2d+3)}$ we get that
\[
\sum_{s\ge1}\pr{V_T\left(\lambda s^{1/b}\right)\ge\exp\left(s^{1/b}\right),A}\quad\le\quad \sum_{s\ge1}\exp\left(\frac14-\frac14\exp\left(s^{1/b}\right)(2d+3)^{-\lambda s^{1/b}}\right)
\]
\[
\asymp\quad\sum_{s\ge1}\exp\left(-\frac14\exp\left(s^{1/b}\left(1-\lambda\log(2d+3)\right)\right)\right)\quad\lesssim\quad1.
\]
This finishes the proof for $(T,X,\til{\xi})$. The proof for $(T',X',\til{\xi}')$ is analogous.
\end{proof}

\begin{proof}[Proof of \Cref{speed_concentration}]
Using \Cref{bound_sigma1} and that the events $\Omegatyp{\sigma_1}\cap\Omegabound{\sigma_1}$ and $\Omegatyp{\sigma'_1}\cap\Omegabound{\sigma'_1}$ have probability bounded away from 0, we get that the variables $\left(\sigma'_1\mid\Omegatyp{\sigma'_1}\cap\Omegabound{\sigma'_1}\right)$, $\left(\varphi'_1\mid\Omegatyp{\sigma'_1}\cap\Omegabound{\sigma'_1}\right)$, $\left(\sigma_1\mid\Omegatyp{\sigma_1}\cap\Omegabound{\sigma_1}\right)$ and $\left(\varphi_1\mid\Omegatyp{\sigma_1}\cap\Omegabound{\sigma_1}\right)$ each have expectation $\asymp1$ and variance $\lesssim1$.

This shows $\nu\asymp1$.

From \Cref{iid_decomp} we know that conditional on $\Omegatyp{\sigma_\ell}\cap\Omegabound{\sigma_\ell}$, the variables
\[
\left((\sigma_k-\sigma_{k-1}),(\varphi_k-\varphi_{k-1})\right)_{k=2,...,\ell}
\]
are independent and have the same distribution as $\left(\left(\sigma'_1,\varphi'_1\right)|\Omegatyp{\sigma'_1}\cap\Omegabound{\sigma'_1}\right)$. They are also independent of the pair $(\sigma_1,\varphi_1)$, which is distributed as $\left((\sigma_1,\varphi_1)|\Omegatyp{\sigma_1}\cap\Omegabound{\sigma_1}\right)$. Using this and the bounds on the expectation and variance of these variables, we get that for any $\theta$, for sufficiently large $C$ and for each $i\le\ell$ we have
\begin{align*}
\pr{\left|d_T(\rho,X_{\sigma_i})-\nu\sigma_i\right|>C\sqrt{\sigma_i},\:\Omegatyp{\sigma_\ell},\:\Omegabound{\sigma_\ell}}\hspace{5cm}\\
\le\quad\prscond{\left|d_T(\rho,X_{\sigma_i})-\nu\sigma_i\right|>C\sqrt{\sigma_i}}{\Omegatyp{\sigma_\ell},\:\Omegabound{\sigma_\ell}}{}\quad<\quad\theta.
\end{align*}

We can extend the result for all times $t$, and we can prove the second bound similarly to the proof of \cite[Lemma 2.19]{weighted_random_matching}. For the parts of the proof where we want to use the concentration of the sum of iid random variables we condition on $\Omegatyp{\sigma_\ell}\cap\Omegabound{\sigma_\ell}$, while for parts where we want to bound the probability of backtracking many levels, we just include the event $\Omegatyp{\sigma_\ell}\cap\Omegabound{\sigma_\ell}$ in the probability.
\end{proof}

\begin{proof}[Proof of \Cref{entropy_concentration}]
Let us write $A_{i}:=\Omegatyp{\sigma_{i-1},\sigma_i}\cap\Omegabound{\sigma_{i-1},\sigma_i}$, $A_{\le j}:=\bigcap_{i\le j}A_i$, $A_{\ge j}:=\bigcap_{j\le i\le\ell}A_i$ and $A:=A_{\le\ell}$. Let $\Delta_i=\varphi_i-\varphi_{i-1}$.

{Also let $Y_i:=-\log\frac{\prscond{\xi_{\varphi_i}\in\til{\xi}}{T,X}{} }{\prscond{\xi_{\varphi_{i-1}}\in\til{\xi}}{T,X}{}}$ for $i\geq 2$ and let $Y_1:= -\log\prscond{\xi_{\varphi_1}\in\til{\xi}}{T,X}{}$.}
 Then 
{\[
 -\log\prscond{\xi_{\varphi_k}\in\til{\xi}}{T,X}{}\quad =\quad \sum_{i=1}^{k}Y_i.
 \]}

By Markov's inequality
\[
\pr{\left|\sum_{i=1}^{k}Y_i-\varphi_k\h\right|>C\sqrt{k\V},A}\quad\le\quad\frac{\E{\left(\sum_{i=1}^{k}(Y_i-\Delta_i\h)\right)^2\1{A}}}{C^2k\V}.
\]
In what follows, we will prove that
\begin{align}\label{eq:bound_sum_Yisq}
\E{\left(\sum_{i=1}^{k}(Y_i-\Delta_i\h)\right)^2\1{A}}\quad\lesssim\quad k\log\log N,
\end{align}
which implies the desired concentration bound along the regeneration times. Once we have this, we can extend the result for all times similarly to the proof of \cite[Proposition 2.20]{weighted_random_matching}. The details of this are omitted.

In what follows, we work towards the proof of \eqref{eq:bound_sum_Yisq} by establishing bounds for
\[
\E{\left(Y_i-\Delta_i\h\right)^2\1{A}}\qquad\text{and}\qquad\E{\left(Y_i-\Delta_i\h\right)\left(Y_j-\Delta_j\h\right)\1{A}}
\]
for each $i$ and each $i\ne j$

Note that $(Y_i)_{i\ge2}$ is stationary and each term is distributed as {$Y':=-\log\prscond{\xi'_{\varphi'_1}\in\til{\xi}'}{T',X'}{}$. }
Let us define $A'_1:=\Omegatyp{\sigma'_1}\cap\Omegabound{\sigma'_1}$ corresponding to the walk $X'$ in the definition of $Y'$.

Also note that $Y_i=-\log\prscond{X_{\sigma_i}\in\xi(i)}{(X_t)_{t\ge\sigma_{i-1}},T(X_{\sigma_{i-1}})}{}$,
where $\xi(i)$ is the loop-erasure of a random walk $X^i$ on $T(X_{\sigma_{i-1}})$ started from $X_{\sigma_{i-1}}$, {stopped upon hitting an atypical ball}\footnote{{Here we are using that for a given realisation of a quasi-tree $T^a$ (with parent of the root $\rho^a$) the distribution of the loop-erasure of the simple random walk (stopped at $\tauatyp$) from $\rho$ is the same as the distribution of the loop-erasure of the simple random walk (stopped at $\tauatyp$) from $\rho$ conditioned on not hitting $\rho^a$.}}, but otherwise independent of $(X_t)_{t\ge\sigma_{i-1}}$.

Firstly, we prove that $\E{\left(Y_i-\Delta_i\h\right)\left(Y_j-\Delta_j\h\right)\1{A}}\lesssim1$ for all $i,j\in\{1,2,...,k\}$.
\begin{itemize}
\item From \Cref{bound_entropy_T} we know that
\[\E{Y_1^2\1{A}}\quad\le\quad\E{Y_1^2\1{A_1}}\quad\lesssim\quad1,\]
and that for any $i\ge2$ we have
\[\E{Y_i^2\1{A}}\quad\le\quad\E{Y_i^2\1{A_i}}\quad=\quad\E{(Y')^2\1{A'_1}}\quad\lesssim\quad1.\]
Also, by \Cref{bound_entropy_T} we have $\h\lesssim1$ and by \Cref{bound_sigma1} we have $\E{\Delta_i^2\1{A}}\lesssim1$, hence for each $i\in\{1,2,...,\ell\}$ we have
\[\E{\left(Y_i-\Delta_i\h\right)^2\1{A}}\quad\lesssim\quad1.\]
\item Using Cauchy-Schwarz we get that $\E{\left(Y_i-\Delta_i\h\right)\left(Y_j-\Delta_j\h\right)\1{A}}\lesssim1$ for any $i,j\in\{1,2,...,k\}$.
\end{itemize}
Now we proceed to show that there exist positive constants $u$ and $U$ such that for all $2\le i<j\le\ell$ with $j-i\ge U\log\log N$ we have
\begin{align}\label{eq:bound_YiYj}
\E{\left(Y_i-\Delta_i\h\right)\left(Y_j-\Delta_j\h\right)\1{A}}\quad\lesssim\quad e^{-u\left(\frac{j-i}{\log\log N}\right)^{1/3}}+\frac{1}{\log N}\:.
\end{align}
Since $k\le\ell\asymp\log N$, this concludes the proof of \eqref{eq:bound_sum_Yisq}.
\begin{itemize}
\item We define
\begin{align*}
B_{i,j}:=&\quad\left\{\left|R_{[\sigma_{i-1},\sigma_{i}]}\right|\le \frac{j-i}{U'\log\log N}\right\},\\
Y_{i,j}:=&\quad-\log\prscond{X_{\sigma_i}\in\xi(i,j)}{X,T(X_{\sigma_{i-1}})}{},
\end{align*}
where $U'$ is a sufficiently large constant, $R_{[t_1,t_2]}$ denotes the range of $X$ in the time interval $[t_1,t_2]$, and $\xi(i,j)$ is the loop-erasure of $X^i$ up to the first time that it hits level $\varphi_{j-1}$ {or hits an atypical ball, whichever occurs first}. (Recall that $X^i$ is a random walk on $T(X_{\sigma_{i-1}})$ started from $X_{\sigma_{i-1}}$, but otherwise independent from $(X_t)_{t\ge\sigma_{i-1}}$.)
\item Then
\begin{enumerate}[(i)]
\item \label{item:indep} $B_{i,j}$ and $Y_{i,j}\1{B_{i,j}}$ are independent of $Y_j\1{A_{\ge j}}$,\\
${}$\hfill(since they depend on different layers between regenerations)
\item \label{item:bound_prob} $\pr{B_{i,j}^c\cap A_i}\lesssim e^{-\const\cdot\left(\frac{j-i}{\log\log N}\right)^{1/3}}$,\hfill(from \Cref{bound_R_sigma1})
\item \label{item:bound_diff} $|Y_i-Y_{i,j}|\1{B_{i,j}}\1{A_i}\le e^{-\const\cdot(j-i)}$.\hfill (see proof below)
\item \label{item:bound_Y} $Y_{i,j}\1{B_{i,j}}\1{A_i}\lesssim(j-i)$.\hfill (see proof below)
\end{enumerate}
\item We can prove~\eqref{item:bound_diff} as follows. Let
\[Z_i:=\prscond{X_{\sigma_i}\in\xi(i)}{X,T(X_{\sigma_{i-1}})}{}\quad\text{and}\quad Z_{i,j}:=\prscond{X_{\sigma_i}\in\xi(i,j)}{X,T(X_{\sigma_{i-1}})}{}.\]
Then $|Y_i-Y_{i,j}|\le\frac{|Z_i-Z_{i,j}|}{Z_i\wedge Z_{i,j}}$. Note that having $X_{\sigma_i}\in\xi(i)$, but $X_{\sigma_i}\not\in\xi(i,j)$ or the other way around would require backtracking $\ge(j-i)$ levels. The probability of backtracking $\ge(j-i)$ levels without visiting an atypical ball is $\le e^{-\const\cdot(j-i)}$. Also, by considering a path from $X_{\sigma_{i-1}}$ to $X_{\sigma_i}$ on $R_{[\sigma_{i-1},\sigma_{i}]}$
and noting that on event $A_i$ each vertex on the path has degree $\lesssim\log N$ we get that on event $B_{i,j}\cap A_i$ we have $(Z_{i}\wedge Z_{i,j})\gtrsim e^{-\const\cdot\log\log N\cdot\frac{j-i}{U'\log\log N}}$. Choosing $U'$ to be sufficiently large this shows that $|Y_i-Y_{i,j}|\1{B_{i,j}}\1{A_i}\le e^{-\const\cdot(j-i)}$.
\item Part~\eqref{item:bound_Y} also follows from noting that on the event $B_{i,j}\cap A_i$ we have $Z_{i,j}\gtrsim e^{-\const\cdot\log\log N\cdot\frac{j-i}{U'\log\log N}}$.
\item Using Cauchy-Schwarz, we get that 
\[
\E{\left(Y_i-\Delta_i\h\right)\left(Y_j-\Delta_j\h\right)\1{A}\1{B_{i,j}^c}}\:\le\:\left(\E{\left(Y_i-\Delta_i\h\right)^4\1{A}}\E{\left(Y_j-\Delta_j\h\right)^4\1{A}}\pr{B_{i,j}^c\cap A}^2\right)^{\frac14}.
\]
From \Cref{bound_entropy_T}, \Cref{bound_sigma1} and~\eqref{item:bound_prob} we get that this is $\lesssim e^{-\const\cdot\left(\frac{j-i}{\log\log N}\right)^{1/3}}$.
\item Also, we have
\begin{align*}
&\E{\left(Y_i-\Delta_i\h\right)\left(Y_j-\Delta_j\h\right)\1{A}\1{B_{i,j}}}\\
&=\quad\E{\left(Y_i-Y_{i,j}\right)\left(Y_j-\Delta_j\h\right)\1{A}\1{B_{i,j}}}+ \E{\left(Y_{i,j}-\Delta_i\h\right)\left(Y_j-\Delta_j\h\right)\1{A}\1{B_{i,j}}}\\
&\le\quad\left(\E{\left(Y_i-Y_{i,j}\right)^2\1{A_i}\1{B_{i,j}}}\E{\left(Y_j-\Delta_j\h\right)^2\1{A_j}}\right)^{\frac12}\\
&\hspace{3cm}+\quad\E{\left(Y_{i,j}-\Delta_i\h\right)\1{A_{\le j-1}}\1{B_{i,j}}} \E{\left(Y_j-\Delta_j\h\right)\1{A_{\ge j}}}.
\end{align*}
From \Cref{bound_entropy_T} and~\eqref{item:bound_diff} we get that the first term in the last line is $\lesssim e^{-\const\cdot(j-i)}$.
\item We can write
\[
\E{\left(Y_j-\Delta_j\h\right)\1{A_{\ge j}}}\quad=\quad\E{\left(Y_j-\Delta_j\h\right)\1{A_{j}}}-\E{\left(Y_j-\Delta_j\h\right)\1{A_j}\1{A_{\ge j+1}^c}}.
\]
Note that $\E{Y_j\1{A_j}}=\E{\Delta_j\1{A_j}}\h$, hence the first term is 0. Also, by Cauchy-Schwarz, the second term has absolute value
\[\le\quad\left(\E{\left(Y_j-\Delta_j\h\right)^2\1{A_j}}\pr{A_{\ge j+1}^c}\right)^{\frac12}.\]
We can show that this is $\lesssim\left(\frac{1}{\log N}\right)^2$ by using \Cref{bound_entropy_T} and that by \Cref{rmk:atypical_ball}, \Cref{bound_atyp_prob} and \Cref{bound_sigma1} we have $\pr{A_{\ge j+1}^c}\lesssim(\ell-j)\left(\frac{1}{\log N}\right)^5\lesssim\left(\frac{1}{\log N}\right)^{4}$.
\item Using~\eqref{item:bound_Y} and that $\h\lesssim1$ and $\E{\Delta_i\1{A_i}}\lesssim1$, we also get
\[\left|\E{\left(Y_{i,j}-\Delta_i\h\right)\1{A_{\le j-1}}\1{B_{i,j}}}\right|\quad\lesssim\quad j-i\quad\lesssim\quad\log N.\]
\item Putting these bounds together, we conclude the proof of \eqref{eq:bound_YiYj}.\qedhere
\end{itemize}
\end{proof}

\section{Relating $G^*$ and $T$}\label{sec:Gstar_and_T}

In this section, we explain how to use $T$ to approximate $G^*$, we define what we mean by sufficiently nice neighbourhoods of $G^*$, and we conclude the bound on the time $\taublue^{(2)}$ transition probabilities between vertices in such neighbourhoods.

We know from \Cref{max_deg} and \Cref{volume_high_prob_bound} that with high probability the maximal degree in $G^*$ is at most $\frac{\log N}{\log \log N}$ and that there are constants $C$ and $\Cvol$ such that with high probability for all $r>C\log\log  N$ the volume in $G^*$ of ball of radius $r$ (in black, green, red and blue edges) around any vertex is at most $e^{\Cvol r}$. For the rest of the section, we work on the event that $G^*$ satisfies both of these high probability events. The constant $\Cvol$ can be taken to be sufficiently large and to agree with the constant in the definition of atypical ball.

\subsection{Truncation and $K$-roots}

Later we will present a way of exploring a random neighbourhood in $G^*$ and coupling it to the first $L^+$ levels of a quasi-tree $T$. To increase the chances of this coupling succeeding, we would like to limit the volume of the region that we reveal in $G^*$. In particular, if a part of $T$ is unlikely to get visited by a random walk, we would not like to continue exploring $G^*$ in that direction. Below we define a `truncation criterion' that roughly speaking holds for a given long-range edge $e$ if it is unlikely that a walk on $T$ ever crosses $e$. The exact definition will be a bit more intricate, to make sure that the criterion only depends on parts of $T$ that we have already revealed by the time we arrive to $e$ in the exploration.

We define two quantities below, $\til{W}$ and $W$. We use $\til{W}$ in the definition of the truncation criterion, and to approximate $W$.

\begin{definition}\label{def:W_Wtil}
For a long-range edge $e$ in a quasi-tree $T$ let
\[
W_T(e):=\quad -\log\prscond{e\in\xi}{T}{\rho},\qquad
\til{W}_T(e):=\quad -\log\prscond{X_{\tau_\ell(e)}=e^+}{T}{\rho}.\qedhere
\]
\end{definition}

Recall that $\tau_\ell$ denotes the first hitting time of level $\ell$ and that we only consider the walk $X$ on $T$ up to time $\tauatyp$, and we define the loop-erasure $\xi$ based on this. In particular, it is possible that $X$ or $\xi$ does not reach level $\ell(e)$.

We bound the difference between $W$ and $\til{W}$ as follows.

\begin{lemma}\label{compare_W_Wtil}
There exists a positive constant $C$ so that for any realisation $T$ of a random quasi-tree and any long-range edge $e$ in $T$ we have
\[
W_T(e)\quad\ge\quad\til{W}_T(e)-CR^4.
\]
Also,
\[
W_T(e)\quad\le\quad\til{W}_T(e)+\log\left(\frac1\pesc\right),
\]
where $\pesc$ is the constant from \Cref{transience_until_nonbinary}.
\end{lemma}

\begin{proof} We drop the conditioning on $T$ from the notation.

The proof of the first bound follows the proof of~\cite[Lemma 4.3]{random_matching} with the following modifications.

Instead of $\infty$ we write $\tauatyp$.

Note that in a typical $(R+R)$-ball each vertex has degree $\le\dmax\asymp\frac{\log N}{\log\log N}$, hence for any two vertices of the ball there is a path from one to the other which has probability at least $\left(\frac{1}{\dmax}\right)^{4R}\ge\exp\left(-cR^2\right)$. Hence~\cite[(4.6)]{random_matching} holds for $i\le\ell(e)-c_3R^2$, and the right-hand side of~\cite[(4.7)]{random_matching} becomes $\exp\left(-cR^2\right)$.

The second bound is immediate, since \[\prstart{e\in\xi}{\rho}\quad\ge\quad\prstart{X_{\tau_{\ell(e)}}=e^+}{\rho}\prstart{\tau_{e^-}>\tauatyp}{e^+}\quad\ge\quad\prstart{X_{\tau_{\ell(e)}}=e^+}{\rho}\pesc,\]
using Corollary~\ref{transience_until_nonbinary} for the last inequality.
\end{proof}

We define the truncation criterion as follows.

\begin{definition}\label{def:trunc}
Let us consider any realisation $T$ of a random quasi-tree. For a positive constant $A$ and a long-range edge $e$ in $T$, we define the corresponding \emph{truncation event} as
\[
\trunc{e}{A}:=\quad\left\{\til{W}_T(e)>\frac12(1+c^+_L)\log N+A\sqrt{\frac{\V\log N}{\h}}\right\},
\]
where $\h$ and $\V$ are as in \Cref{entropy_concentration}.
\end{definition}

We use the comparison between $W$ and $\til{W}$ to bound the probability that a walk on $T$ crosses an edge satisfying the truncation event.

\begin{lemma}\label{bound_trunc_prob}
For any $\theta\in(0,1)$ there exists $A$ such that for any realisation of $T$
\[
\prscond{\bigcup_{t\le\tau_{L^+}}\trunc{\left(X_{t-1},X_t\right)}{A},\:\Omegatyp{\tau_{L^+}},\:\Omegabound{\tau_{L^+}}}{T}{}\quad<\quad\theta.
\]
\end{lemma}

\begin{proof}
The proof is analogous to the proof of~\cite[Lemma 4.5]{random_matching}. We need that
\[
\sqrt{\frac{\V\log N}{\h}}\quad\gtrsim\quad R^4+\sqrt{L^+\V},
\]
which indeed holds since $L^+\asymp\frac{\log N}{\h}\asymp\log N$ and $R\asymp\log\log N$.
\end{proof}

In our arguments later, we will also need that for most level $L^-$ edges in $T$, the probability that a walk on $T$ reaches level $L^-$ at that particular edge is small. We define the following event that indicates that this probability is too large for a given edge. (The definition concerns level $\left(L^--K\right)$ edges as we will use this criterion with respect to a subtree of $T$ that starts at level $K$ of $T$.)

\begin{definition}\label{def:truncprime}
Let us consider any realisation $T$ of a random quasi-tree. For a long-range edge $e$ in $T$, we define the corresponding \emph{level $\left(L^--K\right)$ truncation event} as
\[
\truncprime{e}:=\quad\left\{\til{W}_T(e)<\frac12\log N+\sqrt{\log N}\right\} \cap\left\{\ell(e)=L^--K\right\}.\qedhere
\]
\end{definition}

We show that a walk on the quasi-tree is unlikely to reach level $\left(L^--K\right)$ at a long-range edge satisfying this truncation event.

\begin{lemma}\label{bound_truncprime_prob}
For any $\theta\in(0,1)$ for sufficiently large values of $C_L^-$ we have
\[
\pr{\truncprime{\left(X_{\tau_{(L^--K)}-1},X_{\tau_{(L^--K)}}\right)},\:\Omegatyp{\tau_{(L^--K)}},\:\Omegabound{\tau_{(L^--K)}}}\quad<\quad\theta.
\]
\end{lemma}

\begin{proof}
We use \Cref{compare_W_Wtil} and \Cref{entropy_concentration}.
\end{proof}

In what follows, we would like to focus on vertices of $G^*$ where we already know that a small neighbourhood around them has the structure of a quasi-tree. We give the following definition.

\begin{definition}\label{def:Kroot}
We say that a vertex $x$ in $G^*$ is a \emph{$K$-root} if in the quasi-tree $T_x(G^*)$ corresponding to $G^*$ (see \Cref{def:quasi-tree_of_Gstar}) the vertices in the first $K$ levels correspond to different vertices of $G^*$.
\end{definition}

\subsection{Nice vertices}\label{sec:nice}

In this section we define the notion of a sufficiently nice neighbourhood around a vertex in $G^*$.

\begin{definition}\label{def:That}
Given a quasi-tree $T$, a corresponding \emph{quasi-tree with blue half-edges}, denoted by $\what{T}$, is obtained by adding $\db(v)\eqdist\db$ blue half-edges to each vertex $v$ of $T$, independently of each other and everything else.

{Given the graph $G^*$ and a vertex $x$, the corresponding quasi-tree with blue half-edges $\what{T}_x(G^*)$ is obtained from $T_x(G^*)$ by adding $\db(\iota(v))$ blue half-edges to each of its vertices $v$.}
\end{definition}

\begin{definition}\label{def:nice}
We say that a vertex $x$ in $G^*$ is \emph{nice} if it satisfies both of the following properties.
\begin{enumerate}[(i)]
\item $x$ is a $K$-root.
\item $\prscond{\Omegatyp{\tau_K},\:\Omegabound{\tau_K},\:\taublue>\tau_K}{{\what{T}_{x}(G^*)}}{x}\ge\cnice$, where $\taublue$ denotes the first time that the walk crosses a blue (half-)edge, and  $\cnice\in(0,1)$ is a sufficiently small constant to be specified later.
\end{enumerate}

We denote the set of nice vertices by $\Vnice$.
\end{definition}

{Note that the event that $x$ is nice only depends on the first $K$ levels of $\what{T}_{x}(G^*)$, and in case $x$ is a $K$-root this can be identified with a neighbourhood of $x$ in $G^*$ up to $K$ levels.}

We also introduce the following modified notion of niceness, where we are allowed to ignore one red or blue edge from the given vertex. (This will be useful in \Cref{sec:hit_nice}, where we will show that a walk hits a nice neighbourhood quickly by showing that the walk crosses a blue or red edge that leads to a still unexplored region of $G^*$ that is likely to be nice, and we would like to be able to disregard that edge.)

\begin{definition}\label{def:almost_nice}
We say that a vertex $x$ in $G^*$ is \emph{almost nice}, if there exists a red or a blue edge $e$ emanating from $x$ such that the vertex $x$ satisfies the conditions of \Cref{def:nice} in the graph~$G^*\setminus\{e\}$, and the other endpoint of $e$ is not contained in the neighbourhood of $x$ up to $K$ levels in~$G^*\setminus\{e\}$. We will call such an edge $e$ a \emph{distinguished edge} corresponding to $x$.

We denote the set of almost nice vertices by $\Vniceprime$.
\end{definition}

\begin{lemma}\label{niceprime_hit_level_K}
{Let $x\in\Vniceprime$ and let $e=\{x,x^a\}$ be a corresponding distinguished edge. Let $\what{T}_{x}^a(G^*)$ be obtained by considering the quasi-tree with blue half-edges $\what{T}_{x}(G^*\setminus\{e\})$ corresponding to $G^*\setminus\{e\}$ around $x$, and adding an additional edge $\{x,x^a\}$ to a new vertex $x^a$. Then a simple random walk on $\what{T}_{x}^a(G^*)$ satisfies
\[\prscond{\Omegatyp{\tau_K},\:\Omegabound{\tau_K},\:\taublue>\tau_K,\:\tau_{x^a}>\tau_K}{\what{T}_{x}^a(G^*)}{x}\quad\ge\quad\cnice',\]
where $\cnice'$ is a positive constant depending on $\cnice$.}
\end{lemma}

{Note that the probability in \Cref{niceprime_hit_level_K} only depends on the first $K$ levels of $\what{T}_{x}^a(G^*)$.}

\begin{proof}
Let us drop the subscript $x$ and argument $G^*$ of the trees and drop the conditioning on the trees from the notation. Let us also write $\Omega(s)=\Omegatyp{s}\cap\Omegabound{s}\cap\{\taublue>s\}\cap\{\tau_{x^a}>s\}$ when considering a walk on $\what{T}^a$, and $\Omega(s)=\Omegatyp{s}\cap\Omegabound{s}\cap\{\taublue>s\}$ when considering a walk on $\what{T}$.

By considering the number of times that the walk returns to $x$, we see that
\begin{align}
\nonumber
&\prstartup{\Omega(\tau_K)}{\what{T}^a}{x}\quad=\quad\sum_{k\ge0}\prstartup{\Omega(\tau_K),\text{ exactly }k\text{ returns to }x}{\what{T}^a}{x}\\
\nonumber
&=\quad\sum_{k\ge0}\prstartup{\tau_x^+<\tau_K,\:\Omega(\tau_x^+)}{\what{T}^a}{x}^k \prstartup{\tau_x^+>\tau_K,\:\Omega(\tau_K)}{\what{T}^a}{x}\quad=\quad \frac{\prstartup{\tau_x^+>\tau_K,\:\Omega(\tau_K)}{\what{T}^a}{x}}{1-\prstartup{\tau_x^+<\tau_K,\:\Omega(\tau_x^+)}{\what{T}^a}{x}}\\
\label{eq:That_vs_T_prob}
&=\quad\frac{\frac{\deg(x)}{\deg(x)+1}\prstartup{\tau_x^+>\tau_K,\:\Omega(\tau_K)}{\what{T}}{x}}{1-\frac{\deg(x)}{\deg(x)+1}\prstartup{\tau_x^+<\tau_K,\:\Omega(\tau_x^+)}{\what{T}}{x}}\quad=\quad \frac{\prstartup{\tau_x^+>\tau_K,\:\Omega(\tau_K)}{\what{T}}{x}}{\frac{1}{\deg(x)}+1-\prstartup{\tau_x^+<\tau_K,\:\Omega(\tau_x^+)}{\what{T}}{x}}\:,
\end{align}
where $\deg(x)$ denotes the degree of $x$ in $\what{T}$.

Note that
\begin{align*}
1-\prstartup{\tau_x^+<\tau_K,\:\Omega(\tau_x^+)}{\what{T}}{x}\quad\ge\quad1-\prstartup{\tau_x^+<\tauatyp}{\what{T}}{x}\quad=\quad\prstartup{\tau_x^+>\tauatyp}{\what{T}}{x}\quad\gtrsim\quad\frac{1}{\deg(x)},
\end{align*}
where in the $\gtrsim$ we used \Cref{transience_until_nonbinary}. Then
\begin{align*}
\eqref{eq:That_vs_T_prob}\quad\asymp\quad\frac{\prstartup{\tau_x^+>\tau_K,\:\Omega(\tau_K)}{\what{T}}{x}}{1-\prstartup{\tau_x^+<\tau_K,\:\Omega(\tau_x^+)}{\what{T}}{x}}\quad=\quad\prstartup{\Omega(\tau_K)}{\what{T}}{x},
\end{align*}
which finishes the proof.
\end{proof}

\subsection{Crossing the second blue edge}

Now we are ready to state the precise result we would like to prove regarding the time $\taublue^{(2)}$ transition probabilities.

\begin{proposition}\label{transition_probs_on_Vnice}
There exist positive constants $c$ and $C$ such that with high probability $G^*$ has the following property. Let $X$ be a simple random walk on $G^*$ and let $\taublue^{(2)}$ be the second time that it crosses a blue edge. Then for all $x\in\Vniceprime$ and $y\in\Vnice$ such that the neighbourhoods of $x$ and $y$ up to the first $K$ levels are disjoint and do not contain the other endpoint of the distinguished edge $e_x$ associated to $x$, we have
\[\prscond{X_{\taublue^{(2)}-1}=y,\:\taublue^{(2)}\le C\log N}{G^*}{x}\quad\ge\quad c\frac{\log N}{N}\db(y)\:.\]
\end{proposition}

\subsection{Coupling}\label{sec:coupling}

In this section we present an exploration process revealing a random neighbourhood around two vertices in $G^*$ and coupling them to two independent random quasi-trees. Then we also present a coupling between two independent walks on these neighbourhoods of $G^*$ and two independent walks on the two independent quasi-trees.

{Consider two vertices $x$ and $y$ of $G^*$, and a red or blue neighbour $x^a$ of $x$ such that the following properties hold. The vertex $x$ is a $K$-root in the graph $G^*\setminus\left\{\{x,x^a\}\right\}$, while $y$ is a $K$-root in $G^*$, their neighbourhoods up to the first $K$ levels in $G^*\setminus \{x,x^a\}$ and $G^*$ respectively are disjoint from each other, do not contain $x^a$, and take values $\what{T}_{x,K}(G^*\setminus\{e\})=\what{T}_{x,K}$ and $\what{T}_{y,K}(G^*)=\what{T}_{y,K}$ respectively. Let $\Ggr^a=\Ggr\cup\left\{\{x,x^a\}\right\}$.}

Let $\what{T}_x$ and $\what{T}_y$ be independent random quasi-trees with blue half-edges, conditioned to have their first $K$ levels equal to $\what{T}_{x,K}$ and $\what{T}_{y,K}$ respectively. Let $\what{T}_x^a$ be obtained by adding a long-range edge from the root $x$ of $\what{T}_x$ to a new vertex $x^a$. Say that $x^a$ is at level -1. Note that in $\what{T}_x^a$ we view $\{x,x^a\}$ as a long-range edge and we do not treat it as a blue edge even if it corresponds to a blue edge in~$G^*$.

\subsubsection*{Exploration in $G^*$, coupling with $\what{T}_x^a$ and $\what{T}_y$}

Firstly, we define an exploration process that reveals some random neighbourhoods of $x$ and $y$ in $G^*$, and couples them with $\what{T}_x^a$ and $\what{T}_y$, respectively.

We denote the centres of the level $K$ balls of $\what{T}^a_x$ by $z_1,...,z_{L_x}$ and the centres of the level $K$ balls of $\what{T}_y$ by $z_{L_x+1},...,z_{L_x+L_y}$. We start by exploring $\Ggr$ from the ball of $z_1$, level by level up to level $L^+$ as follows.

We know (by definition) that the $(R+R)$-ball around $z_1$ in $T^a_x$ agrees with the one in $\Ggr$. We consider the long-range half-edges from this ball one by one, and for each long-range half-edge $e$ we do the following. If the truncation event $\trunc{e}{A}$ holds with respect to the tree $T(z_1)$ rooted at $z_1$, then we truncate $e$, and we do not explore further in this direction. Otherwise, we choose ($\iota$ of) its other endpoint in $T^a_x$, and we choose its other endpoint in $\Ggr$ with the correct distributions (the distribution in $\Ggr$ depends on what we have revealed so far), and we couple them via their optimal coupling. If this coupling fails, we truncate $e$, and we do not explore further in this direction. If the coupling of the edge is successful, then we reveal the $(R+R)$-ball around its new endpoint in $T^a_x$ and in $\Ggr$ (including the red degrees from its vertices) using the optimal coupling, revealing the ball edge by edge, stopping if the optimal coupling fails or if we find that the ball is atypical in $T^a_x$ (e.g.\ by reaching a too large volume) or if it intersects with the already explored region (we call this an `overlap'). If the optimal coupling fails or the ball is atypical or there is an overlap, then we truncate $e$, and we do not explore further in this direction. After we have considered all long-range half-edges from the ball of $z_1$, we continue analogously by considering the long-range half-edges in the level $(K+1)$ balls (that are connected to the ball of $z_1$ via a non-truncated edge), and then we continue analogously level by level. At later levels we also truncate half-edges $e$ that satisfy $\truncprime{e}$ with respect to the tree $T(z_1)$.

So to summarise, we reveal and couple part of $G^*$ and $T(z_1)$, proceeding level by level up to level $L^+$, and \emph{truncating} a long-range (half-)edge $e$ if any of the following hold:
\begin{enumerate}[(i)]
\item it satisfies $\trunc{\cdot}{A}$ or $\truncprime{\cdot}$ with respect to the tree $T(z_1)$ rooted at $z_1$ \footnote{Here $A$ is as in \Cref{bound_trunc_prob} for a choice of $\theta$ to be specified later.};
\item the optimal coupling of the other endpoint fails;
\item the $(R+R)$-ball around the other endpoint is atypical according to \Cref{def:atypical_ball};
\item the optimal coupling of the $(R+R)$-ball around the other endpoint fails (this coupling involves the green edges in the $(R+R)$-balls, as well as the red degree of each vertex in the ball); or
\item\label{item:overlap} the $(R+R)$-ball around the other endpoint contains any of the previously revealed vertices.
\end{enumerate}

We write $\F_1$ for the $\sigma$-algebra generated by this exploration and the $\sigma$-algebras of $\widehat{T}_{x,K}^a(G^*\setminus\{x,x^a\})$ and~$\widehat{T}_{y,K}(G^*)$.

Now we go through the vertices we explored above, level by level, and couple the blue degree of each vertex in $G^*$ with the blue degree in $\what{T}_x$, using the optimal coupling. If in an $(R+R)$-ball at least one of these optimal couplings fails, then we say that the long-range edge leading to that $(R+R)$-ball is \emph{blue-truncated}.

We write $\what{\F}_1$ for the $\sigma$-algebra generated by $\F_1$ and this additional exploration of blue degrees.

Then we run a similar exploration and coupling from each of $z_2$, $z_3$, ..., $z_{L_x+L_y}$. (When exploring from $z_i$, we consider $\trunc{\cdot}{A}$ and $\truncprime{\cdot}$ with respect to $T(z_i)$. In \eqref{item:overlap} we also consider previously revealed vertices from the previous trees $T(z_j)$.) We write $\F_i$ for the $\sigma$-algebra generated by $\what{\F}_{i-1}$ and the exploration of $\Ggr$ from $z_i$, and write $\what{\F}_i$ for the $\sigma$-algebra generated by $\F_i$ and the exploration of blue degrees from $z_i$.

We say $z_i$ is \emph{good} if its $(R+R)$-ball does not intersect any of the $(R+R)$-balls revealed during the explorations from $z_1$, ..., $z_{i-1}$. Otherwise we call $z_i$ \emph{bad}.

\subsubsection*{Coupling of walks on $\Ggr^a$ with walks on $T_x^a$ and $T_y$:}

Now we construct a coupling between random walks on $\Ggr^a$ from $x$ and $y$, and random walks on $T_x^a$ and $T_y$ from $x$ and $y$, respectively. ($T_x^a$, $T_y$, etc are the versions of $\what{T}_x^a$, $\what{T}_y$, etc without the blue half-edges.)

Let $Z^{(1)}$ and $Z^{(2)}$ be independent random walks on $T_x^a$ and $T_y$ respectively (starting from $x$ and $y$ respectively) terminated at their corresponding $\tauatyp$. Let $X^{(1)}$ and $X^{(2)}$ be independent random walks on $\Ggr\cup\{\{x,x^a\}\}$ from $x$ and $y$, respectively so that $X^{(i)}$ agrees with $Z^{(i)}$ up to the first time that $Z^{(i)}$ crosses a truncated edge, hits the boundary of an $(R+R)$-ball, reaches level $L^+$, gets terminated because of visiting an atypical $(R+R)$-ball, or in case of $Z^{(1)}$ hits the vertex $x^a$. (It is not difficult to see that such a coupling is possible.)

We say that the above coupling is \emph{successful} if the following are all satisfied for $i=1,2$.

\begin{enumerate}[(i)]
\item \label{item:hit_K}
$Z^{(i)}$ reaches level $K$, in case of $Z^{(1)}$ without hitting $x^a$.
\item \label{item:not_backtrack_from_K}
After reaching level $K$, the walk $Z^{(i)}$ does not revisit level $(K-1)$.
\item \label{item:Lplus_without_crossing_truncated}
$Z^{(i)}$ reaches level $L^+$ without crossing a truncated edge.
\item \label{item:not_backtrack_from_Lminus}
After reaching level $L^-$, the walk $Z^{(i)}$ does not revisit level $L^--M$.
\item \label{item:Omegabound_holds}
$\Omegabound{\tau_{L^+}}$ holds, i.e.\ until $Z^{(i)}$ first hits level $L^+$, it does not visit $R$ or more different vertices in any one $(R+R)$-ball.
\item \label{item:tauLplus_bound}
We have $\tau_{L^+}\le \frac{2}{\nu}L^+$.
\item \label{item:enough_small_degrees}
At least $\frac12\left(L^+-L^-\right)$ vertices visited by $Z^{(i)}$ between levels $L^-$ and $L^+$ before hitting level $L^+$ have $\dgr(\cdot)\le\Cdeg$. Here $\Cdeg$ is a sufficiently large constant, to be specified later.
\end{enumerate}

Let $\Succ$ denote the event that the coupling is successful.

\subsubsection*{Coupling of walks on $G^*$ with walks on $\what{T}_x^a$ and $\what{T}_y$:}

Finally, we construct a coupling between random walks on $G^*$ from $x$ and $y$, and random walks on $\what{T}_x^a$ and $\what{T}_y$ from $x$ and $y$, respectively.

Let $\what{Z}^{(1)}$ and $\what{Z}^{(2)}$ be independent walks on $\what{T}_x^a$ and $\what{T}_y$ respectively, defined up to the first time that they cross a blue half-edge or reach their corresponding $\tauatyp$, and agreeing with $Z^{(1)}$ and $Z^{(2)}$ up to this time. Let $\what{X}^{(1)}$ and $\what{X}^{(2)}$ be independent random walks on $G^*$ from $x$ and $y$ respectively so that $\what{X}^{(i)}$ agrees with $\what{Z}^{(i)}$ up to the first time that $\what{Z}^{(i)}$ crosses a blue half-edge, crosses a truncated or blue-truncated edge, hits the boundary of an $(R+R)$-ball, reaches level $L^+$, gets terminated because of visiting an atypical $(R+R)$-ball, or in case of $\what{Z}^{(1)}$ hits the vertex $x^a$.

We say that this coupling is \emph{successful} if the following are all satisfied for $i=1,2$.
\begin{enumerate}[(i)]
\item\label{item:Succ_holds}
$\Succ$ holds.
\item\label{item:no_blue_before_K}
$\what{Z}^{(i)}$ does not cross a blue half-edge before reaching level $K$.
\item\label{item:no_bluetrunc_before_Lplus}
${Z}^{(i)}$ does not cross a blue-truncated edge before reaching level $L^+$.
\item\label{item:no_blue_between_K_Lminus}
$\what{Z}^{(i)}$ does not cross a blue half-edge between levels $K$ and $L^-$.
\item\label{item:blue_between_Lminus_Lplus}
Before $\tau_{L^+}$\footnote{Here $\tau_{L^+} $ is the hitting time of level $L^{+}$ by the walk $\widehat{Z}^{(i)}$.}, the walk $\what{Z}^{(i)}$ crosses a blue half-edge between levels $L^-$ and $L^+$.
\end{enumerate}

Let $\Succhat$ denote the event that this coupling is successful.

\subsection{Estimates regarding the success probability}

In this section we estimate the probability that the above couplings succeed.

Let us consider an almost nice vertex $x$ with distinguished edge $\{x,x^a\}$, and a nice vertex $y$ such that the neighbourhood of $x$ in $G^*\setminus\{\{x,x^a\}\}$ up to $K$ levels, and the neighbourhood of $y$ in $G^*$ up to $K$ levels are disjoint, the latter one does not contain $x^a$, and the two neighbourhoods {take values $\what{T}_{x,K}(G^*\setminus\{e\})\cup\{\{x,x^a\}\}=\what{T}^a_{x,K}$ and $\what{T}_{y,K}(G^*)=\what{T}_{y,K}$ respectively}. (In what follows we will abbreviate this event as $\left\{\what{T}^a_{x,K},\what{T}_{y,K}\right\}$ and the analogous event on $\Ggr^a$ as $\left\{T^a_{x,K},T_{y,K}\right\}$.)

For the rest of this section, the estimates on the probabilities are uniform over all vertices $x$, $y$ and possible realisations of $\what{T}^a_{x,K},\what{T}_{y,K}$ satisfying the above niceness and disjointness conditions. 

The main results of this section are the following.

\begin{proposition}\label{Succhat_prob}
There exists $\theta\in(0,1)$ such that we have
\[
\prscond{\prscond{\Succhat}{\what{\F}_{L_x+L_y}}{x,y}>\theta}{\left\{\what{T}^a_{x,K},\what{T}_{y,K}\right\}}{}\quad\ge\quad1-o\left(\frac{1}{N^2}\right).
\]
\end{proposition}

\begin{proposition}\label{bound_prob_specific_blue_edge}
For each blue half-edge $e$ in $\what{T}_x$ we have
\[
\prscond{\what{Z}^{(1)}_{\taublue}=e,\:\what{Z}^{(1)}\text{ satisfies the conditions of }\Succhat}{\what{\F}_{L_x+L_y}}{x}\quad\le\quad\frac{1}{\sqrt{N\log N}}.
\]
We also have an analogous bound for $\what{Z}^{(2)}$ and $\what{T}_y$.
\end{proposition}

The proof of \Cref{bound_prob_specific_blue_edge} is short and presented below.

\begin{proof}[Proof of \Cref{bound_prob_specific_blue_edge}]
Let $f$ be the level $L^--M$ ancestor of $e$. Since a walk satisfying the conditions of $\Succhat$ hits level $L^-$ at a non-truncated edge, and it does not backtrack from level $L^-$ to level $L^--M$, we can bound the required probability by the probability of $Z^{(1)}$ hitting level $L^-$ at a non-truncated descendant of $f$. There are $\le (e^{\Cvol R}\dmax)^{M+1}\le e^{C(\log\log N)^2}$ non-truncated descendants of $f$ at level $L^-$ (where $C$ is a large constant), and by the definition of $\truncprime{\cdot}$ the probability of hitting level $L^-$ at a specific one of them is $\le\frac{1}{\sqrt N}e^{-\sqrt{\log N}}$. This gives the required bound.
\end{proof}

The proof of \Cref{Succhat_prob} is more involved. In what follows, we establish a series of auxiliary results, and we conclude the proof at the end of this section.

Firstly, we show that most vertices $z_i$ at level $K$ of $T_{x,K}^a$ and $T_{y,K}$ are good.

\begin{lemma}\label{bound_num_explored_num_bad}
For all sufficiently large values of $n$, the total number of vertices explored during the explorations from $z_1$, ..., $z_{L_x+L_y}$ is
\[
\le\quad N^{\frac12\left(1+c_L^+\right)}\exp\left(2A\sqrt{\frac{\V\log N}{\h}}\right).
\]
Also, there exists a positive constant~$C$ (not depending on $\what{T}_{x,K}$ and $\what{T}_{y,K}$) so that the number $\rm{Bad}$ of bad vertices $z_i$ satisfies
\[
\prscond{\rm{Bad}\ge C}{\left\{\what{T}^a_{x,K},\what{T}_{y,K}\right\}}{}\quad=\quad o\left(\frac{1}{N^2}\right).
\]
\end{lemma}

\begin{proof}
By the definition of $\trunc{\cdot}{A}$, the number of vertices at each level of a given tree $T(z_i)$ that are not truncated due to the truncation criterion $\trunc{\cdot}{A}$ is $\le N^{\frac12(1+c_L^+)}\exp\left(A\sqrt{\frac{\V\log N}{\h}}\right)e^{\Cvol 2R}$. The number of levels is $\asymp\log N$ and the number of $z_i$ is $\le e^{\Cvol 2R{(K+1)}}{\dmax^{(K+1)}}\le e^{\const\cdot(\log\log N)^2}$. (Here $\Cvol$ is the constant from \Cref{def:atypical_ball}\eqref{item:large_vol}.) Using that $\sqrt{\frac{\V\log N}{\h}}\gg(\log\log N)^2$ we get the bound on the number of explored vertices.

In a random quasi-tree, the probability of the new endpoint of a newly revealed long-range edge taking any specific value is $\le\max_{x,y}\pred_{x,y}\lesssim\frac{1}{N^{1-\beta}\log N}$. Hence each time we reveal a new $(R+R)$-ball, the probability that it intersects the $(R+R)$-ball of some  $z_i$ is $\le\frac{e^{\Cvol2R(K+1)}\dmax^{K+1}e^{\Cvol 4R}}{N^{1-\beta}\log N}\le \frac{e^{\const( \log\log N)^2}}{N^{1-\beta}\log N}$. The probability that it happens at least $C$ times is at most
\[
\exp\left(C\left(\frac12\left(1+c_L^+\right)\log N+2A\sqrt{\frac{\V\log N}{\h}}+\const (\log \log N)^2-(1-\beta)\log N-\log\log N\right)\right).
\]
For sufficiently small $c_L^+$ such that $\frac12\left(1+c_L^+\right)<(1-\beta)$ and for sufficiently large $C$ this is $o\left(\frac{1}{N^2}\right)$. The inequality $\frac12\left(1+c_L^+\right)<(1-\beta)$ is true for sufficiently small  $c_L^+$, as long as $\beta<\frac12$. We note that this is the only place where a nontrivial constraint on the value of $\beta$ appears.
\end{proof}

Next, we show that for a good vertex $z_i$ at level $L$ of $T_x$, a walk $Z$ on $T_x(z_i)$ with positive probability satisfies (iii)-(vii) from the definition of $\Succ$ in \Cref{sec:coupling}. Then we show that a walk $\what{Z}$ on $\what{T}_x$ from $z_i$ with positive probability satisfies (ii)-(vii) from the definition of $\Succ$ and (iii)-(v) from the definition of~$\Succhat$.

\begin{lemma}\label{Succ_i_prob_on_Tzi}
For any constant $\theta\in(0,1)$ the following holds. Let $i\in\{1,2,...,L_x\}$ and let $Z$ be a simple random walk on $T_x(z_i)$ starting from $z_i$. Let $\Succ^{T(z_i)}$ denote the event that $Z$ satisfies (iii)-(vii) from the definition of $\Succ$ in \Cref{sec:coupling}. Then we have
\[
\prscond{\Succ^{T(z_i)}}{\what{\F}_{i-1}}{z_i}\quad\ge\quad(1-\theta)\1{\{z_i\text{ good}\}}.
\]
An analogous bound holds for $i\in\{L_x+1,...,L_x+L_y\}$ and a walk on $T_y(z_i)$.
\end{lemma}

\begin{proof}
The proof is similar to the proof of~\cite[Lemma 5.6]{random_matching}. Let us work on the event that $z_i$ is good.

Firstly, we show that the probability of $\left\{\Omegatyp{\tau_{L^+}},\Omegabound{\tau_{L^+}},\tau_{L^+}\le\frac{2}{\nu}L^+\right\}=:\A$ can be lower bounded by a constant arbitrarily close to 1.
\begin{itemize}
\item By \Cref{speed_concentration} we can bound $\pr{\tau_{L^+}>\frac{2}{\nu}L^+,\:\Omegatyp{\sigma_{L^+}},\:\Omegabound{\sigma_{L^+}}}$ by an arbitrarily small constant.
\item Using \Cref{bound_sigma1} and that by \Cref{iid_decomp} we have $\sigma_{k}\wedge\tauatyp\wedge\taubound\stle\sum_{j=1}^{k}s_j$ where $s_j$ are independent, $s_1\eqdist\sigma_{1}\wedge\tauatyp\wedge\taubound$, and $s_j\eqdist\sigma'_{1}\wedge\tauatyp'\wedge\taubound'$ for $j\ge2$, we get that\\
$\pr{\sigma_{L^+}>(\log N)^{3/2},\:\Omegatyp{(\log N)^{3/2}},\:\Omegabound{(\log N)^{3/2}}}\ll1$.
\item From \Cref{bound_atyp_prob} we know that for each newly sampled $(R+R)$-ball, the probability that it is atypical is $\lesssim\left(\frac{1}{\log N}\right)^2$. Hence the probability of any of the first $(\log N)^{3/2}$ different balls visited by $Z$ being atypical is $o(1)$, therefore $\pr{\Omegatyp{(\log N)^{3/2}}^c}\ll1$.
\item By \eqref{item:hit_boundary} in the definition of atypical balls (\Cref{def:atypical_ball}) we know that starting a walk in a typical ball, the probability that it visits at least $R$ vertices in a ball before $\tauatyp$ is $\le\left(\frac{1}{\log N}\right)^2$. Hence $\pr{\Omegabound{(\log N)^{3/2}}^c,\:\Omegatyp{(\log N)^{3/2}}}\ll1$.
\end{itemize}
Combining these, we can upper bound $\pr{\A^c}$ by an arbitrarily small constant.

Now we bound the probability that the event $\A$ holds, but $Z$ crosses a truncated edge by time~$\tau_{L^+}$.
\begin{itemize}
\item By \Cref{bound_trunc_prob} and \Cref{bound_truncprime_prob} we can bound the probability that $\A$ holds but $Z$ crosses an edge satisfying $\trunc{\cdot}{A}$ or $\truncprime{\cdot}$ by level $L^+$ by an arbitrarily small constant.
\item Let us consider a vertex $v$ in $T_x(z_i)$ that is the newly revealed endpoint of a long-range edge and whose $(R+R)$-ball we are just about to reveal. Let $\D$ be the set of already explored vertices. Then we can explore the ball of $v$ as follows. Consider $v$ and for each $w$ in $G$ let us sample whether $v\simg w$ holds in $T_x(z_i)$ and in $G^*$, using the appropriate distribution in each graph, coupled via the optimal coupling. If $v\not\simg w$ for all $w\in\D$ then the optimal coupling is successful. Then consider the green and black neighbours of $v$ one by one, and for each of these vertices and each $w$ in $G$ sample whether there is a green edge between them. Continue similarly, until the whole of the $(R+R)$-ball is revealed or until one of the couplings fails. Note that on the event that the $(R+R)$-ball in $G$ does not intersect $\D$, the coupling is successful.
\item Similarly, we can see that given an $(R+R)$-ball in $T_x(z_i)$, where we are just about to reveal the red degrees, the probability that the optimal coupling of these red degrees in $T_x(z_i)$ and $G^*$ fails is bounded by the probability that there is a red edge in $G^*$ leading from the ball to the set $\D$ of already explored vertices.
\item Now consider a vertex $v$ in $T_x(z_i)$ and a long-range half-edge from it whose other endpoint we are just about to reveal. Let $\D$ be the set of already revealed vertices. By the above, the probability that the long-range edge gets truncated due to the optimal coupling of the endpoint failing, due to the optimal coupling of the ball around the other endpoint failing, or due to the ball around the other endpoint containing an already revealed vertex can be bounded by the probability that a randomly sampled red edge from $v$ leads to a vertex that is in $\D$ or within distance $2R$ from $\D$ in $G^*$.
\item The probability of a red edge leading to any specific vertex is $\lesssim\frac{1}{N^{1-\beta}\log N}$. From \Cref{bound_num_explored_num_bad} we know that $|\D|\le N^{\frac12\left(1+c_L^+\right)}\exp\left(2A\sqrt{\frac{\V\log N}{\h}}\right)$. From \Cref{volume_high_prob_bound} we know that whp any ball of radius $2R$ in $G^*$ has volume $\lesssim e^{CR}$ where $C$ is a positive constant. So the probability of a newly revealed long-range edge getting truncated due to an overlap or due to the optimal coupling failing is $\lesssim N^{\frac12\left(1+c_L^+\right)}\exp\left(2A\sqrt{\frac{\V\log N}{\h}}\right)\exp\left(CC_R\log\log N\right)\frac{1}{N^{1-\beta}\log N}\log N$, where the final $\log N$ term comes from \Cref{cond_green_red_deg} and from the fact that whp all vertices in our random graph have degree $\le \log N$. Hence, the probability that the walk $Z$ up to time $\frac{2}{\nu}L^+$ crosses a long-range edge that is truncated, and the first such edge is truncated due to an overlap or due to the optimal coupling failing, is $o(1)$. 
\end{itemize}
This bounds the probability that \eqref{item:Lplus_without_crossing_truncated} fails while \eqref{item:Omegabound_holds} and \eqref{item:tauLplus_bound} hold.

Finally, we can bound the probability of \eqref{item:not_backtrack_from_Lminus} or \eqref{item:enough_small_degrees} failing as follows.
\begin{itemize}
\item The probability of $Z$ backtracking $M$ levels from level $L^-$ is $\le(1-\pesc)^M\ll1$.
\item Every time a new $(R+R)$-ball is revealed, its centre has $\dgr(\cdot)$ distributed as $\dgr$, independently of everything else revealed before. From \Cref{degree_exp_moment} we know that this has a bounded expectation and variance. For each level between $L^-$ and $L^+$, consider the first vertex at that level visited by the walk. Their $\dgr(\cdot)$ degrees are independent with the above distribution, so for a sufficiently large constant $\Cdeg$, with probability $1-o(1)$, at least half of them have $\dgr(\cdot)\le\Cdeg$. \qedhere
\end{itemize}
\end{proof}

We can deduce a similar statement for a walk $Z$ on $T_x$ rather than just $T_x(z_i)$.

\begin{corollary}\label{Succ_i_prob}
There exists $\theta\in(0,1)$ such that the following holds. Let $i\in\{1,2,...,L_x\}$ and let $Z$ be a simple random walk on $T_x$ starting from $z_i$. Let $\Succ^{i}$ denote the event that $Z$ satisfies (ii)-(vii) from the definition of $\Succ$ in \Cref{sec:coupling}. Then we have
\[
\prscond{\Succ^i}{\what{\F}_{i-1}}{z_i}\quad\ge\quad\theta\1{\{z_i\text{ good}\}}.
\]
An analogous bound holds for $i\in\{L_x+1,...,L_x+L_y\}$ and a walk on $T_y$.
\end{corollary}

\begin{proof}We can couple $Z$ to a walk on $T_x(z_i)$ started from $z_i$ such that on the event that $Z$ does not visit level $(K-1)$ up to time $\tauatyp\wedge\tau_{L^+}$, the two walks agree up to $\tauatyp\wedge\tau_{L^+}$.

By \Cref{transience_until_nonbinary} this event has probability at least $\pesc$. Then use \Cref{Succ_i_prob_on_Tzi} for the walk on $T_x(z_i)$.
\end{proof}

\begin{lemma}\label{Succhat_i_prob}
There exists $\theta\in(0,1)$ such that the following holds. Let $i\in\{1,2,...,L_x\}$ and let $Z$ and $\Succ^i$ be defined as in \Cref{Succ_i_prob}. Let $\what{Z}$ be a simple random walk on $\what{T}_x$ starting from $z_i$, defined up to the first time it hits level $L^+$ or crosses a blue half-edge, coupled to $Z$ such that they agree up to this time. Let $\Succhat^{i}$ denote the event that $\Succ^i$ holds and $Z$ and $\what{Z}$ satisfy (iii)-(v) from the definition of $\Succhat$ in \Cref{sec:coupling}. Then we have
\[
\prscond{\Succhat^i}{\what{\F}_{i-1}}{z_i}\quad\ge\quad\theta\1{\{z_i\text{ good}\}}.
\]
An analogous bound holds for $i\in\{L_x+1,...,L_x+L_y\}$ and a walk on $\what{T}_y$.
\end{lemma}

\begin{proof}
Let us work on the event that $z_i$ is good. From \Cref{Succ_i_prob} we know that $\Succ^i$ holds with positive probability. In what follows let us condition on $(Z_t)_{t\le\tau_{L^+}}$ and work on the event that it satisfies $\Succ^i$.

Firstly, we bound the probability that \eqref{item:no_bluetrunc_before_Lplus} in the definition of $\Succhat$ fails.

Let us consider a vertex $v$ whose blue degree we are just about to reveal. If the set of vertices with already revealed blue degrees is denoted by $\D$, then the probability that the optimal coupling for the blue degree of $v$ fails can be bounded as $\pr{\Bin{\log N\sum_{v\in D}\db(v)}{\pblue}>0}\le\frac{|\D|}{N\log N}(\log N)^2$ using \Cref{cond_green_red_deg} and that with high probability the graph does not have any vertex with degree greater than $\log N$. 

Using that by \Cref{bound_num_explored_num_bad} the overall number of explored vertices is $\le N^{\frac12(1+c_L^++o(1))}$ and that by definition a typical ball has volume $\le e^{\Cvol C_R\log\log N}$, we get that the probability of a given long-range edge leading to a typical ball, but being blue-truncated is $\ll\frac{1}{N^{\const}}$. If we condition on $(Z_t)_{t\le\tau_{L^+}}$ and work on the event that it satisfies $\Succ^i$, and we reveal the blue degrees in the balls as they are first visited by $Z$, we get that the probability of $(Z_t)_{t\le\tau_{L^+}}$ crossing a blue-truncated long-range edge is $\ll\frac{1}{\nu}L^+\frac{1}{\log N}\asymp1$.

Now we turn to lower bounding the probability that \eqref{item:no_blue_between_K_Lminus} in the definition of $\Succhat$ holds.

By assumption the walk $(Z_t)_{t\le\tau_{L^+}}$ visits $\le\frac{2}{\nu}L^+\asymp\log N$ vertices before level $L^-$, counted with multiplicity. Denote the set of these vertices by $\Scal$. Let $\Scal_{\mathrm{many}}$ be the set of vertices before  level $L^-$ that $(Z_t)_{t\le\tau_{L^+}}$ visits $\ge C$ times. Then $|\Scal_{\mathrm{many}}|\le\frac{\frac{2}{\nu}L^+}{C}$. The sum of the blue degrees of these vertices in $\what{T}$ is $\stle\Bin{N|\Scal_{\mathrm{many}}|}{\pblue}$, where $\pblue\asymp\frac{1}{N\log N}$. Hence the probability of any vertex in $\Scal_{\mathrm{many}}$ having at least one blue half-edge tends to 0 as $C\to\infty$.

The sum over vertices in $\Scal$ of the blue degrees in $\what{T}$ is $\stle\Bin{N\frac{2}{\nu}L^+}{\pblue}$, hence the probability that this is $\ge C$ tends to 0 as $C\to\infty$.

So with positive probability there is no blue half-edge from the set $\Scal_{\mathrm{many}}$ and the total number of blue half-edges from $\Scal$ is $\le C$. On this event, the probability that $(\what{Z}_t)_{t\le\tau_{L^+}}$ does not cross a blue half-edge below level $L^-$ is bounded away from 0.

Finally, we lower bound the probability that \eqref{item:blue_between_Lminus_Lplus} in the definition of $\Succhat$ holds.

By the definition of $\Succ^i$ we know that at least $\frac12(L^+-L^-)\asymp\log N$ vertices visited by $Z$ between levels $L^-$ and $L^+$ have $\dgr(\cdot)\le\Cdeg$. For a given vertex $v$ we have $\pr{\db(v)>0}\asymp\frac{1}{\log N}$, so with positive probability at least one of the above vertices has at least one blue half-edge. On this event the probability of $(\what{Z}_t)_{t\le\tau_{L^+}}$ crossing a blue half-edge between levels $L^-$ and $L^+$ is bounded away from 0.
\end{proof}

Now we are ready to conclude the proof of \Cref{Succhat_prob}.

\begin{proof}[Proof of \Cref{Succhat_prob}]
The proof uses a martingale argument similar to the one in the proof of \cite[Proposition 5.7]{random_matching}.

Conditional on $\what{\F}_{L_x+L_y}$, the walks $\what{Z}^{(1)}$ and $\what{Z}^{(2)}$ are independent, so it is sufficient to prove that for some positive constant $\theta$ we have
\begin{align}\label{eq:Succhat_prob_one_walk}
\prscond{\prscond{\what{Z}^{(j)}\text{ satisfies the conditions of } \Succhat}{\what{\F}_{L_x+L_y}}{}>\theta}{\left\{\what{T}^a_{x,K},\what{T}_{y,K}\right\}}{}\quad\ge\quad1-o\left(\frac{1}{N^2}\right)
\end{align}
for $j=1,2$. We present the proof for $\what{Z}^{(1)}$ below, the proof for $\what{Z}^{(2)}$ follows analogously.

For each $i\in\{1,2,...,L_x\}$ let $h(z_i)=\prscond{\what{Z}^{(1)}_{\tau_K}=z_i,\:\tau_K<\tauatyp\wedge\taubound\wedge\taublue\wedge\tau_{\rho^a}}{\left\{\what{T}^a_{x,K},\what{T}_{y,K}\right\}}{x}$. \footnote{Recall that $\taubound$ denotes the first time that the walk visits the $R$th different vertex in an $(R+R)$-ball, i.e.\ the first time that $\Omegabound{\cdot}$ fails.} Let
\begin{align*}
\what{V}:=\quad\left\{z_i:\:i\in\{1,...,L_x\},\:\prscond{\left(\Succhat^i\right)^c}{\what{\F}_{L_x+L_y}}{z_i}\ge1-\theta_1\right\},
\end{align*}
where $\theta_1\in(0,1)$ is a constant that is smaller than the constant in \Cref{Succhat_i_prob}.

We will prove that
\begin{align}\label{eq:bound_hVhat}
\prscond{h\left(\what{V}\right)\ge\theta_4+(1-\theta_5)\sum_{i=1}^{L_x}h(z_i)}{\left\{\what{T}^a_{x,K},\what{T}_{y,K}\right\}}{}\quad\ll\quad\frac{1}{N^2},
\end{align}
for some given constant $\theta_5\in(0,1)$ and any constant $\theta_4\in(0,1)$.

Once we have this, we can conclude the proof of \eqref{eq:Succhat_prob_one_walk} as follows. Note that
\begin{align*}
&\prscond{\what{Z}^{(1)}\text{ fails the conditions of }\Succhat}{\what{\F}_{L_x+L_y}}{}\\
&\le\quad\left(1-\sum_{i}h(z_i)\right)+ \sum_{z_i\in\what{V}}h(z_i)\prscond{\left(\Succhat^i\right)^c}{\what{\F}_{L_x+L_y}}{z_i} +\sum_{z_i\not\in\what{V}}h(z_i)\prscond{\left(\Succhat^i\right)^c}{\what{\F}_{L_x+L_y}}{z_i}\\
&\le\quad \left(1-\sum_{i}h(z_i)\right)+h\left(\what{V}\right)+h\left(\what{V}^c\right)(1-\theta_1)\quad =\quad1-h\left(\what{V}^c\right)\theta_1.
\end{align*}

By \eqref{eq:bound_hVhat}, with probability $1-o\left(\frac{1}{N^2}\right)$ we have
\begin{align*}
1-h\left(\what{V}^c\right)\quad=\quad\left(1-\sum_{i}h(z_i)\right)+h\left(\what{V}\right)\quad\le\quad \left(1-\sum_{i}h(z_i)\right)+\theta_4+(1-\theta_5)\sum_{i}h(z_i)&\\
\le\quad1+\theta_4-\theta_5\sum_{i}h(z_i)\quad\le\quad1+\theta_4-\theta_5\cnice'\:,&
\end{align*}
where in the last $\le$ we used \Cref{niceprime_hit_level_K}. Choosing $\theta_4$ sufficiently small we can make $1+\theta_4-\theta_5\cnice'$ bounded away from 1, hence making $1-h\left(\what{V}^c\right)\theta_1$ bounded away from 1. This finishes the proof of \eqref{eq:Succhat_prob_one_walk}.

Now we present the proof of \eqref{eq:bound_hVhat}.

By \Cref{bound_num_explored_num_bad} with probability $1-o\left(\frac{1}{N^2}\right)$ we have $\#\left\{i:\:z_i\text{ bad}\right\}\lesssim1$. Also, by \Cref{transience_until_nonbinary} we have $h(z_i)\le(1-\pesc)^K\ll1$ for each $i$. So it remains to bound $\sum_{i}h(z_i)\1{z_i\in\what{V}}\1{\left\{z_i\text{ good}\right\}}$.

Using \Cref{Succhat_i_prob} and Markov's inequality, we get that
\begin{align*}
\econd{h(z_i)\1{\{z_i\in\what{V}\}}\1{\{z_i\text{ good}\}}}{\what{\F}_{i-1}}\quad=\quad h(z_i)\1{\{z_i\text{ good}\}}\prcond{\prscond{\left(\Succhat^i\right)^c}{\what{\F}_{L_x+L_y}}{z_i}\ge1-\theta_1}{\what{\F}_{i-1}}\\
\le\quad h(z_i)\1{\{z_i\text{ good}\}}\frac{1-\theta_3}{1-\theta_1},
\end{align*}
where $\theta_3\in(0,1)$ is the constant from \Cref{Succhat_i_prob}. We chose $\theta_1<\theta_3$, hence $\frac{1-\theta_3}{1-\theta_1}<1$.

Let
\[
R_i:=\quad h(z_i)\1{\{z_i\in\what{V}\}}\1{\{z_i\text{ good}\}},\qquad M_k:=\quad\sum_{i=1}^{k}\left(R_i-\econd{R_i}{\what{\F}_{i-1}}\right).
\]
Then $(M_k)$ is a martingale with respect to $\left(\what{\F}_{k}\right)$, and $|M_k-M_{k-1}|\le h(z_k)$. Also \[
M_{L_x}\quad\ge\quad\sum_{i}R_i-\frac{1-\theta_3}{1-\theta_1}\sum_ih(z_i).
\]

Then using the Azuma-Hoeffding inequality and that $h(z_i)\le(1-\pesc)^K$ for all $i$, we get that for sufficiently large values of $C_K$, for any $\theta_4>0$ we have
\begin{align*}
\prscond{\sum_{i}h(z_i)\1{\{z_i\in\what{V}\}}\1{\{z_i\text{ good}\}}>\theta_4+\frac{1-\theta_3}{1-\theta_1}\sum_ih(z_i)}{\left\{\what{T}^a_{x,K},\what{T}_{y,K}\right\}}{}\quad\le\quad\prscond{M_{L_x}>\theta_4}{\left\{\what{T}^a_{x,K},\what{T}_{y,K}\right\}}{}&\\
\le\quad\exp\left(-\frac{c\theta_4^2}{\sum_{i}h(z_i)^2}\right)\quad\le\quad\exp\left(-c\theta_4^2(\log N)^2\right)\quad\ll\quad\frac{1}{N^2}.&
\end{align*}

This finishes the proof of \eqref{eq:bound_hVhat} with $\theta_5=1-\frac{1-\theta_3}{1-\theta_1}$, hence finishing the proof of \Cref{Succhat_prob}.
\end{proof}

\subsection{Concluding the proof of \Cref{transition_probs_on_Vnice}} \label{sec:transition_probs_on_Vnice}

Now we are ready to conclude the proof of \Cref{transition_probs_on_Vnice}.

\begin{proof}[Proof of \Cref{transition_probs_on_Vnice}]
We can write
\begin{align*}
&\prscond{X_{\taublue^{(2)}-1}=y,\:\taublue^{(2)}\le C\log N}{G^*}{x}\quad\\
&\ge\quad\sum_{z}\prscond{X_{\taublue}=z,\:\taublue\le  \frac12C\log N}{G^*}{x} \prscond{X_{\taublue-1}=y,\:\taublue\le  \frac12C\log N}{G^*}{z}\\
&=\quad\sum_{z}\prscond{X_{\taublue}=z,\:\taublue\le  \frac12C\log N}{G^*}{x} \prscond{X_{\taublue-1}=z,\:\taublue\le  \frac12C\log N}{G^*}{y}\frac{\db(y)}{\db(z)}\\
&=\quad\db(y)\sum_{w,z}\prscond{X_{\taublue-1}=w,\:\taublue\le  \frac12C\log N}{G^*}{x} \frac{1}{\db(w)}\\
&\hspace{6cm}\cdot\prscond{X_{\taublue-1}=z,\:\taublue\le  \frac12C\log N}{G^*}{y} \frac{1}{\db(z)}\1{z\simb w}\\
&=\quad\db(y)\sum_{e,f}\prscond{X_{\taublue}=e,\taublue\le  \frac12C\log N}{G^*}{x} \prscond{X_{\taublue}=f,\taublue\le  \frac12C\log N}{G^*}{y}\1{e\simb f}
\end{align*}
where the last sum is taken over blue half-edges, $X_{\taublue}=e$ means that the first blue half-edge crossed by $X$ is $e$, and $e\simb f$ means that $e$ and $f$ are the two ends of the same blue edge.

Let $x$ and $y$ be two vertices of $G$ and let $\left\{\what{T}_{x,K}^a,\what{T}_{y,K}\right\}$ denote the event that $x$ is almost nice, with distinguished edge $\{x,x^a\}$, $y$ is nice, the level $K$ neighbourhood of $x$ in $G^*\setminus\left\{\{x,x^a\}\right\}$ and the level $K$ neighbourhood of $y$ in $G^*$ are disjoint and do not contain $x^a$, and these neighbourhoods look like $\what{T}_{x,K}$ and $\what{T}_{y,K}$ respectively.

Let us consider the explorations and couplings defined in \Cref{sec:coupling}, and let us abbreviate $\what{\F}_{L_x+L_y}$ as $\what{\F}$. Let $\what{\F}_b$ denote the $\sigma$-algebra generated by $\what{\F}$ and all the blue degrees in $G^*$. Let $A$ denote the high probability event of \Cref{blue_almost_unif} for the blue degrees $(\db(v))$ in $G^*$.

Then we can write the following.
\begin{align*}
&\prscond{\prscond{X_{\taublue^{(2)}-1}=y,\:\taublue^{(2)}\le C\log N}{G^*}{x}\:<\:c\frac{\log{N}}{N}\db(y),\: A}{\left\{\what{T}_{x,K}^a,\what{T}_{y,K}\right\}}{}\\
&\le \quad\mathbb{P}\Bigg(\sum_{e,f}\prscond{X_{\taublue}=e,\:\taublue\le  \frac12C\log N}{G^*}{x}\\
&\hspace{3cm}\cdot\prscond{X_{\taublue}=f,\:\taublue\le  \frac12C\log N}{G^*}{y}\1{e\simb f}\:<\:c\frac{\log{N}}{N},\: A\:\Bigg|\:\left\{\what{T}_{x,K}^a,\what{T}_{y,K}\right\}\Bigg)\\
&\le\quad\mathbb{P}\Bigg(\sum_{e,f}\prscond{X_{\taublue}=e,\:\taublue\le  \frac12C\log N,\:\taublue<\tau_{x^a}}{\what{\F}_b}{x}\\
&\hspace{3cm}\cdot\prscond{X_{\taublue}=f,\:\taublue\le  \frac12C\log N}{\what{\F}_b}{y}\1{e\simb f}\:<\:c\frac{\log{N}}{N},\: A\:\Bigg|\:\left\{\what{T}_{x,K}^a,\what{T}_{y,K}\right\}\Bigg)\\
&\le\quad\mathbb{P}\Bigg(\sum_{e,f}\prscond{X_{\taublue}=e,\:\taublue\le  \frac12C\log N,\:\taublue<\tau_{x^a}}{\what{\F}_b}{x}\\
&\hspace{3cm}\prscond{X_{\taublue}=f,\:\taublue\le  \frac12C\log N}{\what{\F}_b}{y}\1{e\simb f} \:<\:c\frac{\log{N}}{N},\\ &\hspace{4cm}\:\prscond{\Succhat}{\what{\F}_b}{}>\theta,\:\sum_{v}\db(v)<\Cblue\frac{N}{\log N},\: A \:\Bigg|\:\left\{\what{T}_{x,K}^a,\what{T}_{y,K}\right\}\Bigg) +o\left(\frac{1}{N^2}\right).
\end{align*}

In the last line, $\Cblue$ is a sufficiently large positive constant, $\theta$ is a sufficiently small positive constant, and we used \Cref{Succhat_prob} and \Cref{total_blue_degree}.
Now let us fix $\what{\F}_b$ such that $\prscond{\Succhat}{\what{\F}_b}{}>\theta$, $\sum_{v}\db(v)<\Cblue\frac{N}{\log N}$, and $A$ holds. Then assuming that $C>\frac{8}{\nu}c_L$ and using the notation from \Cref{sec:coupling} we can write
\begin{align*}
&\sum_{e,f}\prscond{X_{\taublue}=e,\taublue\le  \frac12C\log N,\taublue<\tau_{x^a}}{\what{\F}_b}{x} \prscond{X_{\taublue}=f,\taublue\le  \frac12C\log N}{\what{\F}_b}{y}\1{e\simb f}\quad\\
&\quad\ge\quad\sum_{e,f}\prscond{\what{X}^{(1)}_{\taublue}=e,\:\taublue<\tau_{x^a},\:\what{X}^{(2)}_{\taublue}=f,\:\Succhat}{\what{\F}_b}{x,y}\1{e\simb f}\\
&\quad=\quad\sum_{e,f}\prscond{\what{Z}^{(1)}_{\taublue}=e,\:\what{Z}^{(2)}_{\taublue}=f,\:\Succhat}{\what{\F}_b}{x,y}\1{e\simb f}\:.
\end{align*}
Let $\I$ denote the set of blue half-edges. By assumption we know that $|\I|\le\Cblue\frac{N}{\log N}$ and $\sum_{e,f}\prscond{\what{Z}^{(1)}_{\taublue}=e,\:\what{Z}^{(2)}_{\taublue}=f,\:\Succhat}{\what{\F}_b}{x,y}=\prscond{\Succhat}{\what{\F}_b}{}>\theta$, and from \Cref{bound_prob_specific_blue_edge} we also know that $\prscond{\what{Z}^{(1)}_{\taublue}=e,\:\what{Z}^{(2)}_{\taublue}=f,\:\Succhat}{\what{\F}_b}{x,y}\le\frac{1}{N\log N}$ for each pair $(e,f)$.

Using \cite[Lemma 5.1]{cutoff_NBRW_on_sparse_random_graphs} with weights $w_{e,f}=\prscond{\what{Z}^{(1)}_{\taublue}=e,\:\what{Z}^{(2)}_{\taublue}=f,\:\Succhat}{\what{\F}_b}{x,y}$, and using that by \Cref{blue_almost_unif} on event $A$ the matching of the blue half-edges is distributed as a uniform random matching of $\I$ conditioned on a high probability event, we get that for sufficiently small values of $c$ we have
\begin{align*}
\prscond{\sum_{e,f}\prscond{\what{Z}^{(1)}_{\taublue}=e,\:\what{Z}^{(2)}_{\taublue}=f,\:\Succhat}{\what{\F}_b}{x,y}\1{e\simb f}\:<\:c\frac{\log N}{N}}{\what{\F}_b}{}\quad\ll\quad\frac{1}{N^2}.
\end{align*}
This finishes the proof using the union bound.
\end{proof}

\section{Hitting an almost nice vertex}\label{sec:hit_nice}

The goal of this section is to prove the following result.

\begin{proposition}\label{hit_almost_nice_quickly}
For any $\theta\in(0,1)$ there exists $C>0$ such that whp $G^*$ satisfies the following. For any vertex $x$, with probability at least $\theta$, a simple random walk on $G^*$ starting from $x$ hits $\Vniceprime$ by time $C\log N$.
\end{proposition}

In the proof, we will consider an exploration of a random neighbourhood of $x$ and a coupling of this neighbourhood to a certain random quasi-tree. Note that the quasi-tree, exploration and coupling used in this section are different from the ones used in Sections \ref{sec:T} and \ref{sec:Gstar_and_T}.

In this section we consider \emph{random quasi-trees} with balls of radius
\[
\Rnice:=\quad c_R\log N,
\]
where $c_R$ is a sufficiently small constant, with the green and black edges being considered for determining the balls, and the red and blue edges acting as long-range edges. \footnote{The precise definition is as follows. A random quasi-tree is a random graph $T$ together with a map $\iota$ from its vertices to the vertex set of $G$ obtained as follows. We start from a root vertex $\rho$ with some given $\iota(\rho)$. Then we consider a ball of radius $\Rnice$ around $\iota(\rho)$ in a random copy of $\Gg$ and attach it around $\rho$. We call it an $\Rnice$-ball, and we let $\iota$ map its vertices to the corresponding vertices in $G$. Then for each vertex $v$ in the ball we sample $\dr(v)\eqdist\dr$ and $\db(v)\eqdist\db$ independently and attach $\dr(v)$ red and $\db(v)$ blue long-range edges to it. For a red long-range edge from $v$ we sample the $\iota$ of its new endpoint so that it takes value $u$ with probability proportional to $\pred_{v,u}$, and for a blue long-range edge we sample according to $\pblue_{v,\cdot}$ (all of these independently). If $w$ is the new endpoint of a long-range edge, we attach a ball around it corresponding to a ball of radius $\Rnice$ around $\iota(w)$ in an independent copy of $\Gg$, and we let $\iota$ map its vertices to the corresponding vertices of $G$. We proceed similarly.}

The \emph{boundary} of a given ball consists of vertices that are at graph distance $\Rnice$ from the centre of the ball, where the graph distance is with respect to the quasi-tree. We sometimes write $\partial$ for the set of boundary vertices. Given a walk $X$ on $T$ or $\what{T}$, let $\taubound$ be the first time that $X$ hits the boundary of a ball. (Note that in the previous sections $\taubound$ referred to the first time that the walk hits the $R$th different vertex within a ball.)

We define $T(G^*)$, the quasi-tree induced by $G^*$ analogously to \Cref{def:quasi-tree_of_Gstar}, and we also use notation and terminology analogous to the one in \Cref{def:quasi-tree_notation}.

In what follows, we often write $R$ for $\Rnice$.

Note that a transience-like result analogous to \Cref{transience_until_boundary} holds for $T$.

\begin{lemma}\label{nice:stopped_walk_transience}
There exists a positive constant $\qesc$ (not depending on any of the parameters, nor on the dimension $d$) with the following property.

For any realisation $T$ of a random quasi-tree, and any vertex $x$ of $T$, the probability that a walk from $x$ does not return to $x$ without first hitting the boundary of a ball can be lower bounded as
\[
\prstart{\tau_x^+>\taubound}{x}\quad\ge\quad\frac{\qesc}{\deg_T(x)},
\]
while for any long-range edge $e$, the probability that a walk from $x$ does not cross $e$ without hitting a boundary first can be lower bounded as
\[
\prstart{\tau_e>\taubound}{x}\quad\ge\quad\qesc.
\]
\end{lemma}

The proof is the same as the proof of \Cref{transience_until_boundary} and hence omitted.

\subsection{Exploration and coupling}

In what follows, we describe an exploration of a random neighbourhood of a given vertex in $G^*$ and a coupling of this neighbourhood to a random quasi-tree $T$ rooted at $x$.

In the following sections, we prove that a walk on $G^*$ starting from $x$ likely exits this neighbourhood in $\lesssim\log N$ steps, and likely exits via a red or blue edge leading to an almost nice vertex. This will form a major step in the proof of \Cref{hit_almost_nice_quickly}.

We start by giving some auxiliary definitions. Firstly, we introduce a quantity that roughly speaking corresponds to the probability of ever hitting a given long-range edge, but it is defined in a way that it only depends on the balls that are ancestors of the given edge.

\begin{definition}\label{def:thetaT}
Consider a quasi-tree $T$ and let $e$ be a long-range edge. Let $\ell$ be the level of $e$ and for $i\in\{1,2,...,\ell\}$ let $e_i$ be the long-range edge that is the level $i$ ancestor of $e$. Then let us consider a graph that consists of the balls containing $e_i^-$ for each $i$, the long-range half-edges from these balls, and the long-range edges $(e_i)_i$. Let us consider a walk on this graph started from $\rho$ that takes steps like a simple random walk, except for the following. It is not allowed to cross from $e_i^+$ to~$e_i^-$ for all $i\leq \ell$. If it crosses a long-range half-edge\footnote{Note that $(e_i)$ are not considered as half-edges.}, with probability $\qesc$ it gets terminated, otherwise, it crosses back and continues walking. It is also terminated if it hits the boundary of a ball or it hits $e^+$. Let $\theta_T(e)$ denote the probability that this walk reaches $e^+$ without getting terminated for some other reason first. Here $\qesc$ denotes the constant from \Cref{nice:stopped_walk_transience}. 
\end{definition}

We use $\theta_T$ to define a notion of truncation as follows.

\begin{definition}\label{def:trunc_nice}
Given a quasi-tree $T$ with root $\rho$, and a long-range edge $e$ in $T$, let us say that $e$ is \emph{truncated} if $\theta_{T}(e)\le\frac{\ctrunc}{\log N}$, where $\ctrunc$ is a positive constant to be chosen later.
\end{definition}

Now we present the exploration and coupling process.

\begin{definition}\label{def:nbhd_for_hitting_nice}
Consider a random quasi-tree $T$ rooted at a given vertex $x$. Let us reveal the ball around $x$, together with the long-range half-edges from it, and see which of the long-range half-edges (if any) are truncated according to \Cref{def:trunc_nice}. Then in $G^*$ reveal the ball of radius $R$ around $x$ according to the green and black edges, and couple it exactly to the ball around the root of $T$. Also reveal the red and blue degrees of each vertex in this ball in $G^*$ and couple them exactly to the corresponding red and blue degrees in $T$. (The two balls and the red and blue degrees of their vertices have the exact same distribution, so they can be coupled so that they always agree.)

One by one, consider the long-range half-edges from the revealed ball in $T$. If a long-range half-edge is truncated, then we also truncate it in $G^*$ and do not explore further in that direction. Otherwise, we reveal the ball in $T$ around its other endpoint, and we also reveal the red and blue degrees from this ball. We also reveal the other endpoint of the corresponding long-range edge in $G^*$, the ball of radius $R$ around it according to the green and black degrees, and the red and blue degrees of each vertex in the ball, coupling all these to the ball in $T$ using the optimal coupling. If the new ball in $T$ contains any of the previously revealed vertices or the optimal coupling fails, we do the following. In case this is the first time it happens, we denote the long-range edge leading to the new ball by $\til{e}$ and truncate it, not exploring further in this direction. If it is not the first time this happens, then we declare the exploration unsuccessful and terminate the process.

Once we considered all the long-range half-edges from the ball around $x$, we continue similarly with the long-range half-edges from the level 1 balls, and then continue similarly level by level (until we need to terminate the exploration and declare it unsuccessful, or until we successfully couple a certain subgraph of $G^*$ with all the balls of $T$ that can be reached from the root without crossing a truncated edge or $\til{e}$.)

If the exploration was successful, let $A$ denote the set of truncated long-range half-edges in $T$, except for~$\til{e}$.
\end{definition}

The main results about this exploration are the following.

\begin{proposition}\label{nice:coupling_success_and_nbhd_size}
Fix a vertex $x$ in $G$. With probability $1-o\left(\frac1N\right)$ the exploration described in \Cref{def:nbhd_for_hitting_nice} succeeds, and the explored random neighbourhood has size $\le N^c$ where $c>0$ is a constant that can be made arbitrarily small by choosing $c_R$ sufficiently small.
\end{proposition}

\begin{proposition}\label{nice:nbhd_properties}
For any constant $\theta\in(0,1)$ there exist positive constants $C$ and $C'$ with the following property.
For any fixed vertex $x$ in $G$, with probability $1-o\left(\frac1N\right)$ the exploration described in \Cref{def:nbhd_for_hitting_nice} succeeds and a walk $X$ on $T$ from $x$ satisfies the following.
\begin{enumerate}[(i)]
\item\label{item:nice:bound_num_visited} With high probability the walk $X$ visits $\le C'\log N$ different vertices by the time it hits $A\cup\{\til{e}\}$.
\item\label{item:nice:bound_time} With probability at least $1-\theta$ the number of steps it takes for the walk $X$ to visit $C'\log N$ different vertices or hit $A\cup\{\til{e}\}\cup\partial$ is $\le C\log N$.
\item\label{item:nice:not_hit_etil} With probability at least $\qesc$ the walk $X$ hits $A\cup\partial$ before hitting $\til{e}$. (Here $\qesc$ is the constant from \Cref{nice:stopped_walk_transience}, and it does not depend on any other parameter.)
\item\label{item:nice:bound_prob_given_e} For each $e\in A$ the probability of $X$ hitting $A\cup\{\til{e}\}\cup\partial$ at $e$ is $\le\frac{\ctrunc}{\log N}$.
\item\label{item:nice:not_hit_boundary} {With high probability the walk $X$ hits $A\cup\{\til{e}\}$ before hitting $\partial$.}
\end{enumerate}
\end{proposition}

\subsection{Some auxiliary estimates}

In this section we establish some auxiliary results that we will use in the next sections.

Firstly, note the following simple properties of $\theta_T$.

\begin{lemma}\label{nice:thetaT_properties} For any realisation of $T$, the function $\theta_T$ defined in \Cref{def:thetaT} satisfies the following properties.
\begin{enumerate}[(i)]
\item\label{item:thetaT_multiplicative} For $(e_i)_{i=1}^{\ell}$ as in \Cref{def:thetaT} and $e_0^+=\rho$, we have $\theta_T(e)=\prod_{i=1}^{\ell}\theta_{T(e_{i-1}^+)}(e_i)$.
\item\label{item:thetaT_exp_decay} $\theta_T(e)\le(1-\qesc)^{\ell(e)}$.
\item\label{item:thetaT_hitting} $\theta_T(e)\ge\prstart{\tau_e<\taubound}{\rho}$.
\item\label{item:thetaT_sum} $\sum_{e:\:\ell(e)=1}\theta_T(e)\le \frac{1}{\qesc}$.
\end{enumerate}
\end{lemma}
\begin{proof}
 We present the proof of \eqref{item:thetaT_sum}.
The proofs of other items are simple using  \Cref{def:thetaT} and \Cref{nice:stopped_walk_transience}, and are omitted. Notice that for the edges in the first level, $\theta_T(e)$ only depends on the first $R$-ball. Consider a walk $Y$ from $\rho$ on the first $R$-ball and all long-range half-edges from it, which, every time it crosses a half-edge, gets terminated with probability $\qesc$ and otherwise returns to the same vertex and continues. Then $\theta_T(e)=\pr{Y\text{ hits }e}$. Therefore, if we set $L(Y)$ to be the number of different long-range half-edges $Y$ visits before it is terminated  \[\sum_{e:\:\ell(e)=1}\theta_T(e)= \sum_{e:\:\ell(e)=1}\pr{Y \text{ hits }e}= \E{L(Y)}\le \frac{1}{\qesc}\] as every time $Y$ visits a long-range half-edge it has probaility $\qesc$ of being terminated.  \end{proof}

Using these properties, we can bound the number of balls in $T$ that can be reached from the root without crossing a truncated edge as follows.

\begin{lemma}\label{nice:bound_nontrunc_size}
There exists a positive constant $C$ such that for any realisation $T$ of a random quasi-tree, the number of balls in $T$ that can be reached from the root $\rho$ without crossing a truncated edge is $\le Ce^{C\log\log N}$, and all of these balls are at levels $\le C\log\log N$.
\end{lemma}

\begin{proof}
Let us consider the finite tree $t_0$ obtained by removing the truncated edges of $T$, keeping only the connected component of $\rho$, and contracting each ball into a single vertex. By \Cref{nice:thetaT_properties}\eqref{item:thetaT_exp_decay} we know that $t_0$ has $\lesssim\log\log N$ levels.

Let $A_1$ denote the set of edges of $t_0$ leading to leaves. From \Cref{nice:thetaT_properties} \eqref{item:thetaT_multiplicative} and \eqref{item:thetaT_sum} and using that $t_0$ has $\lesssim\log\log N$ levels, we get that $\sum_{e\in A_1}\theta_{T}(e)\le \qesc^{-C_1\log\log N}$ for some constant $C_1>0$. Then $|A_1|\le\frac{\log N}{\ctrunc}\qesc^{-C_1\log\log N}\lesssim e^{C_2\log\log N}$ for some constant $C_2>0$.

Let $t_1$ be obtained by removing the edges in $A_1$ from $t_0$ (and the vertices that get disconnected from the root this way). Let $A_2$ be the set of edges leading to leaves in $t_1$. Similarly to the above, we see that $|A_2|\lesssim e^{C_2\log\log N}$. Continuing similarly, we see that the edges of $t_0$ can be partitioned into $\lesssim\log\log N$ `layers', each of which has $\lesssim e^{C_2\log\log N}$ edges. Hence $t_0$ has $\lesssim e^{C_2\log\log N}(\log\log N)\lesssim e^{C_3\log\log N}$ vertices where $C_3>0$ is a constant.
\end{proof}

\begin{lemma}\label{prob_nice}
{For any given vertex $x$ in $G$ we have $\pr{x\not\in\Vnice}\lesssim\frac{1}{(\log N)^{3/2}}$ and for any given $x$ and $y$ in $G$ we have $\pr{x\text{ not nice in }G^*\setminus\{\{x,y\}\}}\lesssim\frac{1}{(\log N)^{3/2}}$.}  Furthermore, we have that whp $|\Vnice^c|\le \frac{N}{(\log N)^{4/3}}$. \end{lemma}

\begin{proof}

Let us explore the level $K$ neighbourhood of $x$ level by level.
From \Cref{volume_tail_bound} we know that with probability $1-o\left(\frac1N\right)$ the whole neighbourhood has volume $\lesssim e^{C(\log\log N)^2}$ where $C$ is a large constant. Every time the new endpoint of a red edge is revealed, the probability of it being any specific vertex is $\lesssim\frac{1}{N^{1-\beta}\log N}$. Hence the probability of the neighbourhood having an overlap is $\lesssim\frac{1}{N^{1-\beta}\log N}e^{2C(\log\log N)^2}\ll\frac{1}{(\log N)^2}$.

Now consider exploring (part of) the level $K$ neighbourhood of $x$ in $\Ggr$ by revealing balls as the walk $(X_t)_{t\le\tau_K}$ on $\Ggr$ started from $x$ first visits them. From \Cref{bound_atyp_prob} we know that for each newly revealed ball the probability that it is atypical is $\lesssim\left(\frac{1}{\log N}\right)^2$.

From \Cref{bound_sigma1} we know that for a sufficiently large $C$ the probability of $\tau_K\wedge\taubound\wedge\tauatyp\ge C(\log\log N)^{16}$ is $\lesssim\frac{1}{(\log N)^2}$. From \Cref{bound_atyp_prob} we know that the probability of any of the first $C(\log\log N)^{16}$ balls visited being atypical is $\lesssim\frac{(\log\log N)^{16}}{(\log N)^2}$. From the definition of typical balls, we know that the probability that the first $C(\log\log N)^{16}$ balls visited are all typical, but the walk hits the boundary in one of them, is $\lesssim\frac{(\log\log N)^{16}}{(\log N)^2}$.

Together these show that with probability $1-O\left(\frac{(\log\log N)^{16}}{(\log N)^2}\right)$ the walk visits $\le C(\log\log N)^{16}$ different vertices by $\tau_K$, and the events $\Omegatyp{\tau_K}$ and $\Omegabound{\tau_K}$ hold.

Now let us reveal the blue degree of each vertex $X$ visits by time $\tau_K\wedge\taubound\wedge\tauatyp$ and let us consider a walk $\what{X}$ on $G^*$ that agrees with $X$ until $\taublue\wedge\tau_K\wedge\taubound\wedge\tauatyp$. For each of the first $C(\log\log N)^{16}$ different vertices visited by $X$, if we condition on the blue degrees of the previous vertices and work on the event that they sum to at most one, the probability of it having a blue edge is $\lesssim\frac{1}{\log N}$ and the probability of it having at least two blue edges is $\lesssim\frac{1}{(\log N)^2}$. Hence the probability of having more than one blue edge emanating from the first $C(\log\log N)^{16}$ different vertices visited by $X$ is $\lesssim\frac{(\log\log N)^{32}}{(\log N)^2}$.  Using Markov's inequality, this gives that with probability at least $ 1- \frac{1}{(\log N)^{3/2}}$ the graph $G^*$ and $x$ are such that with probability at least $1-\frac{1}{(\log N)^{1/3}}$ the walk starting from $x$ visits at most one vertex with exactly one blue edge and no vertices with more than one blue edges by  time $\tau_K$. 
By \Cref{transience_until_nonbinary} for any realisation of the graph, the probability that $X$ revisits a given vertex $v$ more than $C'\deg(v)$ times by $\taublue\wedge\tau_K\wedge\taubound\wedge\tauatyp$ is $\le\eps$, where $\eps$ can be made arbitrarily small by choosing $C'$ sufficiently large. Applying this to the first vertex visited by $X$ that has a blue half-edge, we see that with positive probability $\what{X}$ does not cross the blue half-edge from this vertex by time $\taublue\wedge\tau_K\wedge\taubound\wedge\tauatyp$.

Combining the above estimates finishes the proof of the first part.

The proof of the second part is analogous to the first one. The third part follows from the first one and Markov's inequality, as the expected number of vertices which are not nice is at most $\frac{N}{(\log N)^{3/2}}$.
\end{proof}

\subsection{Success probability, size of the explored region (proof of \Cref{nice:coupling_success_and_nbhd_size})}

\begin{proof}[Proof of \Cref{nice:coupling_success_and_nbhd_size}]
From \Cref{volume_tail_bound}, we know that each time a new ball is revealed in $T$, with probability $1-o\left(\frac{1}{N^2}\right)$, its volume plus the total number of half-edges emanating from it is $\le e^{C_1R}$ where $C_1$ is a positive constant. For any $c_2>0$ we can get $e^{C_1R}\le e^{c_2\log N}$ by choosing $c_R$ sufficiently small. By \Cref{nice:bound_nontrunc_size} the number of balls in the explored region of $T$ is $\lesssim e^{C_3\log\log N}$ for some $C_3>0$, so the overall volume of this region plus the number of half-edges from it is $\lesssim e^{c_2\log N}e^{C_3\log\log N}$, which can be made $\le e^{c_4\log N}$ for any $c_4>0$.

We will show that each time a new ball in $T$ is revealed, on the event that the total number of already revealed vertices plus half-edges is $\lesssim N^{c_4}$, the probability that the new ball has an overlap or its optimal coupling with the ball in $G^*$ fails is
\begin{align}\label{eq:nice:bound_overlap_prob}
\lesssim N^{c_4+c_2}\frac{1}{N^{1-\beta}\log N}+\pr{|B_{R+1}|\ge N^{c_2}}.
\end{align}

Using this we see that the probability that there is more than one instance when there is an overlap or the optimal coupling fails is $\le\pr{\Bin{C_5N^{c_4}}{C_6\frac{N^{c_4+c_2}}{N^{1-\beta}\log N}}>1}+C_5N^{c_4}\pr{|B_{R+1}|\ge N^{c_2}}$, which is $o\left(\frac{1}{N}\right)$ for sufficiently small values of $c_2$ and $c_4$, and sufficiently small values of $c_R$ in terms of $c_2$.

We can prove \eqref{eq:nice:bound_overlap_prob} as follows.  For a set of vertices $S$ let $\Gamma(S)$ denote the set of points that are in $S$ or adjacent to a vertex in $S$ in $G^*$.

Let $\mathcal{A}$ denote the set of already explored vertices and let $e$ be a given half-edge from $\mathcal{A}$. Let $Z=e^{+}$ be the other endpoint of $e$, let $\mathcal{B}(Z)$ be the $R$-ball in green and black edges around $Z$ and let $\mathcal{B}$ be an independent copy of an  $R$-ball in green and black edges in $G^*$ centred at the origin. Then for any values of $A$, $\Gamma_{A}$, $e'_{-}$, $z$, $B$ and $\Gamma_B$ with $(z+\Gamma_B)\cap\Gamma_A=\emptyset$ and $(e_{-}',z)\notin z+\Gamma_B$ we have
\begin{align*}
\prcond{Z=z,\mathcal{B}(Z)=z+B,\Gamma(\mathcal{B}(Z))={(z+\Gamma_B)\cup (e_{-}',z)}}{\mathcal{A}=A,\Gamma(\mathcal{A})=\Gamma_A,e_{-}=e'_{-}}\qquad\qquad&\\
\quad=\quad\prcond{Z=z,\mathcal{B}=B',\Gamma(\mathcal{B})=\Gamma_B}{e_{-}=e'_{-}}.&
\end{align*}

This shows that the probability that the optimal coupling succeeds and that there is no overlap can be lower bounded by the probability that $Z$, $\mathcal{B}$ and $\Gamma(\mathcal{B})$ satisfy $(Z+\Gamma(\mathcal{B}))\cap\Gamma_A=\emptyset$ and $(e_{-}',Z)\notin Z+\Gamma_B$ for a given realisation of $A',\Gamma'_A, e'_{-}$.

For any realisation of $\Gamma_B$ and $\Gamma_A$ we have $\#\left\{z:\:(z+\Gamma_B)\cap\Gamma_A\ne\emptyset\right\}\le|\Gamma_A|\cdot|\Gamma_B|$, for any specific value of $z$ we have $\pr{Z=z}\lesssim\frac{1}{N^{1-\beta}\log N}$ and the probability that any specific red or blue edge $(e_{-}'-z,0)$ is in $
\Gamma(\mathcal{B})$ is at most $\lesssim\frac{1}{N^{1-\beta}\log N}$. Therefore, the probability of overlap or coupling failing given the realisation of $e_{-}, \Gamma(\mathcal{A}), \Gamma(\mathcal{B})$ is bounded by
\begin{align*}& \prcond{Z\in \left\{z:\:(z+\Gamma_B)\cap\Gamma_A\ne\emptyset\right\}}{\Gamma(\mathcal{A})=\Gamma_A,e_{-}=e'_{-}, \Gamma(\mathcal{B})=\Gamma_B}
\\&+ \prcond{(e'_{-},Z)\in Z+\Gamma_B}{\Gamma(\mathcal{A})=\Gamma_A,e_{-}=e'_{-}, \Gamma(\mathcal{B})=\Gamma_B}\lesssim \frac{|\Gamma_A||\Gamma_B|+1}{N^{1-\beta}\log N}
\end{align*}
and using the assumption $|\Gamma(\mathcal{A})|\le N^{c_4}$ the desired inequality follows. 
\end{proof}

\subsection{Cool boxes, bounding the range}

A major idea in the proof of \Cref{nice:nbhd_properties} \eqref{item:nice:bound_num_visited} {and \eqref{item:nice:not_hit_boundary}} is considering so-called cool boxes. In this section, we give the definition and prove some important properties that together {with \Cref{nice:bound_nontrunc_size}} imply a bound corresponding to \Cref{nice:nbhd_properties} \eqref{item:nice:bound_num_visited} {and \eqref{item:nice:not_hit_boundary}}.

\begin{definition}\label{def:cool_box}
Let $r$ be a large constant to be chosen later. Let us partition the torus $G$ into boxes of side length $r$. Let us call a box \emph{cool} if in $G^*$ there are at least two red edges emanating from it, and each of its vertices has degree $\le\dcool$, where $\dcool$ is a constant to be chosen later.
\end{definition}

\begin{remark}
The following arguments would also work if we considered blue and red edges instead of just red ones.
\end{remark}

The main results of this section are the following.

\begin{lemma}\label{nice:visit_many_cool_boxes}
Let $\Ccool$ be any constant. Then for sufficiently large values of $r$ and $\dcool$ the following holds. For any vertex $x$ of $G$, with probability $1-o\left(\frac1N\right)$, the graph $G^*$ is such that {with high probability} a walk $X$ on it from $x$ visits $\ge\frac{R}{dr}$ different boxes before reaching graph distance $R$ from $x$, and out of the first $\frac{R}{dr}$ different boxes visited $\ge\Ccool\frac{\log N}{\log\log N}=:\tcool$ are cool.
\end{lemma}

\begin{lemma}\label{nice:cross_many_red_edges}
Fix any vertex $x$ of $G$, and consider any realisation of $G^*$ with at least $\frac{\log N}{\log\log N}$ cool boxes. Then with high probability a walk $X$ on $G^*$ from $x$ crosses $\ge\frac{\log N}{(\log\log N)^2}$ different red edges by the time it visits $\frac{\log N}{\log\log N}$ different cool boxes.
\end{lemma}

\begin{lemma}\label{nice:reach_level_CloglogN}
For any positive constant $C$, the following holds. For any realisation of the quasi-tree $T$, with high probability a walk $X$ on it from its root $x$ reaches level $C\log\log N$ of $T$ or hits the boundary of a ball before crossing $\frac{\log N}{(\log\log N)^2}$ different red edges.
\end{lemma}

We present their proofs below.

\begin{proof}[Proof of \Cref{nice:visit_many_cool_boxes}]
Let us explore parts of $G^*$ as the walk visits them. First, let us reveal the green, red and blue degrees of each vertex in the box containing $x$. This will in particular tell us whether the box is cool. Then always proceed with the exploration as follows. Take steps with the walk. If the walk steps along a green, red or blue half-edge whose other endpoint is not revealed yet, then sample the other endpoint of the half-edge with the correct distribution and let the walk step there. If the walk enters a box that has not been visited before, then reveal the green, red and blue degrees of the box, hence also revealing whether the box is cool.

Each box has diameter $\le dr$, hence the walk visits at least $\frac{R}{dr}$ different boxes before reaching graph distance $R$ from $x$.

 We will now show that there is a small constant $c$ and a large constant $C$ such that the first $\frac{R}{dr}$ times that a box is visited for the first time, that box has probability at least $1-\puncool$ of being cool, regardless of what was explored so far, where $\puncool\le C\left( e^{-cr^d}+\frac{r^d}{\dcool}\right)$. Indeed, assume that the above process has so far revealed $<\frac{R}{dr}$ different boxes and that the walk now steps into a new box and let the box have vertices labelled by $v_1$, ..., $v_{r^d}$. As $\frac{R}{dr}r^d= N^{o(1)}$ for each of these $v_i$s the probability of having a red edge leading to one of the yet unrevealed vertices is lower bounded by a positive constant, meaning that the probability that the box has fewer than $2$ red edges can be upper bounded by $e^{-\text{const}(r^d-1)}(r^d+1)\le Ce^{-cr^d}$. To bound the probability of any of the vertices having degree larger than $\dcool$, first notice that the expected sum of green, red and blue degrees to yet unrevealed vertices is upper bounded by the expected sum of green, red and blue degrees of a vertex in a graph where nothing was revealed, which is $1$. We know by \Cref{max_deg} that the probability of any vertex having degree $>\dmax$ is $o\left(\frac1N\right)$. Even if all of the so far revealed vertices had degree $\dmax$, using \Cref{cond_green_red_deg}  we get that the expected number of edges between these and a given $v_i$ is $\lesssim\dmax\sum_{y:\:|y|\le\left(\frac{R}{dr}r^d\right)^{\frac1d}}p_{0,y}\lesssim\frac{\log N}{\log\log N}\sum_{y:\:|y|\le(C\log N)^{\frac1d}}\frac{1}{\log N}\frac{1}{|y|^d}\asymp1$. Markov's inequality gives that the probability of at least one $v_i$ having degree $>\dcool$ is $\lesssim\frac{r^d}{\dcool}$.

Using these and \Cref{Bin_rough_tail_bound} we get that the probability of seeing less than $\tcool$ different cool boxes is
\begin{align*}
&\le\:\pr{\Bin{\frac{R}{dr}}{1-\puncool}\le\tcool}\:\le\:\left(\frac{R}{dr}\right)^{\tcool}\puncool^{\frac{R}{dr}-\tcool}\:\le\:\left(\frac{R}{dr}\right)^{\tcool}\left(Ce^{-cr^d}+C\frac{r^d}{\dcool}\right)^{\frac12\frac{R}{dr}}
\\&=\:\exp\left(\Ccool\frac{\log N}{\log\log N}\left(\log\log N-\log\left(\frac{dr}{c_R}\right)\right)+\frac12\frac{c_R \log N}{dr}\log\left(Ce^{-cr^d}+C\frac{r^d}{\dcool}\right)\right)
\end{align*}
This can be made $\ll\frac{1}{N^2}$ by setting $r$ sufficiently large in terms of $\Ccool, C, c, c_R$ and $d$, so that $\Ccool+\log(2C)\frac{c_R}{2dr}-r^{d-1}c\frac{c_R}{2d}<-2 $ and setting $\dcool$ to satisfy $r^de^{cr^d}\le \dcool$. The proof now follows by Markov's inequality. 
\end{proof}

\begin{figure}[h]
\centering
\includegraphics[width=100mm]{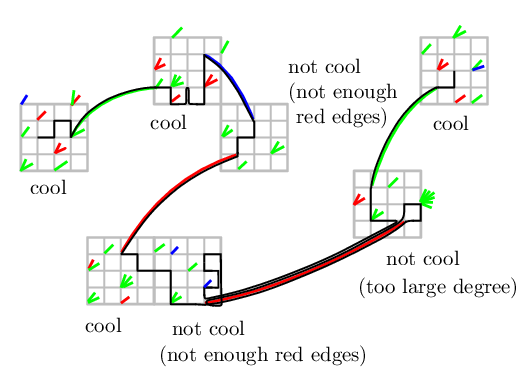}
\caption{Illustration for the proof of \Cref{nice:visit_many_cool_boxes}. The walk is marked with black, the boxes revealed are shown in grey, and the green, red and blue degrees of the vertices within these boxes are also indicated.}
\end{figure}

\begin{proof}[Proof of \Cref{nice:cross_many_red_edges}]
Note that each time the walk $X$ enters a cool box, there is a positive probability that it exits the box via a red edge that is different from the edge it entered through. Say this probability is at least $q$ where $q$ is a small positive constant depending on $r$ and $\dcool$.

Then the probability of the walk crossing $<\frac{\log N}{(\log\log N)^2}$ different red edges by the time it visits $\frac{\log N}{\log\log N}$ different cool boxes is
\[
\le\quad\pr{\Bin{\frac{\log N}{\log\log N}}{q}<\frac{\log N}{(\log\log N)^2}}\quad\ll\quad1.\qedhere
\]
\end{proof}

\begin{proof}[Proof of \Cref{nice:reach_level_CloglogN}]
Let $\ell\in\left\{1,...,C\log\log N\right\}$. We know from \Cref{nice:stopped_walk_transience} that every time the walk crosses from level $(\ell-1)$ to level $\ell$, there is a positive probability that it will not return to level $(\ell-1)$ without hitting the boundary of a ball first. Therefore the probability of crossing from level $(\ell-1)$ to $\ell$ more than $\frac{\log N}{C(\log\log N)^3}$ times without hitting a boundary is $\ll\frac{1}{\log\log N}$. Taking the union bound over $\ell$ gives the result.
\end{proof}

\subsection{Concluding the proof of \Cref{nice:nbhd_properties}}

In what follows, we present some auxiliary results to establish \Cref{nice:nbhd_properties}\eqref{item:nice:bound_time}, then we conclude the proof of \Cref{nice:nbhd_properties}.

\begin{lemma}\label{bound_sum_of_first_k_deg}
For any $C_0>0$ there exists a $C>0$ such that for any $c\in(0,C_0)$ the following is satisfied. For any $x$, with probability $1-o\left(\frac1N\right)$, the graph $G^*$ is such that, whp, the sum of the degrees of the first $c\log N$ different vertices visited by a walk $X$ from $x$ is $\le Cc\log N$.
\end{lemma}

\begin{proof}
Let us explore (part of) the graph $G^*$ as we take steps with the walk $X$ as follows.

If the walk $X$ is at a vertex $v$ that is visited for the first time, let us sample the green, red and blue degrees of $v$ according to the correct distributions based on what we have revealed so far. Then pick one of the half-edges from $X$ uniformly randomly and if its other endpoint is not known yet, let us sample it according to the correct distribution. Then let $X$ step to the other endpoint of the edge.

If the walk $X$ is at a vertex $v$ that has been visited before, then we already know the green, red and blue degree of $v$. Let us pick a half-edge from $v$ uniformly at random, if its other endpoint is not revealed yet, then sample it according to the correct distribution, and let $X$ step to the other endpoint of the edge.

Let us continue this until $X$ visits $c\log N$ different vertices. Let $v_k$ denote the $k$th new vertex visited (let $v_0=x$) and let $\mathcal{H}_{k-1}$ be the $\sigma$-algebra generated by everything we know until then. Using \Cref{cond_green_red_deg} and the independence of edges of different colours, we get that for any realisation of $\mathcal{H}_{k-1}$ with $\deg(v_i)\ll\log N$ for all $i\in\{0,1,...,k-1\}$ we have
\begin{align*}
\left(\dg(v_k),\dr(v_k),\db(v_k)\mid\mathcal{H}_{k-1}\right)\quad\stle\quad\left(\sum_{y}B_{g,y},\:\sum_{y}B_{r,y},\:\sum_{y}B_{b,y}\right)
\end{align*}
where $B_{g,y}$, $B_{r,y}$ and $B_{b,y}$ are independent Bernoulli random variables, with $B_{g,y}$, $B_{r,y}$ and $B_{b,y}$ having parameters $\pgreen_{v_k,y}$, $\pred_{v_k,y}$ and $\pblue_{v_k,y}$ respectively for $y\not\in\{v_0,...,v_{k-1}\}$, and having parameters $C_1\dg(y)\pgreen_{v_k,y}$, $C_1\dr(y)\pred_{v_k,y}$ and $C_1(\log N)\db(y)\pblue_{v_k,y}$ respectively for $y\in\{v_0,...,v_{k-1}\}$, where $C_1$ is a positive constant.

If $k\le C_0\log N$ and $\dg(v_i),\dr(v_i),\db(v_i)\le C_2\frac{\log N}{\log\log N}$ for all $i\in\{0,1,...,k-1\}$, then we have 
\begin{align*}
\left(\sum_{y}B_{g,y},\:\sum_{y}B_{r,y},\:\sum_{y}B_{b,y}\right)\quad\stle\quad\left(\sum_{y}I_{g,y},\:\sum_{y}I_{r,y},\:\sum_{y}I_{b,y}\right)
\end{align*}
where $I_{g,y}$, $I_{r,y}$ and $I_{b,y}$ are independent Bernoulli random variables, with $I_{g,y}$ and $I_{b,y}$ having parameters $C_1C_2\frac{\log N}{\log\log N}\pgreen_{0,y}$ and $C_1C_2\frac{(\log N)^2}{\log\log N}\pblue_{0,y}$ respectively for the $C_0\log N$ vertices $y$ with the smallest graph distance from $0$; $I_{g,y}$ and $I_{b,y}$ having parameters $\pgreen_{0,y}$ and $\pblue_{0,y}$ respectively for the other vertices $y$; $I_{r,y}$ having parameter $C_1C_2\frac{\log N}{\log\log N}\pred_{0,y}$ for the $C_0\log N$ vertices $y$ that have the smallest graph distance from $0$ while satisfying $|y|>D^{1-\beta}$; and $I_{r,y}$ having parameter $\pred_{0,y}$ for the other vertices $y$\footnote{{The stochastic domination follows from the fact that if $p>q$ and $s>1$ then for independent Bernoulli random variables we have $\Bern{p}+\Bern{sq}\stle\Bern{sp}+\Bern{q}$.}}.

Let $(I_{g,k,y},I_{r,k,y},I_{b,k,y})$ be independent copies of $(I_{g,y},I_{r,y},I_{b,y})$ for each $k$. Then by the above, there exists a coupling between $(G^*,X)$ and $(I_{g,k,y},I_{r,k,y},I_{b,k,y})_{k}$ such that on the event that $\deg(v_k)\le C_2\frac{\log N}{\log\log N}$ for all $k\le C_0\log N$, we have $\dg(v_k)\le\sum_{y}I_{g,k,y}$, $\dr(v_k)\le\sum_{y}I_{r,k,y}$ and $\db(v_k)\le\sum_{y}I_{b,k,y}$ for all $k\le C_0\log N$.

Note that $\E{\sum_{y}I_{g,y}}\asymp\E{\sum_{y}I_{r,y}}\asymp1$ and $\E{\sum_{y}I_{b,y}}\ll1$, hence from \Cref{sum_of_Berns} we get that for sufficiently large $C$ for any $c\le C_0$ we have
\begin{align*}
\pr{\sum_{k=0}^{c\log N}\sum_{y}(I_{g,k,y}+I_{r,k,y}+I_{b,k,y})\ge Cc\log N}\quad \le\quad\exp\left(-C_3\log N\right),
\end{align*}
where $C_3$ is a constant that can be made arbitrarily large by choosing $C$ large. In particular, we get a $o\left(\frac{1}{N}\right)$ bound if $C$ is sufficiently large.

From \Cref{max_deg} we know that with probability $1-o\left(\frac{1}{N}\right)$ the graph $G^*$ is such that all vertices have degree $\le C_2\frac{\log N}{\log\log N}$ where $C_2$ is a sufficiently large constant.

Putting these estimates together and using Markov's inequality finishes the proof.
\end{proof}

\begin{lemma}\label{edge_expected_num_crossings}
For any realisation $T$ of a random quasi-tree and any directed edge $(x,y)$ in $T$, the expected number of times a walk on $T$ started from $x$ or $y$ crosses $(x,y)$ before hitting the boundary can be bounded as
\[
\estart{\#\text{crossings of }(x,y)\text{ by time }\taubound}{y}\quad\le\quad\estart{\#\text{crossings of }(x,y)\text{ by time }\taubound}{x}\quad\lesssim\quad1.
\]
\end{lemma}

\begin{proof}
The $\le$ is immediate. The $\lesssim$ can be derived as follows.
\begin{align*}
&\estart{\#\text{crossings of }(x,y)\text{ by }\taubound}{x}\quad=\quad\sum_{k\ge0}\prstart{\ge k\text{ returns to }x\text{ by }\taubound,\text{ step to }y\text{ after the }k\text{th one}}{x}\\
&=\quad\sum_{k\ge0}\prstart{\ge k\text{ returns to }x\text{ by }\taubound}{x}\frac{1}{\deg(x)}\quad\lesssim\quad1,
\end{align*}
where in the last estimate we used \Cref{nice:stopped_walk_transience}. \end{proof}

\begin{lemma}\label{nice:reach_range_clogN}
For any $C_0>0$ and $\theta\in(0,1)$ there exists a $C>0$ such that for any $c\in(0,C_0)$ the following is satisfied. For any $x$, with probability $1-o\left(\frac1N\right)$, the graph $G^*$ is such that with probability at least $1-\theta$ the time it takes for a walk $X$ from $x$ on the quasi-tree $T(G^*)$ induced by $G^*$ to visit $c\log N$ different vertices or hit the boundary of a ball is $\le Cc\log N$.
\end{lemma}

\begin{proof}
Let $s:=c\log N$ and let $\alpha_s$ be the number of steps it takes for the walk to visit $s$ different vertices.
By \Cref{bound_sum_of_first_k_deg} we know that w.p.\ $1-o\left(\frac{1}{N}\right)$ the graph $G^*$ is such that whp the degrees of the first $s$ different vertices visited by $X$ sum to $\le C_1s$ where $C_1$ is a large positive constant. If the degrees sum to $\le C_1s$ then by time $\alpha_s$ the walk crosses at most $C_1s$ different directed edges. Let $e_k$ be the $k$th new directed edge crossed. By \Cref{edge_expected_num_crossings} the expected number of total crossings of $e_1$, ..., $e_{C_1s}$ by time $\taubound$ is $\le C_2C_1s$ where $C_2$ is a large constant. Markov's inequality gives that with prob $\ge1-\frac12\theta$ the number of crossings is $\le \frac2\theta C_2C_1s$, hence $\alpha_s\le\frac2\theta C_2C_1s$.
\end{proof}

\begin{proof}[Proof of \Cref{nice:nbhd_properties}] Note that on the event that the exploration of \Cref{def:nbhd_for_hitting_nice} succeeds, a walk on $T$ can be coupled to a walk on $T(G^*)$ so that they agree until they first hit $A\cup\{\til{e}\}$.

From \Cref{nice:bound_nontrunc_size} we know that the walk on $T(G^*)$ hits $A\cup\{\til{e}\}$ by the time it reaches level $C\log\log N$, where $C$ is a large constant. Combining \Cref{nice:visit_many_cool_boxes}, \Cref{nice:cross_many_red_edges} and \Cref{nice:reach_level_CloglogN} we get that with probability $1-o\left(\frac1N\right)$ the graph $G^*$ is such that with high probability the walk $X$ on $T(G^*)$ reaches level $C\log\log N$ or hits $\partial$ by the time it reaches graph distance $R$ from $x$. {This shows \eqref{item:nice:not_hit_boundary}.} Also, whp the walk visits $\le\frac{r^{d-1}c_R}{d}\log N$ different vertices by the time it reaches level $C\log\log N$ or hits $\partial$. This proves \eqref{item:nice:bound_num_visited}.

From \Cref{nice:reach_range_clogN} we also get \eqref{item:nice:bound_time}. Part \eqref{item:nice:not_hit_etil} follows from \Cref{nice:stopped_walk_transience}, and part \eqref{item:nice:bound_prob_given_e} follows by noting that by definition $\theta_T(e)\le\frac{\ctrunc}{\log N}$ for all $e\in A$, and by \Cref{nice:thetaT_properties} we have $\prstart{\tau_e<\taubound}{\rho}\le\theta_T(e)$.
\end{proof}

\subsection{Concluding the proof of \Cref{hit_almost_nice_quickly}}

Now we are ready to conclude the proof of \Cref{hit_almost_nice_quickly}.

Let us fix a vertex $x$ of $G$ and let us consider an exploration of a random neighbourhood of $x$ in $G^*$ as in \Cref{def:nbhd_for_hitting_nice}.

Let $\G_0$ be the $\sigma$-algebra generated by this exploration. Let $\Omeganbhd$ be the $1-o\left(\frac1N\right)$ probability event that $G^*$ satisfies the properties in \Cref{nice:nbhd_properties}, and in what follows let us work on this event. Let $A=\{e_1,...,e_a\}$ and let us do the following for $k=1,2,...,a$ in this order. In the $k$th step, let us consider $e_k$ and let us sample its yet unrevealed endpoint $e_k^+$ in $G^*$. Let us also sample the neighbourhood of the other endpoint in $G^*\setminus\{e_{{k}}\}$ up to $K$ levels (according to the definition of a quasi-tree with blue half-edges in \Cref{sec:Gstar_and_T}). If $e_k^+$ is almost nice according to \Cref{def:almost_nice} with distinguished edge $e_k$, then let us say that the edge $e_k$ is nice. Let $\G_k$ be the $\sigma$-algebra generated by $\G_{k-1}$ and this exploration around the other endpoint of $e_k$.

For each $k\in\{1,2,...,a\}$ let $q_k=\prscond{\tau_{A}=\tau_{e_k}<\tau_{\partial}\wedge\tau_{\til{e}}}{\G_0}{x}$ and let $I_k=\1{e_k\text{ nice}}$. Note that all $q_k$ are $\G_0$-measurable, while $I_k$ is $\G_k$-measurable.

We start by showing that at each step of the exploration there is a positive probability that the currently considered endpoint is nice.

\begin{lemma}
There exist positive constants $c_1$ and $c_2$, not depending on the choice of $\ctrunc$, and a positive constant $C_3$, with the following property. For any $k\in\{1,2,...,a\}$ and any realisation of the exploration in the first $(k-1)$ steps such that $\Omeganbhd$ is satisfied, and the number of explored vertices and long-range edges is $\le N^{c_1}$, we have
\begin{align*}
\prscond{I_k=1}{\G_{k-1}}{}\quad&\ge\quad c_2,\\
\prscond{\#\left(\text{vertices and half-edges revealed in step }k\right)\le e^{C_3(\log\log N)^2}}{\G_{k-1}}{} \quad&=\quad1-o\left(\frac{1}{N^2}\right).
\end{align*}
\end{lemma}

\begin{proof}
The second bound follows from \Cref{volume_tail_bound}. We can prove the first bound as follows.

Let $y=e_k^-$ be the endpoint of $e_k$ in the already explored region. For a given vertex $z$ we know from \Cref{prob_nice} that without any conditioning we would have $\pr{z\text{ not nice in }G^*\setminus\{\{y,z\}\}}\lesssim\frac{1}{(\log N)^{3/2}}$. Also, for any given $z$ we have $\prscond{e_k^+=z}{\G_{k-1}}{}\lesssim\frac{1}{N^{1-\beta}\log N}$. Similarly to the proof of \Cref{nice:coupling_success_and_nbhd_size} we  can couple a copy of the unconditioned graph $G^*\setminus\{\{y,z\}\}$ and a copy of $\left(G^*\setminus\{\{y,z\}\}\mid\G_{k-1},\:e_k^+=z\right)$ such that on the event that the level $K$ neighbourhood of $z$ in $G^*\setminus\{\{y,z\}\}$ does not intersect the set of explored vertices in $\left(G^*\setminus\{\{y,z\}\}\mid\G_{k-1},\:e_k^+=z\right)$, the level $K$ neighbourhood of $z$ agrees in the two copies.

Because of this, we can write
\begin{align*}
&\prscond{z\text{ is not almost nice with distinguished edge }\{y,z\}}{\G_{k-1},\:e_k^+=z}{}\quad\\
&\le\pr{z\text{ not nice in }G^*\setminus\{\{y,z\}\}}\\
&\:+\prscond{\text{the level }K\text{ nbhd of }z\text{ in an indep copy of }G^*\setminus\{\{y,z\}\}\text{ intersects the explored region}}{\G_{k-1}}{}.
\end{align*}

Then for sufficiently small values of $c_1$ (for $c_1<1-\beta$) we have
\begin{align*}
&\prscond{e_k\text{ not nice}}{\G_{k-1}}{}\quad\\
&=\quad\sum_{z}\prscond{e_k^+=z}{\G_{k-1}}{}\prscond{z\text{ is not almost nice with distinguished edge }\{y,z\}}{\G_{k-1},\:e_k^+=z}{}\\
&\lesssim\quad\sum_{z}\prscond{e_k^+=z}{\G_{k-1}}{}\frac{1}{(\log N)^{3/2}}\\
&\quad+\quad\frac{1}{N^{1-\beta}\log N}\econd{\#(z\text{ within graph dist. }C_KC_R(\log\log N)^2\text{ from the explored region})}{\G_{k-1}}\\
&\lesssim\quad\frac{1}{(\log N)^{3/2}}+\frac{1}{N^{1-\beta}\log N}N^{c_1}e^{C_3(\log\log N)^2}+\frac{1}{N^{1-\beta}\log N}N\frac{1}{N^2}\quad\lesssim\quad\frac{1}{(\log N)^{3/2}}\quad\ll\quad1,
\end{align*}
where in the first inequality we have used that the set of vertices $z$ whose level-$K$ neighbourhood intersects the explored region is a subset of the set of vertices which have graph distance at most $C_KC_R(\log\log N)^2$ from the explored region. In the second inequality, we use that the probability that all $C_KC_R(\log\log N)^2$ balls in $G^*$ have size at most $e^{C_3(\log\log N)^2}$ is at least $1-\frac{1}{N^2}$, and we otherwise bound the quantity in the expectation by $N$.
\end{proof}

\begin{proof}[Proof of \Cref{hit_almost_nice_quickly}]We have just shown that there exists a constant $c_2$, not depending on any of the parameters, such that for any realisation of $\G_0$ satisfying $\Omeganbhd$, we can couple $(I_k)$ with a sequence $(J_k)$ of iid $\Bern{c_2}$ random variables such that with probability $1-o\left(\frac1N\right)$ we have $I_k\ge J_k$ for all $k$. Using this, using that on $\Omeganbhd$ we have $\sum_{k}q_k\ge c_4$ (where $c_4$ is a positive constant, not depending on any other parameters) and $q_k\le\frac{\ctrunc}{\log N}$ for all $k$, and using the concentration bound in \Cref{sum_indep_RVs} we get that on $\Omeganbhd$ we have
\begin{align*}
&\prscond{\sum_{k=1}^{a}q_kI_k\le\frac12c_2c_4}{\G_0}{}\quad\le\quad\pr{\sum_{k=1}^{a}q_kJ_k\le\frac12c_2c_4}+o\left(\frac1N\right)\quad\\
&=\quad\pr{\sum_{k=1}^{a}\frac{\log N}{\ctrunc}q_kJ_k\le\frac{c_2c_4}{2\ctrunc}\log N}+o\left(\frac1N\right)\quad \le\quad\exp\left(-\frac{3c_2c_4}{20\ctrunc}\log N\right)+o\left(\frac1N\right).
\end{align*}
By choosing $\ctrunc$ sufficiently small, we can make the right-hand side $o\left(\frac1N\right)$.

Using this and \Cref{nice:nbhd_properties} with a sufficiently small $\theta'$ in \eqref{item:nice:bound_time}, we get that there is a constant $C$ depending on $\theta'$ such that for any starting vertex $x$, with probability $1-o\left(\frac1N\right)$ the graph $G^*$ is such that with probability at least $\theta'$ a walk from $x$ hits an almost nice vertex in $\le C\log N$ steps. By the union bound, we now know that whp the graph is such that for all starting points, with probability at least $\theta'$, the walk hits an almost nice vertex in $C\log N$ steps. Then, using repeated trials, we also see that for any $\theta\in(0,1)$ with probability at least $\theta$ the walk hits an almost nice vertex in $\lesssim\log N$ steps, uniformly over the starting point. This finishes the proof of \Cref{hit_almost_nice_quickly}.
\end{proof}

\section{Proof of the main results}\label{sec:conclusion}

\subsection{Mixing at a random time}

In this section we prove \Cref{tauhat_transition_probs}. We remark that for our application, namely, for the sake of proving \Cref{hit_upper_bound}, it would have been enough to establish \Cref{tauhat_transition_probs} for $y\in\Vniceprime$ rather than any $y$, since by \Cref{prob_nice} $\pi(\Vniceprime)=1-o(1)$.

In the proof of  \Cref{tauhat_transition_probs}, we use the following auxiliary result.

\begin{lemma}\label{most_blue_nice}
Let $\pblue=\frac{\cb}{N\log N}$ where $\cb\asymp1$. Then for any $\theta\in(0,1)$ there exists $C>0$ such that whp the following properties hold. The set $\A:=\left\{y\in\Vnice:\:\db(y)\ge1\right\}$ has size $|\A|\le (1+\theta)\cb\frac{N}{\log N}$. The set $\A_C:=\left\{y\in\A:\:\deg(y)\le C,\db(y)= 1, \text{ and } \deg(z)\le C\, \mathrm{for}\, z\text{ with }z\simb y\right\}$ has size $|\A_C|\ge(1-\theta)\cb\frac{N}{\log N}$. For any $x\in\Vniceprime$ the set $\A_x:=\left\{y\in\A:\:K\text{-neighbourhoods of }x,y\text{ are disjoint}\right\}$ has size $|\A_x|\ge(1-\theta)\cb\frac{N}{\log N}$.
\end{lemma}

The proof of \Cref{most_blue_nice} is deferred to \Cref{app:deg_vol}. Once we have this, the proof of \Cref{tauhat_transition_probs} proceeds as follows.

\begin{proof}[Proof of \Cref{tauhat_transition_probs}]
Throughout the proof, the probabilities are conditional on $G^*$, but this is dropped from the notation. {We work on the high probability event that $G^*$ satisfies the statement of \Cref{hit_almost_nice_quickly} with $\theta=\frac78$, the statement of \Cref{transition_probs_on_Vnice} and the statement of \Cref{most_blue_nice} with a sufficiently small value of $\theta$.}

\begin{itemize}
\item Let $C_1$ be the constant from \Cref{hit_almost_nice_quickly} such that for any starting point with probability at least $\frac78$ we have $\tauniceprime\le C_1\log N$, where $\tauniceprime$ denotes the first time that the walk visits $\Vniceprime$. Let $C_2$ be the constant bounding $\frac{\taublue^{(2)}}{\log N}$ in \Cref{transition_probs_on_Vnice}.
\item Let us consider any $x\in V$. If $x\in\Vniceprime$ then let $\A_x$ be as in \Cref{most_blue_nice}, while if $x\not\in\Vniceprime$, let $\A_x:=\left\{z:\:\sum_{y\in \Vniceprime}\prstart{X_{\tauniceprime}=y,\tauniceprime\le C_1\log N}{x}\1{z\in\A_y}\ge\frac12\right\}$. (Note that this latter definition is also consistent for $x\in\Vniceprime$.) Intuitively, $\A_x$ is the set of nice vertices with blue edges for which there is probability at least $\frac12$ that their $K$-neighbourhoods are disjoint from the $K$-neighbourhood of the first nice vertex that a walk from $x$ hits. (If the walk does not hit a nice vertex by time $C_1\log N$, then the second neighbourhood is treated as not well-defined and the neighbourhoods are not treated as disjoint).
\item {We can lower bound the size of $|\A_x|$ as follows.
\begin{align*}
|\A_x|\quad&=\quad\sum_{z\in\A}\1{\sum_{y}\prstart{X_{\tauniceprime}=y,\tauniceprime\le C_1\log N}{x}\1{z\in\A_y}\ge\frac12}\quad\\
&\ge \quad\sum_{z\in\A}\left(2\sum_{y}\prstart{X_{\tauniceprime}=y,\tauniceprime\le C_1\log N}{x}\1{z\in\A_y}-1\right)\\
&=\quad2\sum_{y}\prstart{X_{\tauniceprime}=y,\tauniceprime\le C_1\log N}{x}|\A_y|\quad-\quad|\A|
\end{align*}
By \Cref{most_blue_nice} we have $|\A_y|\ge(1-\theta)c_b\frac{N}{\log N}$ for all $y$ and $|\A|\le(1+\theta)c_b\frac{N}{\log N}$. Using this and that $\sum_{y}\prstart{X_{\tauniceprime}=y,\tauniceprime\le C_1\log N}{x}\ge\frac78$ we get that $|\A_x|\ge\left(\frac34-\frac{11}{4}\theta\right)c_b\frac{N}{\log N}$.}
\item Let $\tau_1=\tauniceprime\wedge C_1\log N$ and let $\tau_2=\taublue^{(2)}\wedge C_2\log N$.
\item For any $x\in V$ and $z\in\A_x$, using \Cref{transition_probs_on_Vnice} and the definition of $\A_x$ we get that
\begin{align}\label{eq:Xtau1tau2}
&\prstart{X_{\tau_1+\tau_2-1}=z}{x}\quad=\quad\sum_{y}\prstart{X_{\tau_1}=y}{x}\prstart{X_{\tau_2-1}=z}{y}\quad\\
\nonumber
&\ge\quad\sum_{y}\prstart{X_{\tau_1}=y}{x}\1{z\in\A_y}\prstart{X_{\tau_2-1}=z}{y}\quad
\gtrsim\quad\sum_{y}\prstart{X_{\tau_1}=y}{x}\1{z\in\A_y}\frac{\log N}{N}\:\gtrsim\:\frac{\log N}{N}.
\end{align}
\item Let $C_3=C_1+C_2$ and let $\tau_3\sim\Unif{\{1,...,C_3\log N+1\}}$, independently of everything else. Let $C$ be a sufficiently large constant, to be specified later.
\item We will now show that for any $x,y\in V$ we have,
\begin{align*}
&\prstart{X_{\tau_1+\tau_2+\tau_3}=y}{x}\quad\ge\quad\sum_{{z\in\A_x\cap\A_y\cap\A_C}}\prstart{X_{\tau_1+\tau_2-1}=z, X_{\tau_1+\tau_2+1}=z}{x}\prstart{X_{\tau_3-1}=y}{z}\\
&\gtrsim\quad\frac{\log N}{N}\sum_{z\in\A_x\cap\A_y\cap\A_C}\frac{1}{\log N}\sum_{k=0}^{C_3\log N}\prstart{X_k=y}{z}\:
=\:\frac{1}{N}\sum_{z\in\A_x\cap\A_y\cap\A_C}\sum_{k=0}^{C_3\log N}\frac{\deg(y)}{\deg(z)}\prstart{X_k=z}{y}\\
&\gtrsim\quad\frac{\deg(y)}{N}\sum_{z\in\A_x\cap\A_y\cap\A_C}\prstart{X_{\tau_1+\tau_2-1}=z}{y}\quad\gtrsim\quad\frac{\deg(y)}{N}\left|\A_x\cap\A_y\cap\A_C\right|\frac{\log N}{N}.
\end{align*}
For the second inequality, we used~\eqref{eq:Xtau1tau2} and the fact that each $z\in\A_C$ has blue degree 1 and its blue neighbour has degree at most $C$, so the probability that the walk immediately backtracks the blue edge crossed at $z$ is at least $\frac{1}{C}\gtrsim 1$. For the penultimate inequality, we used that by definition $\tau_1+\tau_2\le C_3\log N$, while for the final inequality we used~\eqref{eq:Xtau1tau2} again. Intuitively, the above inequalities say the following. The probability of getting from any $x$ to any $y$ in time $\tau_1+\tau_2+\tau_3$  can be lower bounded by summing over all suitable $z$ \footnote{More precisely, the sum is over all nice $z$ with one blue edge which likely have disjoint $K$-neighbourhood from the first nice vertex walk from $x$ or from $y$ would hit and which have bounded degree and for which the other endpoint of their blue edge has bounded degree. The number of such vertices is of the same order as the number of vertices with blue edges.} the probability that the walk from $x$ crosses a blue edge at $z$ at time $\tau_1+\tau_2$, goes back on that edge immediately (this is why we require the other endpoint to have bounded degree) and then hits $y$ in another $\tau_3-1$ steps. The probability of hitting $y$ from $z$ after reversing is a suitable ratio of degrees times the probability that at time $\tau_3-1$ the walk from $y$ hits $z$. However, as  $z$ has one blue edge and bounded degree, there is a constant probability that if the walk is at $z$ at time $\tau_3-1$, it would go on to cross that blue edge in the next time step. Moreover, as $\tau_1+\tau_2\le C_3\log N$, by definition, this means that we can lower bound the probability of getting from $y$ to $z$ at time $\tau_3-1$ by the probability that the time $\tau_3$ exactly happens to agree with the time $\tau_1+\tau_2$, which as $\tau_3$ is uniform, is of order $\frac{1}{\log N}$, multiplied with the probability of walk at time  $\tau_1+\tau_2-1$ being at vertex $z$. The reason this final lower bound is not actually very wasteful is that  $z$ has a blue edge, and typically the walk visits a constant number of vertices with blue edges in logarithmic time, which means that unless $\tau_3-1$ agrees with some of these constant number of times, it could not possibly be at $z$. Therefore, it is not unreasonable that insisting that $\tau_3-1$ agrees with one specific such time (namely $\tau_1+\tau_2-1$) when the walk is at a vertex with a blue edge would give a lower bound of the correct order.

\item {Using \Cref{most_blue_nice}, the lower bound on $|\A_x|$ and $|\A_y|$ that we obtained above, and \Cref{hit_almost_nice_quickly} we see that for sufficiently small values of $\theta$ and sufficiently large values of $C$ we have $\left|\A_x\cap\A_y\cap\A_C\right|\asymp\frac{N}{\log N}$, hence we get that $\prstart{X_{\tau_1+\tau_2+\tau_3}=y}{x}\gtrsim\frac{\deg(y)}{N}$. By definition $\tau_1+\tau_2+\tau_3\le(C_1+C_2+C_3)\log N$. This finishes the proof.}\qedhere
\end{itemize}
\end{proof}

\subsection{Bounding the hit time and the lazy mixing time}

Now we are ready to conclude the proof of \Cref{hit_upper_bound} and \Cref{tmixlazy_upper_bound}.

\begin{proof}[Proof of \Cref{hit_upper_bound}]
Combining \Cref{hit_almost_nice_quickly} and \Cref{tauhat_transition_probs} we get that for any $\alpha\in(0,1)$ we have $\hit{\alpha}{\theta}\lesssim\log N$ for some $\theta\in(0,1)$. By considering multiple trials, we can extend this for all $\alpha,\theta\in(0,1)$.
\end{proof}

\begin{proof}[Proof of \Cref{tmixlazy_upper_bound}]
From~\cite[Theorem 1.1]{mixing-hitting-large-sets} we know that for any $\alpha\in\left(0,\frac12\right)$ the lazy mixing time is of the same order as
\[\tH{\alpha}:=\quad\max_{x,A:\pi(A)\ge\alpha}\estart{\tau_A}{x}.\]
From \Cref{tauhat_transition_probs} we know that for any $x$ and any $A$ we have $\prstart{X_{\tau}\in A,\:\tau\le C\log N}{x}\ge\frac{\theta}{N}\sum_{v\in A}\deg(v)-o(1)$, and from \Cref{total_degree} we know that $\frac1N\sum_{v\in A}\deg(v)\asymp\pi(A)$. Using these we get that $\estart{\tau_A}{x}\lesssim\frac1\alpha\log N$ for any $A$ with $\pi(A)\ge\alpha$.
\end{proof}

\subsection{Bounding $\trel$ and $\trelabs$}

In this section we present the proof of \Cref{trelabs_lower_bound} and \Cref{trel_asymp_trelabs}.

\begin{proof}[Proof of \Cref{trelabs_lower_bound}]
We will show that whp the Cheeger constant satisfies $\Phi^*\lesssim\frac{1}{\log N}$, and then use that $\trelabs\gtrsim\frac{1}{\Phi^*}$.

Let $S=\{1,...,\frac{n}{2}\}\times\Z_n^{d-1}$. Then
\[\Phi^*\quad\le\quad\frac{Q(S,S^c)}{\pi(S)\wedge\pi(S^c)}\quad=\quad\frac{\sum_{x\in S,\:y\in S^c}\1{x\sim y}}{\left(\sum_{v\in S}\deg(v)\right)\wedge\left(\sum_{v\in S^c}\deg(v)\right)}.\]

Note that $\E{\sum_{v\in S}\deg(v)}\asymp\vr{\sum_{v\in S}\deg(v)}\asymp n^d$, so whp $\sum_{v\in S}\deg(v)\asymp n^d$. Similarly, whp $\sum_{v\in S^c}\deg(v)\asymp n^d$.

Also, we have
\[\E{\sum_{x\in S,\:y\in S^c}\1{x\simrgb y}}\quad=\quad\sum_{x\in S,\:y\in S^c}p_{x,y}\quad\asymp\quad\frac{1}{\log n}\sum_{x\in S,\:y\in S^c}\frac{1}{|x-y|^d},\]
\[\vr{\sum_{x\in S,\:y\in S^c}\1{x\simrgb y}}\quad=\quad\sum_{x\in S,\:y\in S^c}p_{x,y}(1-p_{x,y})\quad\asymp\quad\sum_{x\in S,\:y\in S^c}p_{x,y}.\]

For a given distance $r$, we get that there are $\asymp n^{d-1}r^d$ pairs $x\in S$, $y\in S^c$ with $|x-y|=r$. Hence $\sum_{x\in S,\:y\in S^c}p_{x,y}\asymp\frac{n^d}{\log n}$. This shows that whp $\sum_{x\in S,\:y\in S^c}\1{x\sim y}=\sum_{x\in S,\:y\in S^c}\1{x\simrgb y}+n^{d-1}\asymp\frac{n^d}{\log n}$.

So overall we get that whp $\Phi^*\le\frac{Q(S,S^c)}{\pi(S)\wedge\pi(S^c)}\asymp\frac{1}{\log n}\asymp\frac{1}{\log N}$.
\end{proof}

For the proof of \Cref{trel_asymp_trelabs} we use the following result, which roughly speaking says that if a graph is `far from being bipartite', then the relaxation time and the absolute relaxation time of a simple random walk on it are of the same order.

\begin{proposition}[based on \cite{near_bipartiteness}]\label{far_from_bipartite_asymp} For any constant $\eps>0$ there exists a constant $C>0$ with the following property. If a graph $H=(V,E)$ is such that for any partition $A\sqcup B $ of $V$ we have $|E(A)|+|E(B)|\ge\eps|E|$, then a simple random walk on $H$ satisfies $\trel\le\trelabs\le C\trel$.

(For a set $S\subseteq V$ of vertices, $E(S)$ denotes the set of edges in $E$ with both endpoints in $S$.)
\end{proposition}

\Cref{far_from_bipartite_asymp} is a consequence of a result in \cite{near_bipartiteness}, and its proof is presented in \Cref{sec:near_bipartite}.

In what follows, we prove that whp $G^*$ satisfies this `far-from-bipartite' property. Once we have this, the proof of \Cref{trel_asymp_trelabs} immediately follows.

\begin{lemma}\label{Gstar_far_from_bipartite}
There exists a positive constant $c$ such that whp $G^*$ has the following property. For any partition of $V=A\sqcup B$ of its vertices, we have $|E(A)|+|E(B)|\ge c|E|$.
\end{lemma}

\begin{proof}
In what follows, we assume for convenience that $n$ is even, but a very similar proof works for odd $n$. From \Cref{total_degree} we know that whp $|E|\le(2d+2)N$, so it is sufficient to show that whp for any partition $V=A\sqcup B$ we have $|E(A)|+|E(B)|\gtrsim N$.

Let us say that a vertex $v\in\Z_n^d$ is even if the sum of its coordinates is even, and it is odd otherwise. Let us partition $V$ into pairs of adjacent vertices, of the form $\left((2k,i_2,...,i_d),(2k+1,i_2,...,i_d)\right)$ and let us identify the pairs with $W:=\Z_{n/2}\times\Z_{n}^{d-1}$ the natural way.

Consider a partition $A\sqcup B$ of $V$.

Let $S\subseteq W$ correspond to the pairs where both vertices are in $A$ or both are in $B$, and let $U\subseteq W$ correspond to the pairs where the even vertex (i.e.\ the vertex from the pair where the sum of the coordinates is even) is in $A$. This defines a bijection between partitions $A\sqcup B$ of $V$ and pairs $(S,U)$ of subsets of $W$.

Note that for any $s\in S$, the black edge between the two vertices of the corresponding pair runs within $A$ or within $B$. Also note that for any $u\in U$, $v\in U^c$ with $u\sim v$ in $\Z_{n/2}\times\Z_{n}^{d-1}$, there is a black edge between two of the four vertices corresponding to the pairs $u$ and $v$ that runs within $A$ or within $B$. Hence for any positive constant $\eps$: if $|S|\ge \eps N$ or $\sum_{u\in U,\:v\in U^c}\1{u\sim v}\ge \eps N$, then we have $|E(A)|+|E(B)|\ge \eps N$.

In what follows, for any partition satisfying $|S|\le \eps N$ with a sufficiently small $\eps$, we bound the probability of $|E(A)|+|E(B)|\le c N$ (where $c$ is an absolute constant). Then we bound the number of partitions with $|S|\le \eps N$ and $\sum_{u\in U,\:v\in U^c}\1{u\sim v}\le\eps N$. We conclude the proof by taking a union bound over these partitions.

Let us consider a partition $V=A\sqcup B$ with $|S|\le \eps N$.

Let $S'$ be a set that contains exactly one vertex from each of the pairs in $S$, and let $A'$ be a set that contains the vertices in $A$ except for the ones in the pairs in $S$.

A quick calculation shows that $\sum\limits_{x,y\in A'\cup S'}p_{x,y}\asymp N$. Also, $\sum\limits_{x\in S',y\in A'\cup S'}p_{x,y}\le|S'|\le \eps N$, hence for sufficiently small values of $\eps$ we have $\sum\limits_{x,y\in A'}p_{x,y}\ge 2cN$ where $c$ is an absolute constant.

Then using \Cref{sum_of_Berns} we get that
\begin{align}\label{eq:EA_concentration}
\pr{\sum\limits_{x,y\in A'}\1{x\sim y}\le cN}\quad&\le\quad \exp\left(-(1-\log2)c N\right).
\end{align}

Now let us turn to bounding the number of partitions with $|S|\le \eps N$ and $\sum_{u\in U,\:v\in U^c}\1{u\sim v}\le\eps N$.

For each $U\subseteq W$ let us define a corresponding $\what{U}\subseteq W$ such that $u\in \what{U}$ if and only if $u\in U$, $u+e_1\in U^c$ or $u\in U^c$, $u+e_1\in U$ (where $+e_1$ means increasing the first coordinate by one mod $\frac{n}{2}$). Then each $\what{U}\subseteq W$ corresponds to at most $2^{n^{d-1}}$ different choices of $U$.\footnote{More precisely each $\what{U}\subseteq W$ with an even number of elements in each `row' of the form $\{k\}\times\Z_{n}^{d-1}$ corresponds to exactly $2^{n^{d-1}}$ different choices of $U$.} Also, if $\sum_{u\in U,\:v\in U^c}\1{u\sim v}\le \eps N$ then $|\what{U}|\le\eps N$.

Therefore the number of pairs $(S,U)$ with $|S|\le \eps N$, $\sum_{u\in U,\:v\in U^c}\1{u\sim v}\le \eps N$ can be bounded as
\begin{align*}
&\le\quad2^{n^{d-1}}\left(\sum_{m=0}^{\eps N}\binom{N/2}{m}\right)^2\quad\lesssim\quad 2^{n^{d-1}}\left(\eps N\:\binom{N/2}{\eps N}\right)^2 \quad\lesssim\quad 2^{n^{d-1}}\left(\frac{\sqrt{\eps N}}{(2\eps)^{\eps N}(1-2\eps)^{N/2-\eps N}}\right)^2\quad\\
&\lesssim\quad\exp\left(N^{\frac{d-1}{d}}(\log 2)+ N\left(2\eps\log\left(\frac{1}{2\eps}\right)+(1-2\eps)\log\left(\frac{1}{1-2\eps}\right)\right)+\log(\eps N)\right).
\end{align*}

We have $2\eps\log\left(\frac{1}{2\eps}\right)+(1-2\eps)\log\left(\frac{1}{1-2\eps}\right)\to0$ as $\eps\to0$, so by choosing $\eps$ sufficiently small we can make the right-hand side $\lesssim\exp\left(\frac12(1-\log 2)cN\right)$. Then using \eqref{eq:EA_concentration} and taking a union bound finishes the proof.
\end{proof}

\subsection*{Acknowledgements}

Zsuzsanna Baran was supported by DPMMS EPSRC DTP and by the ERC consolidator grant CUTOFF (101123174). Views and opinions expressed are however those of the authors only and do not necessarily reflect those of the European Union or the European Research Council Executive Agency. Jonathan Hermon was supported by NSERC grants. Allan Sly was supported by a Simons Investigator Grant.

\appendix

\section{Some concentration inequalities}\label{app:concentration}

\begin{lemma}[see e.g. \cite{HDP}, Theorem 2.3.1, Exercise 2.3.2]\label{sum_of_Berns}
$ $\\
Let $S=\sum_{i=1}^{k}X_i$, where $X_i\sim\Bern{p_i}$ are independent, and let $\mu=\E{S}$.\\
Then for any $t>\mu$ we have
\[
\pr{S\ge t}\le e^{-\mu}\left(\frac{e\mu}{t}\right)^t,
\]
and for any $t<\mu$ we have
\[
\pr{S\le t}\le e^{-\mu}\left(\frac{e\mu}{t}\right)^t.
\]
\end{lemma}

\begin{lemma}[see e.g.~\cite{intro_to_random_graphs}, Theorem 21.6, Corollary 21.9]\label{sum_indep_RVs}
$ $\\
Let $S=\sum_{i=1}^{k}X_i$, where $X_i$ are independent, taking values in $[0,1]$, and let $\mu=\E{S}$. Then
\begin{align*}
\pr{S\ge \mu+t}\le& \exp\left(-\frac{t^2}{2(\mu+t/3)}\right)&\forall t>0,\\
\pr{S\le \mu-t}\le& \exp\left(-\frac{t^2}{2(\mu-t/3)}\right)&\forall t\in(0,\mu),\\
\pr{S\ge c\mu}\le& \exp\left(-\mu\left(c\log\left(\frac{c}{e}\right)+1\right)\right)&\forall c>1.
\end{align*}
\end{lemma}

\begin{lemma}\label{Bin_rough_tail_bound}
For any $M\in\Z_{\ge2}$, $t\in\{0,1,...,M\}$, and $q\in(0,1)$ we have
\[
\pr{\Bin{M}{1-q}<t}\quad\le\quad M^tq^{M-t}.
\]
\end{lemma}

\begin{proof}
\begin{align*}
\pr{\Bin{M}{1-q}<t}\:=\:\sum_{s=0}^{t-1}\binom{M}{s}(1-q)^{s}q^{M-s}\:\le\:\sum_{s=0}^{t-1}M^{s}q^{M-s}\:=\: q^M\frac{(Mq^{-1})^t-1}{Mq^{-1}-1}\:\le\: M^{t}q^{M-t}.
\end{align*}
\end{proof}

\section{Entropy and related quantities}\label{app:entropy}

Below we list some simple properties of $H_b$ and $H_b^A$ that were introduced in Definitions~\ref{def:Hb}, \ref{def:Hb_conditional} and \ref{def:HbA_and_HbA_conditional}. The proofs of these are similar in flavour to standard proofs about the entropy $H_1$ and hence they are omitted.

\begin{lemma}\label{Hb_properties}
Let $b\in\Zpos$ and let $(p_i)$, $(q_i)$ and $(r_i)$ be sequences of real numbers taking values in $[0,\theta]$ where $\theta\in(0,1)$ is a sufficiently small constant (depending on $b$). Let $W$ be a random variable taking values in a countable set $\mathcal{W}$ {of size larger than some constant depending on $b$}. Then the following properties are satisfied.
\begin{enumerate}[(i)]
\item\label{property:decreasing}
If $p_i\ge q_i$ for all $i$, then
\[
\ent{b}{p_1,p_2,...}\quad\ge\quad \ent{b}{q_1,q_2,...}.
\]
\item\label{property:splitting}
If $p_i=q_i+r_i$ for all $i$ then
\[
\ent{b}{p_1,p_2,...}\quad\le\quad \ent{b}{q_1,q_2,...}+\ent{b}{r_1,r_2,...}.
\]
\item\label{property:factorising}
If $p_i=q_ir_i$ for all $i$ then
\[
\ent{b}{p_1,p_2,...}\quad\le\quad \ent{b}{q_1,q_2,...}+\ent{b}{r_1,r_2,...}.
\]
\item\label{property:bound_by_set_size}
\[
\ent{b}{W}\quad\le\quad\left(\log|\mathcal{W}|\right)^b.
\]
\item\label{property:bound_by_set_size_two}
If $p_1+...+p_N=p<1$ and $p_i=0$ for $i>N$ then
\[
\ent{b}{p_1,p_2,...}\quad\lesssim\quad p\left(\log N\right)^b+p(-\log p)^b.
\]
\end{enumerate}
\end{lemma}

\begin{lemma}\label{Hb_cond_properties}Let $b\in\Zpos$ and let $W$ and $Z$ be random variables as in \Cref{def:Hb_conditional}. Then
\begin{align*}
\entc{b}{W}{Z}\quad&\le\quad \ent{b}{W}+O(1),\\
\ent{b}{W}\quad&\lesssim\quad \entc{b}{W}{Z}+\ent{b}{Z}.
\end{align*}
\end{lemma}

\begin{lemma}\label{HbA_cond_properties} Let $b\in\Zpos$, let $W$ and $Z$ be random variables as in \Cref{def:Hb_conditional}, and let $A\subseteq B$ be events. Then the following are satisfied.
\begin{enumerate}[(i)]
\item\label{item:HbA_cond_i}
\begin{align*}
\entrc{b}{A}{W}{Z}\quad&\le\quad \entr{b}{A}{W}+O(1),\\
\entr{b}{A}{W}\quad&\lesssim\quad \entrc{b}{A}{W}{Z}+\entr{b}{A}{Z}.
\end{align*}
\item\label{item:HbA_cond_ii}
\begin{align*}
\entr{b}{A}{W}\quad&\le\quad\entr{b}{B}{W}+O(1),\\
\entrc{b}{A}{W}{Z}\quad&\le\quad\entrc{b}{B}{W}{Z}+O(1).
\end{align*}
\item\label{item:HbA_cond_iii} For each $z\in\mathcal{Z}$ let $\mathcal{W}_z=\{w:\:\prcond{W=w}{Z=z}{}>0\}$. Then we have
\begin{align*}
\entrc{b}{A}{W}{Z}\quad\lesssim\quad\E{\left(\log|\mathcal{W}_Z|\right)^b\1{A}}+\E{\prcond{A}{Z}{}\left(-\log\prcond{A}{Z}{}\right)^b}.
\end{align*}
\end{enumerate}
\end{lemma}

\section{Auxiliary estimates regarding degrees and volumes}\label{app:deg_vol}

\subsection{Distribution of the blue edges in $G^*$}\label{sec:blue_edges}

In this section, we prove the following statement.

\begin{lemma}\label{blue_almost_unif}
Let $(a_v)_{v\in V}$ be a collection of random variables that is distributed as the collection of blue degrees $\db(v)$ in $G^*$. Given $(a_v)_{v\in V}$, let $H_{(a_v)}$ be a random graph that can be obtained by adding $a_v$ blue half-edges to each vertex of the torus $G$ and then matching the blue half-edges uniformly randomly. Let $G_{(a_v)}$ be a random graph that is distributed like $G_b$ conditional on $\left\{\db(v)=a_v\:\forall v\in V\right\}$.

Then the graph $H_{(a_v)}$ conditioned on not having any loops or double edges has the same distribution as the graph $G_{(a_v)}$. Also, on the event
\begin{align*}
A:=\left\{\sum_{v}a_v\ge\frac{N}{C\log N},\:\sum_{v}a_v(a_v-1)\le\frac{CN}{(\log N)^{\frac32}},\:\sum_{u,v:u\ne v}a_v(a_v-1)a_u(a_u-1)\le\frac{CN^2}{(\log N)^3}\right\}.
\end{align*}
the probability of $H_{(a_v)}$ not having loops or double edges is $\ge1-\frac{C'}{(\log N)^{\frac12}}$ where $C'$ is a constant depending on $C$. Finally, for sufficiently large values of the constant $C$, we have $\pr{A}=1-o(1)$.
\end{lemma}

\begin{proof}
The first part of the statement follows by noting that the probability of a given realisation of $\Gb$ only depends on the number of its blue edges, and that for a given sequence $(a_v)$, all configurations of $H_{(a_v)}$ without loops and double edges have the same probability.

Given the sequence $(a_v)$, the probability of $H_{(a_v)}$ having a loop is $\lesssim\sum_{v}\frac{a_v(a_v-1)}{(\sum_{x}a_x)-1}$ and the probability of $H_{(a_v)}$ having a double edge is $\lesssim\sum_{u,v:u\ne v}\frac{a_v(a_v-1)a_u(a_u-1)}{\left((\sum_{x}a_x)-1\right)^2}$. On the event $A$ these sums are both $\lesssim\frac{1}{(\log N)^{\frac12}}$ as required.

Finally, note that $\E{\sum_{v}a_v(a_v-1)}\asymp\sum_x\sum_{y,z:\:y\ne z}\pr{x\simb y,x\simb z}\asymp N^3(\pblue)^2\asymp\frac{N}{(\log N)^2}$ and\\ 
$\sum_{u,v:u\ne v}a_v(a_v-1)a_u(a_u-1)\lesssim\E{\sum_{v}a_v(a_v-1)}^2+\sum_{u,v:u\ne v}\pr{u\simb v}\E{a_v(a_v+1)}^2\lesssim\frac{N^2}{(\log N)^4}$. Using these and \Cref{total_blue_degree} we get that $\pr{A}=1-o(1)$.
\end{proof}

\subsection{Proof of \Cref{volume_tail_bound}}\label{app:volume_tail_bound}

To prove \Cref{volume_tail_bound}, we first prove the following two auxiliary lemmas.

\begin{lemma}\label{degree_exp_moment}
For any $\theta\in\left[0,\frac14\right]$ we have $\E{e^{\theta\dgrb}}\le e^{2\theta}$.
\end{lemma}

\begin{lemma}\label{volume_exp_moment}
Let $v$ be a fixed vertex and let $B_r$ denote the ball of radius $r$ around $v$ in $G^*$.\\
Then for any $\theta\in\left[0,\frac14\frac{1}{(2d+2)^r}\right]$ we have
\[\E{e^{\theta\left|B_r\setminus B_{r-1}\right|}}\quad\le\quad e^{(2d+2)^r\theta},\]
and for any $\theta\in\left[0,\frac14\frac{1}{(2d+3)^r}\right]$ we have
\[\E{e^{\theta\left|B_r\right|}}\quad\le\quad e^{(2d+3)^r\theta}.\]

Furthermore, the same bounds hold if we consider a ball $B_r$ of radius $r$ around the root of a random quasi-tree $T$.
\end{lemma}

Once we have these, \Cref{volume_tail_bound} follows from a Chernoff bound and using the second part of \Cref{volume_exp_moment} with $\theta=\frac14\frac{1}{(2d+3)^r}$.

\begin{proof}[Proof of \Cref{degree_exp_moment}]
Note that
\begin{align*}
&\E{e^{\theta\dgrb}}\quad=\quad\prod_{v}\left(p_{0,v}e^\theta+(1-p_{0,v})\right)\quad=\quad \prod_{v}\left(1+(e^\theta-1)p_{0,v}\right)\\
&=\quad\exp\left(\sum_{v}\log\left(1+(e^\theta-1)p_{0,v}\right)\right)\quad\le\quad \exp\left(\sum_{v}(e^\theta-1)p_{0,v}\right)\quad=\quad\exp\left(e^\theta-1\right),
\end{align*}
where we used the independence of edges and for the $\le$ we used that for any $x\ge0$ we have $\log(1+x)\le x$.

For $\theta\in\left[0,\frac14\right]$ we have $e^\theta-1\le2\theta$ which finishes the proof.
\end{proof}

\begin{proof}[Proof of \Cref{volume_exp_moment}]
Note that both for balls in $G^*$ and for balls around the root of $T$ we have $\left(\left|B_r\setminus B_{r-1}\right|\right)_{r\ge1}\stle(Z_r)_{r\ge1}$ where $Z_r$ denotes the population size of generation $r$ of a Galton-Watson process with offspring distribution $\zeta\eqdist2d+\dgrb$.

We prove by induction on $r$ that $\E{e^{\theta Z_r}}\le e^{(2d+2)^r\theta}$ for all $\theta\in\left[0,\frac14\frac{1}{(2d+2)^r}\right]$. For $r=0$ the statement is immediate.

Given the statement for $r=k$, and using \Cref{degree_exp_moment}, we can write
\begin{align*}
\econd{e^{\theta Z_{k+1}}}{Z_k}\quad=\quad\E{e^{\theta\zeta}}^{Z_k}\quad=\quad\left(e^{2d\theta}\E{e^{\theta\dgrb}}\right)^{Z_k}\quad\le\quad e^{(2d\theta+2\theta)Z_k},
\end{align*}
hence
\begin{align*}
\E{e^{\theta Z_{k+1}}}\quad\le\quad\E{e^{\theta(2d+2)Z_k}}\quad\le\quad e^{(2d+2)^{k+1}\theta}
\end{align*}
as required.

The second part of the lemma can be proved analogously by noting that $\left(\left|B_r\right|\right)_{r\ge1}\stle(Z_r)_{r\ge1}$ where $Z_r$ denotes the population size of generation $r$ of a Galton-Watson process with offspring distribution $\zeta\eqdist2d+1+\dgrb$.
\end{proof}

\subsection{Proof of \Cref{most_blue_nice}}

\begin{proof}[Proof of \Cref{most_blue_nice}] Using that $\left(\1{x\simb y}\right)_{x,y\in V,\: x\ne y}$ are independent $\Bern{\frac{\cb}{N\log N}}$ random variables, and the concentration bound \Cref{sum_of_Berns}, we get that whp $\left(1-\theta\right)c_b\frac{N}{\log N}\le \sum_{v}\db(v)\le\left(1+\theta\right)c_b\frac{N}{\log N}$.

This implies the required upper bound on $|\A|$. To facilitate the lower bounds on $|\A_C|$ and $|\A_x|$, we first prove that whp $|\A|\ge\left(1-\frac34\theta\right)c_b\frac{N}{\log N}$.

Note that conditional on $\sum_{v}\db(v)=S$, the distribution of the blue degrees can be obtained by choosing $\frac12S$ distinct pairs $\{x,y\}$ uniformly randomly and counting how many times each vertex appears. If $S\ge\left(1-\frac14\theta\right)c_b\frac{N}{\log N}=:S_0$, then this stochastically dominates the numbers obtained by choosing $\frac12S_0$ distinct pairs $\{x,y\}$ uniformly randomly and counting the number of appearances of each vertex.

Let us consider choosing $\frac12S_0$ distinct pairs $\{x,y\}$ one by one, always choosing uniformly from the yet unchosen pairs. At each step, the probability of choosing a pair that contains a vertex that already appeared earlier is $\le\frac{S_0N}{\frac12N(N-1)-\frac12S_0}\asymp\frac{1}{\log N}$. Then using \Cref{sum_of_Berns} we get that whp there are $\lesssim \frac{N}{(\log N)^2}\le\frac18\theta c_b\frac{N}{\log N}$ pairs containing a vertex that already appeared earlier.

Putting these together, we get that whp the number of vertices with $\db\ge1$ is $\ge\left(1-\frac12\theta\right)c_b\frac{N}{\log N}$.

From \Cref{prob_nice} whp $\#\left\{v:\:\db(v)\ge1,\:v\not\in\Vnice\right\}\le \frac{N}{(\log N)^{4/3}}\le\frac14\theta c_b\frac{N}{\log N}$. This finishes the proof of the lower bound on $|\A|$.

To bound the size of $\A_C$, we note that the green and red degrees of the vertices of $G^*$ can be stochastically dominated as follows. Let us consider a random directed graph $\Ggr^{(2)}$ where for each pair $(x,y)$ in $V$ there is a red directed edge from $x$ to $y$ with probability $\pred_{x,y}$, and there is a green directed edge with probability $\pgreen_{x,y}$, independently. Then we have
\[
\left(\deg^{(\mathrm{out})}(v)+\deg^{(\mathrm{in})}(v)\right)_{v\in V}\quad \stge\quad\left(\dgr(v)\right)_{v\in V},
\]
and $\deg^{(\mathrm{out})}(v)\eqdist\dgr$ are iid, and $\deg^{(\mathrm{in})}(v)\eqdist\dgr$ are iid, but the outdegrees are not independent of the indegrees.

For any $v$ we have $\pr{\db(v)\ge2}\ll\frac{1}{\log N}$, hence whp $\#\left\{v:\:\db(v)\ge2\right\}\le\frac18\theta\cb\frac{N}{\log N}$. Conditioning on the blue edges, working on the high probability event that the number of vertices with $\db\ge1$ is $\ge\left(1-\frac34\theta\right)c_b\frac{N}{\log N}$ and $\le (1+\theta)c_b\frac{N}{\log N}$, and using the above domination, we get that for any~$\theta$ for sufficiently large values of $C$ whp we have $\#\left\{v:\db(v)\ge1 \text{ and }\deg(v)\ge C\right\}\le\frac{1}{16}\theta c_b\frac{N}{\log N}$. Therefore, $\#\left\{v\in \A:\db(v)\ge 2 \text{ or }\deg(v)\ge C \text{ or } (v'\simb v\text{ and } \deg(v')\ge C )\right\}\le\frac24\theta c_b\frac{N}{\log N}$.\\ As whp $\left(1-\frac34\theta\right)c_b\frac{N}{\log N}\le|\A|\le(1+\theta)c_b\frac{N}{\log N}$ this finishes the proof of the lower bound on $|\A_C|$.

From \Cref{volume_high_prob_bound} we know that whp $\Ggr$ is such that for each vertex $x$ the number of vertices in its $2K$-neighbourhood is $\ll\frac{N}{\log N}$. This finishes the proof of the lower bound on $|\A_x|$.
\end{proof}

\subsection{Conditional distributions}

\begin{lemma}\label{cond_green_red_deg}
There exists a positive constant $c$ with the following properties. For any $x,y\in V$, any $W\subseteq V\setminus\{x,y\}$ with $|W|\le N^{c}$, and any $k\ll\log N$ we have
\begin{align}\label{eq:simg_cond_prob}
\prcond{x\simg y}{\dg(y)=k,\:y\not\simg w\:\:\:\forall w\in W}{}\quad\asymp\quad k\cdot\pr{x\simg y}.
\end{align}
Furthermore, for any $x,y\in V$, any $W\subseteq V\setminus\{x,y\}$ with $|W|\le N^{1-\beta+c}$, and any $k\ll N^{1-\beta}\log N$ we have
\begin{align}\label{eq:simr_cond_prob}
\prcond{x\simr y}{\dr(y)=k,\:y\not\simr w\:\:\:\forall w\in W}{}\quad\asymp\quad k\cdot\pr{x\simr y},
\end{align}
and for any $x,y\in V$, any $W\subseteq V\setminus\{x,y\}$ with $|W|\le cN$, and any $k\ll N$ we have
\begin{align}\label{eq:simb_cond_prob}
\prcond{x\simb y}{\db(y)=k,\:y\not\simb w\:\:\:\forall w\in W}{}\quad\asymp\quad 
k\cdot(\log N)\cdot\pr{x\simb y}.
\end{align}
\end{lemma}

\begin{proof} We present the proof \eqref{eq:simg_cond_prob}. {The proof of \eqref{eq:simr_cond_prob} is analogous, while \eqref{eq:simb_cond_prob} can be proved analogously or by direct computation of the two sides.}

We assume that $|x-y|<D^{1-\beta}$ as otherwise the result is immediate. In the proof, we only consider green edges, so to simplify notation we drop the $g$ subscripts from $\simg$ and $\dg$. For a set $A\subseteq V$, let us write $d^A(y)=\sum_{z\in A}\1{z\sim y}$, $d^{-A}(y)=d^{V\setminus A}(y)$, and $d^{-A-x}(y)=d^{-(A\cup\{x\})}(y)$.

Note that \eqref{eq:simg_cond_prob} is equivalent to each of the following.
\begin{align*}
\pr{x\sim y,\:d(y)=k,\:y\not\sim w\:\forall w\in W}\quad&\asymp
\quad k\cdot\pr{x\sim y}\pr{d(y)=k,\:y\not\sim w\:\forall w\in W},\\
\pr{x\sim y}\pr{y\not\sim w\:\forall w\in W}\pr{d^{-W-x}(y)=k-1}\:&\asymp
\:k\cdot\pr{x\sim y}\pr{y\not\sim w\:\forall w\in W}\pr{d^{-W}(y)=k},\\
\pr{d^{-W-x}(y)=k-1}\quad&\asymp\quad k\cdot\pr{d^{-W}(y)=k}.
\end{align*}
Note that
\begin{align*}
k\cdot\pr{d^{-W}(y)=k}\quad=\quad k\cdot\pr{x\sim y}\pr{d^{-W-x}(y)=k-1}+ k\cdot\pr{x\not\sim y}\pr{d^{-W-x}(y)=k}.
\end{align*}
We know that $k\cdot\pr{x\sim y}\lesssim\frac{k}{\log N}\ll1$, while $\pr{x\not\sim y}\asymp1$, hence \eqref{eq:simg_cond_prob} is equivalent to
\begin{align}\label{eq:deg_eq_k}
\pr{d^{-W-x}(y)=k-1}\quad\asymp\quad k\cdot\pr{d^{-W-x}(y)=k}.
\end{align}
We can prove \eqref{eq:deg_eq_k} as follows. Note that we have
\begin{align*}
&k\cdot\pr{d^{-W-x}(y)=k}\quad\\
&=\quad\sum_{\substack{U\subseteq(W\cup\{x,y\})^c\\|U|=k-1}}\sum_{z\in(U\cup W\cup\{x,y\})^c}\left(\prod_{u\in U}\pr{u\sim y}\right)\pr{z\sim y}\left(\prod_{w\in(U\cup W\cup\{x,y,z\})^c}\pr{w\not\sim y}\right)\\
&=\quad\sum_{\substack{U\subseteq(W\cup\{x,y\})^c\\|U|=k-1}}\left(\prod_{u\in U}\pr{u\sim y}\right)\left(\prod_{w\in(U\cup W\cup\{x,y\})^c}\pr{w\not\sim y}\right) \sum_{z\in(U\cup W\cup\{x,y\})^c}\frac{\pr{z\sim y}}{\pr{z\not\sim y}}.
\end{align*}
For any given $U$, we have
\[
\sum\limits_{z\in(U\cup W\cup\{x,y\})^c}\frac{\pr{z\sim y}}{\pr{z\not\sim y}}\quad\asymp\quad\sum\limits_{z\in(U\cup W\cup\{x,y\})^c}\pr{z\sim y}\quad=\quad \E{d(y)}-\E{d^{U\cup W\cup\{x\}}(y)},
\]
which is $\asymp1$, since we know that $\E{d(y)}\asymp1$, while $\E{d^{U\cup W\cup\{x\}}(y)}\lesssim\frac{\log|U\cup W\cup\{x\}|}{\log N}$ can be made arbitrarily small by choosing the constant $c$ in the bound $|W|\le N^c$ sufficiently small.

This shows that
\begin{align*}
&k\cdot\pr{d^{-W-x}(y)=k}\quad\\
&\asymp\quad\sum_{\substack{U\subseteq(W\cup\{x,y\})^c\\|U|=k-1}}\left(\prod_{u\in U}\pr{u\sim y}\right)\left(\prod_{w\in(U\cup W\cup\{x,y\})^c}\pr{w\not\sim y}\right)
\quad=\quad\pr{d^{-W}(y)=k-1},
\end{align*}
which finishes the proof of \eqref{eq:deg_eq_k}, and hence the proof of \eqref{eq:simg_cond_prob}.
\end{proof}

\section{Some further auxiliary results}\label{app:further}

\begin{lemma}\label{unit_flow_in_box}
There exists a constant $C$ with the following property. For any $d\ge3$, any $r$, and any two vertices $x$ and $y$ in a box $[0,r]^d$ (with edges induced by $\Z^d$) there exists a unit flow $\theta$ from $x$ to $y$ in the box that has energy $\le C$.
\end{lemma}

\begin{proof}
It is sufficient to prove the statement in the case $x=(0,...,0)$. (If $\theta_x$ is a unit flow $0\to x$ and $\theta_y$ is a unit flow $0\to y$, then $\theta_y-\theta_x$ is a unit flow $x\to y$ with energy $\le2\left(\energy(\theta_x)+\energy(\theta_y)\right)$.)

Let $z_0$, ..., $z_{d}$ be a sequence of vertices in the box such that $|z_0-x|_{\infty}\le1$, $z_{d}=y$, and for each $i\in\{1,2,...,d\}$ the vertices $z_{i-1}$ and $z_{i}$ are opposing corners of a $d$-dimensional cube within the box.

We can prove that such a sequence always exists by induction on $d$ as follows.

For $d=1$ the statement is true. Now let $d\ge2$ and assume that the statement is true for smaller dimensions. Wlog $y=(y_1,...,y_d)$ where $y_1\ge ...\ge y_d\ge0$. Let $r_d=\lfloor\frac{y_{d-1}-y_{d}}{2}\rfloor$ and let $z_{d-1}=y+(-r_d,...,-r_d,r_d)$. Note that in $z_{d-1}$ all coordinates are nonnegative, they are all $\le r$, and the last two coordinates differ by at most 1. Let $x_{d-1}$ agree with $x$ in the first $d-1$ coordinates and differ in the last coordinate by at most 1 such that in $z_{d-1}-x_{d-1}$ the last two coordinates agree. Now use the result for $d-1$ dimensions for $x_{d-1}$ and $z_{d-1}$, identifying the last two coordinates.\hfill //

From~\cite[Exercise 9.1]{MTMC} we know that for each $i$ there exists a unit flow $\theta_i$ in the box from $z_{i-1}$ to $z_i$, with energy bounded by an absolute constant $C'$. Also, there is a unit flow $\theta_0$ in the box from $x$ to $z_0$ with energy bounded by $d$. Then $\theta:=\sum_{i=0}^{d}\theta_i$ is a unit flow in the box from $x$ to $y$ with energy bounded by $(d+1)^2(C'\vee d)$.
\end{proof}

\subsection{Proof of \Cref{far_from_bipartite_asymp}}\label{sec:near_bipartite}

In the proof of \Cref{far_from_bipartite_asymp} we use the following result from~\cite{near_bipartiteness}.

\begin{lemma}[from \cite{near_bipartiteness}, paraphrased]\label{near_bipartiteness}
Let $P$ be the transition matrix of a reversible Markov chain on a state space $\Omega$ of size $n$, with stationary distribution $\pi$.
Let $1=\lambda_1\geq\lambda_2\geq\ldots\geq \lambda_n$ be the eigenvalues of $P$, and let $f_i$ be the corresponding unit eigenfunctions, such that $P f_i=\lambda_i f_i$ and $\estart{f_i f_j}{\pi}=\1{i=j}$.
Then if $1+\lambda_n \leq c'\left(1-\lambda_2\right)$ for some absolute constant $c' \in (0,1) $ then
$\text{Var}_{\pi}\left|f_n\right| \leq \frac{1+\lambda_n}{1-\lambda_2}$.

Moreover, if we let
\[F_{+}:=\left\{x: f_n(x) \geq 0\right\}\quad \text { and }\quad F_{-}:=\left\{x: f_n(x)<0\right\}\]
then $\left|\pi\left(F_{+}\right)-1 / 2\right|=\left|\pi\left(F_{-}\right)-1 / 2 \right|\lesssim\frac{1+\lambda_n}{1-\lambda_2}$, and the parity breaking time defined by
\[S:=\quad\inf\left\{i:\left(X_{i-1}, X_i\right) \in (F_{+}\times F_{+}) \cup (F_{-}\times F_{-})\right\}\] satisfies for some $\beta>\left|\lambda_n\right|$ that for all $k$ we have
\[\left(1-\frac{2\left(1+\lambda_n\right)}{1-\lambda_2}\right)(\beta^{2k}+\beta^{2k+1})\quad\leq\quad \prstart{S>2k}{\pi}+\prstart{S>2k+1}{\pi}\quad\leq\quad (\beta^{2k}+\beta^{2k+1}).\]
\end{lemma}

\begin{proof}[Proof of \Cref{far_from_bipartite_asymp}]
We apply \Cref{near_bipartiteness} to the simple random walk on $H$. We assume, for the sake of contradiction, that $\eps\trelabs\ge \trel$ for a constant $\eps\in(0,1)$ to be chosen later.
Since $F_+$ and $F_-$ form a partition of $V$, by the assumption of the proposition we have $|E(F_-)|+|E(F_+)|\ge c|E|$. As the first edge crossed by a random walk starting from an invariant distribution is a uniformly chosen edge of the graph, this means that $\prstart{S=1}{\pi}\ge c$.

On the other hand, by \Cref{near_bipartiteness} we also have
\[\left(1-\frac{2\left(1+\lambda_n\right)}{1-\lambda_2}\right)(\beta^{2}+\beta^{3}) \quad\leq\quad \prstart{S>2}{\pi}+\prstart{S>3}{\pi}\quad\le\quad 2\prstart{S>1}{\pi}\quad{\le}\quad 2(1-c).\]
We assumed that $\frac{2\left(1+\lambda_n\right)}{1-\lambda_2}\le2\frac{\trel}{\trelabs}\le 2\eps$, and we have that $\beta\ge |\lambda_n|\ge 1-\frac{\eps}{\trel}\ge 1-\eps$. Therefore,
\[(1-2\eps)((1-\eps)^2+(1-\eps)^3)\quad\le\quad 2(1-c).\]
Choosing a sufficiently small constant value of $\eps$ (in terms of $c$), this gives a contradiction. Therefore, we have $\trelabs\lesssim \trel$.
We know that $\trelabs\ge\trel$ always holds, so this completes the proof.
\end{proof}

\bibliography{LRP_arXiv_v2.bib}
\bibliographystyle{plain}

\end{document}